\documentclass[11pt]{amsart}
\usepackage{amssymb,latexsym, amscd,pb-diagram}
\usepackage{sseq}
\usepackage[all]{xy}
\usepackage[square, numbers]{natbib}
\usepackage{graphicx}
\usepackage{mathrsfs}
\usepackage{color}
\usepackage{amsmath}
\usepackage{multirow}
\usepackage{hyperref}
\usepackage{todonotes}

\usepackage[a4paper, total={6in, 8in}]{geometry}
\usepackage[utf8]{inputenc}
\usepackage[T1]{fontenc}
\usepackage{amsfonts,  amsthm}
\usepackage{mathtools}
\usepackage{txfonts}
\usepackage{mathrsfs}
\usepackage{ textcomp }

 \newtheorem{theorem}{Theorem}[section]
 \newtheorem{corollary}[theorem]{Corollary}
  
 \newtheorem{lemma}[theorem]{Lemma}
 \newtheorem{proposition}[theorem]{Proposition}
\newtheorem*{theorem*}{Theorem}
 \theoremstyle{definition}
 \newtheorem{definition}[theorem]{Definition}
\theoremstyle{definition}
\newtheorem*{definition*}{Definition}
 \theoremstyle{definition}
 \newtheorem{example}[theorem]{Example}

 \newtheorem{remark}[theorem]{Remark}
 
 \numberwithin{equation}{section}

\newcommand{\MRep}[2][;\C]{\mathbf{Rep}(#2,\phi #1)}

\newcommand{\Rep}{\operatorname{Rep}}

\newcommand{\MGL}{\operatorname{MGL}}

\newcommand{\K}{\mathbf{K}}

\newcommand{\res}{\mathfrak{res}}

\newcommand{\mg}[1]{{#1}}

\newcommand{\MRF}[3][;\C]{\mathbf{R}(#2,\phi #1 ;#3)}

\newcommand{\rK}{\widetilde{\K}}

\newcommand{\Lm}{\mathbb{L}}

\newcommand{\Hom}{\operatorname{Hom}}

\newcommand{\End}{\operatorname{End}}

\newcommand{\IrrepF}[3][;\C]{\operatorname{Irrep}(#2,\phi #1; #3)}

\newcommand{\Cliff}[2]{\operatorname{Cliff}(\mathbb{R}^{#1,#2})}

\newcommand{\MRR}[2][;\C]{\mathbf{R}(#2,\phi #1)}

\newcommand{\MGrp}{\mathbf{MGrp}}

\newcommand{\SF}{\operatorname{SchurF}}

\newcommand{\MVec}{\mathbf{Vect}}

\newcommand{\Ind}[2][(\mg{G},\phi)]{\operatorname{Ind}_{#2}^{#1}}

\newcommand{\PS}{\operatorname{P}}

\newcommand{\GL}{\operatorname{GL}}

\newcommand{\id}{\operatorname{id}}

\newcommand{\U}{\operatorname{U}}
\newcommand{\SU}{\operatorname{SU}}
\newcommand{\SO}{\operatorname{SO}}

\newcommand{\Map}{\operatorname{Map}}

\newcounter{commentcounter}

\begin{document}

\title[Magnetic Equivariant K-theory]{magnetic Equivariant K-theory}

%\thanks{BU acknowledges the financial support of the Max Planck Institute of Mathematics in Bonn, the International Center for theoeretical Physics in Trese, and the Alexander Von Humboldt Foundation. All authors aknowledge the support of  Conacyt through grant number CB-2017-2018-A1-S-30345-F-3125. }

\author{Higinio Serrano}
\address{Departamento de Matem\'{a}ticas, CINVESTAV, Av. IPN \# 2508, Col. San Pedro Zacatenco, M\'exico, CDMX. CP 07360, M\'exico}
\email{hserrano@math.cinvestav.mx}

\author{Bernardo Uribe}
\address{Departamento de Matem\'{a}ticas y Estad\'istica, Universidad del Norte, Km.5 V\'ia Antigua a Puerto Colombia, 
Barranquilla, 081007, Colombia.}
\email{bjongbloed@uninorte.edu.co}

\author{Miguel Xicot\'encatl}
\address{Departamento de Matem\'{a}ticas, CINVESTAV, Av. IPN \# 2508, Col. San Pedro Zacatenco, M\'exico, CDMX. CP 07360, M\'exico}
\email{xico@math.cinvestav.mx}

%\author{Mickey Mouse}
%\address{}
%\email{}

\date{\today}
\begin{abstract}
We present the fundamental properties of the K-theory groups of complex vector bundles endowed with actions of magnetic groups. In this work we show that the magnetic equivariant K-theory groups define an equivariant cohomology theory, we determine its coefficients, we show Bott's, Thom's and the degree shift isomorphism, we present the Atiyah-Hirzeburh spectral sequence, and we explicitly calculate two magnetic equivariant K-theory groups in order to showcase its applicability.
These magnetic equivariant K-theory
groups are relevant in condensed matter physics since they provide topological invariants of gapped Hamiltonians in magnetic crystals. 

\end{abstract}

\maketitle

\tableofcontents

\section*{Introduction} 

Symmetry in physics has long been a cornerstone for understanding the laws of nature, providing a profound link between invariance and conservation laws, as elegantly formalized in Emmy Noether’s theorem. This fundamental result establishes that every differentiable symmetry of the action of a physical system corresponds to a conserved quantity, such as energy, momentum, or angular momentum. These principles not only underpin classical and quantum mechanics but also extend to more intricate systems, where symmetry governs complex interactions and emergent phenomena.

Magnetic symmetries play a critical role in a wide array of physical systems, from the electronic properties of materials to the classification of topological phases of matter. These symmetries, often described through space magnetic groups, extend the familiar framework of crystallographic symmetry by incorporating time-reversal operations. The study of representations of magnetic groups provides a foundational understanding of how such symmetries manifest in physical systems, particularly in contexts like electronic band structures and spin configurations.

Equally significant is the role of vector bundles in mathematics and physics. When enriched with symmetries, as in the case of magnetic equivariant vector bundles, these structures become central to topological approaches in condensed matter physics, including the classification of topological insulators and superconductors. The application of K-theory, a tool from algebraic topology designed to understand
topological invariants of bundles, allows for a systematic study of electronic band bundles (Bloch bundles), leading to profound insights into their properties and invariants.

In the quest to understand the phases of matter, traditional classifications relied on concepts such as symmetry breaking. However, in the recent decades, a new paradigm has emerged based on topological features that remain unchanged under continuous deformations. These properties have reshaped our understanding of matter in profound ways.

Topological invariants, mathematical quantities that remain constant under such transformations, have proven crucial in distinguishing different phases of matter beyond conventional classifications.

The journey toward recognizing the role of topology in condensed matter physics began with theoretical predictions in the 1980s \cite{PhysRevLett.49.405} \cite{PhysRevLett.61.2015} \cite{Chang2013}, particularly in the study of the quantum Hall effect. In this phenomenon, the electrical conductivity of a two-dimensional electron system exhibits quantized plateaus, a behavior explained by a topological invariant known as the Chern number of the vector bundle of valence states, or Bloch bundle, of the material.

Further explorations \cite{PhysRevLett.95.226801} \cite{PhysRevLett.96.106802} revealed a broader class of materials, including topological insulators, characterized by conducting surfaces despite being insulating in their interior. This behavior defied previous theoretical expectations and highlighted the power of topological invariants in predicting physical properties.  Here the invariant which characterizes a state as trivial or non trivial band insulator is a number in the group $\mathbb{Z}/2$ (regardless if the state exhibits or does not exhibit a quantum spin Hall effect). This number is not zero whenever the valence bands generate the $\mathbb{Z}/2$-invariant which was shown to live in Atiyah’s real K-theory of the 3-dimensional torus \cite{PhysRevLett.95.146802}. This invariant is known as the Kane-Mele invariant.

The theoretical and experimental advancements was highlighted in the 2016 Nobel Prize in Physics, awarded to David J. Thouless, F. Duncan M. Haldane, and J. Michael Kosterlitz \cite{Gibney2016}. Their pioneering work on topological phases of matter and phase transitions unveiled new quantum states governed by topological principles. This recognition underscored the profound interplay between mathematics and physics, showcasing how abstract mathematical concepts can illuminate the structure of the physical universe.

The study of topological invariants of crystals, whether magnetic or not, relies on an explicit understanding of the symmetry group of the crystal and the equivariant K-theory groups of the 2-dimensional  and 3-dimensional torus. The relevant symmetry groups include the spatial symmetries of the crystal as well as symmetries that, when composed with the time-reversal operator, preserve the Hamiltonian. In the physics literature, these groups are known as magnetic groups or Shubnikov groups \cite{Shubnikov1964colored}, and in some mathematics literature as Orientifold groups \cite{MR2929686}. Adopting the physics term "magnetic groups" for this groups, we define their associated equivariant K-theory groups which we will refer to as ``Magnetic Equivariant K-theory groups'' to distinguish these groups from the classical complex equivariant K-theory groups \cite{segal}. 

The development of magnetic equivariant K-theory could be traced back to Atiyah's seminal 1966 article on K-theory and reality \cite{atiyahreal}. In this foundational work, Atiyah studied the smallest magnetic point group (the cylic group with two elements) and complex vector bundles with such symmetry, known as vector bundles with involution.  One year later, Atiyah and Segal \cite{atiyahsegalequivariant} extended this framework to include magnetic groups formed by the semidirect product of a group with $\mathbb{Z}/2$, where the first factor acts linearly and the second antilinearly on the complex vector bundles. This generalization allowed for the inclusion of both spatial and time-reversal symmetries in vector bundles. Subsequently, in 1970, Karoubi \cite{10.1007/BFb0059024} further generalized the theory to encompass any magnetic point group.

In the 1980s and 1990s, researchers began using additional mathematical tools, such as noncommutative geometry and the K-theory of $C^*$-algebras (algebraic K-theory), to study condensed matter systems. These approaches offered new insights and methodologies, as seen in works like \cite{ktheoryofC}, \cite{noncalg} and the references therein, highlighting the growing intersection of mathematical physics and topology.

In 2013, Freed and Moore \cite{twistedmat} applied magnetic equivariant K-theory to the study of topological phases of matter, demonstrating its significance in describing physical systems with symmetries. Building on this foundation, Gomi \cite{Gomi2017FreedMooreK}  expanded upon the K-theory framework introduced by Freed and Moore, developing it in greater detail. He introduced general twistings on the groups as cocycles and within the structure of vector bundles, extending the theory to encompass a wider range of physical phenomena. These advancements further enriched the applicability and mathematical depth of magnetic equivariant K-theory, solidifying its role in both mathematics and physics.

Our main objective on this work is to showcase the properties of the magnetic equivariant K-theory
in a simple, clear and concise way so that it could be used as a reference for non-experts in algebraic topology.
In particular we had in mind the condensed matter physicist interested in determining the topological
invariants of prescribed magnetic crystals.   We have thus avoided the use of categorical language such
as the one of groupoids done in \cite{Gomi2017FreedMooreK} and we have focused our attention
on highlighting the most interesting features of the Magnetic Equivariant K-theory. 

We decided to write our work in an style as similar as possible to the one Atiyah and Segal had in their
papers  \cite{atiyahreal, segal, atiyahsegalequivariant}. In some sense this work simply puts Atiyah and Segal's work in the framework of magnetic groups and
we show that all of their results in the  real and equivariant K-theory  generalize to the magnetic equivariant one.  We put great care in highlighting the simplicity and clarity of the theory so that our work could be used
as a reference manual for calculational purposes.

This work has three main sections. We start in the first section with a review of the theory of representations 
of compact magnetic groups as it is presented in Wigner's seminal book of group theory and quantum mechanics \cite{wigner}. We include  a review on how to determine the type of an irreducible magnetic representation
using the appropriate Schur-Frobenius  indicator.

 We continue with the bulk of our work in the second
section where we define the magnetic equivariant K-theory groups and we show that the Bott isomorphism, the Thom isomorphism and the degree shift isomorphism, among others, generalize to the magnetic setup.
Moreover, we use Atiyah and Segal's approach to the real K-theory groups from the point of view of Clifford modules to determine the coefficients of the magnetic equivariant K-theory as an equivariant cohomology theory.
We end up this section with the description Atiyah-Hirzebruch spectral sequence showcasing the 
appearance of the Bredon cohomology with coefficients in the Grothendieck group of magnetic representations as its second page.

We finish our work in the third section where we present an explicit calculation of magnetic equivariant K-theories relevant in the study of 2-dimensional altermagnetic materials \cite{Altermagnetism1,Altermagnetism2,Altermagnetism3}. We give a concise description of the
incorporation of the spin-orbit interaction in the setup, and how the magnetic group is enhanced to the appropriate double cover and the K-theory groups becomes twisted. We focus our attention
 in the magnetic cyclic group $\mathbb{Z}/4$ and its action on the 2-dimensional torus via 4-fold rotations.
 We explicit calculate its  magnetic equivariant K-theory groups and we show how the
 K-theory groups are calculated in the case that the spin-orbit interaction is taken into account.
 The calculation of our last example shows that 2-dimensional magnetic crystals which preserve the combination
 of the 4-fold rotation symmetry composed with time reversal, possess a bulk $\mathbb{Z}/2$ invariant
 that distinguishes trivial insulators from topological ones. This calculation
allowed Gonz\'alez-Hern\'andez and the first two authors  to show that 2-dimensional
insulating altermagnets with $C_4\mathbb{T}$ symmetry possess a non-trivial
phase such as the one of topological insulators \cite{gonzalezhernandez2024spinchernnumberaltermagnets}. 

%%%%%%%%%%%%%%%%%%%%%%%%%%%

%%%%%%%%%%%%%%%%%%%%%%%%%%%

%%%%%%%%%%%%%%%%%%%%%%%%%%%

%%%%%%%%%%%%%%%%%%%%%%%%%%%

\section*{Summary of results}

Here we highlight some of the results that we have collected in this work in relation to the magnetic equivarant K-theory.

\subsubsection*{\bf Classification of magnetic representations.} We give a summary of Wigner's account \cite{wigner} on the classification of magnetic representations in Thm. \ref{wigner}. For an irreducible
magnetic representation $V$ of the magnetic group $(G,\phi)$, the vector space
\begin{align}
\End_{\MRep[]{\mg{G}}}(V)\cong \mathbb{F}
\end{align}
determines its type with $\mathbb{F} \in \{\mathbb{R}, \mathbb{C}, \mathbb{H} \}$. The Schur-Frobenius indicator of Def. \ref{Schurindicator} gives a formula to determine its type.

\subsubsection*{{\bf Coefficients in magnetic equivariant K-theory.}} The magnetic equivariant K-theory of a point splits into real, complex and symplectic K-theories according to the number of irreducible representations of the magnetic group of real, complex and quaternionic type. The formula appears in Eqn. \ref{decomposition K_G} and reads:
 \begin{align} 
    \K_G^*(*) \cong KO^*(*)^{\oplus n_\mathbb{R}} \oplus KU^*(*)^{\oplus n_\mathbb{C}} \oplus KSp^*(*)^{\oplus n_\mathbb{H}}        
    \end{align}
    where $n_{\mathbb{F}}$ is the number of irreducible representations of $\mathbb{F}$-type.
The proof is on Thm. \ref{coefofK}.

\subsubsection*{{\bf Degree shift isomorphism.}} 
The isomorphism between real K-theory and symplectic K-theory $K\mathbb{R}^* \cong KSp^{*+4}$ is generalized to the context of magnetic equivariant K-theory inducing an isomorphism of degree 4
 \begin{align}
\K^{*}_{G}(X)  \stackrel{\cong}{\longrightarrow} {}^{\widehat{G}}\K^{*+4}_{G}(X).
\end{align}
where the right hand side is the $\widehat{G}$-twisted magnetic $G$-equivariant K-theory with $\widehat{G}$
being defined as the pullback group $\phi^*(\mathbb{Z}/4)$. The proof appears in Thm. \ref{degree shift}.
Whenever the $\mathbb{Z}/2$-extension $\widehat{G}$
is trivial, then the magnetic equivariant K-theory becomes 4-periodic. This is Prop. \ref{4-periodic}.

\subsubsection*{{\bf Thom isomorphism.}} For a magnetic $(G,\phi)$-equivariant vector bundle $E \stackrel{p}{\longrightarrow} X$
over the $G$ compact space $X$, the homomorphism
 \begin{align} \nonumber
\varphi_*: \K^{-p}_{\mg{G}}(X) & \stackrel{\cong}{\longrightarrow} \K^{-p}_{\mg{G}}(E)\\
F & \longmapsto \Lambda^*_E \otimes p^*F 
\end{align}
is the Thom isomorphism. Here  $\Lambda^*_E$ is the Thom class of $E$ and the proof appears in Thm. \ref{tiso}.

\subsubsection*{{\bf Atiyah-Hirzebruch spectral sequence.}}
Calculations are important in condensed matter physics and therefore a procedure for determining the magnetic equivariant K-theory is relevant. The Atiyah-Hirzebruch spectral sequence is one tool that permits calculate such groups. Filtering $\K_G^*(X)$ using
the $G$-CW decomposition of $X$, we obtain a spectral sequence whose second page is
 \begin{align}
E^{n,t}_2\cong H^n\left(X;\K^t_{G}\right) \Longrightarrow \K^*_{G}(X),
\end{align}
and that abuts to $\K^*_{G}(X)$. The second page is Bredon's equivariant cohomology with coefficients in the group of magnetic representations.
The proof appears in Thm. \ref{ahss}.

\subsubsection*{{\bf $C_4 \mathbb{T}$ symmetry in 2-dimensional torus.}}
The symmetry $C_4$ stands for a 4-fold rotation on the
2-dimensional torus, $\mathbb{T}$ stands for time reversal symmetry and. $C_4 \mathbb{T}$ is the composition. We calculate the magnetic $ \langle C_4 \mathbb{T} \rangle$-equiviariant K-theory of the 2-dimensional torus $T^2$ in both the context of no spin orbit interaction in Prop. \ref{C4T NOSOC}  and spin orbit interaction
in Prop. \ref{C4T SOC}:

\begin{align}
\begin{tabular}{|c|cccc|}
 \hline
  $n$  & $0$ & $-1$ & $-2$ & $-3$   \\
      \hline
 $  \K^n_{\mathbb{Z}/4}(T^2) $ &$ \mathbb{Z}^4$ & $\mathbb{Z}/2$ & $\left( \mathbb{Z} \oplus \mathbb{Z}/2 \right) \oplus \mathbb{Z}$ &$ 0 $
\\ \hline
  $  {}^{\mathbb{Z}/8}\K^n_{\mathbb{Z}/4}(T^2) $ &$ \mathbb{Z}^2  \oplus \mathbb{Z}/2$ & $0$ & $ \mathbb{Z}^3\oplus \mathbb{Z}$ &$ \mathbb{Z}/2 $
   \\ \hline
\end{tabular}
\end{align}
 In both cases the magnetic equivariant K-theory are 4-periodic. Here we highlight that the bulk invariants for $n=0$ are trivial for the first case and $\mathbb{Z}/2$ for second case, while for $n=-2$ the bulk invariants are $\mathbb{Z}$ in both cases. Bulk invariants are the ones that are due to the presence of the 2-cell.

%%%%%%%%%%%%%%%%%%%%%%%%%%%%

%%%%%%%%%%%%%%%%%%%%%%%%%%%%

%%%%%%%%%%%%%%%%%%%%%%%%%%%%

%%%%%%%%%%%%%%%%%%%%%%%%%%%%%%

\section*{Notation}

In this work several mathematical structures have been defined. In order to make the text easily accessible, we list the most important symbols that we have used
throughout our work. They will come with a brief explanation and a cross-reference to where they are introduced for the first time. We highlight that
symbols in boldface are the ones associated to magnetic information.

\begin{itemize}
    \item[$(G, \phi)$]  - Magnetic group. Consists of a group $G$ together with a surjective homomorphism $\phi: G \to \mathbb{Z}/2.$ Def. \ref{def magnetic group}.
    \item[$G_0$] - Core of the magnetic group. Kernel of
    $\phi: G \to \mathbb{Z}/2$ consisting of the elements that act unitarily on magnetic representations. Def. \ref{def magnetic group}.
     \item[$\mathbf{Rep}(G,\phi)$] - Category of magnetic representations of $(G,\phi)$. Def. \ref{def representation}.
     \item[$\mathrm{Rep}(G_0)$] - Category of complex
     representations of the group $G_0.$  Def. \ref{def representation}.
     \item[$\mathbb{K}$] - Complex conjugation operator.
     \item[$\mathbf{R}(G, \phi)$] - Grothendieck group of isomorphism classes of magnetic representations of $(G, \phi)$, also denoted $\mathbf{R}(G)$.  Def. \ref{def representation}, Eqn. \eqref{splitring}.
     \item[$\mathbf{R}(G, \phi; \mathbb{F})$] 
     - Grothendieck group of isomorphism classes of magnetic representations of $(G, \phi)$ of type $\mathbb{F}$ for $\mathbb{F} \in \{ \mathbb{R}, \mathbb{C}, \mathbb{H} \}$, also denoted $\mathbf{R}(G, \mathbb{F})$.  Eqn. \eqref{splitring}.
     \item[$R(G_0)$] - Grothendieck ring of complex representations of the group $G_0$.  Def. \ref{def representation}.
     \item[$\mathbb{R}^{p,q}$] - Real vector space  $\mathbb{R}^{q+p} = \mathbb{R}^q \oplus i \mathbb{R}^p$ with the involution $x+iy$ maps to $x-iy$.
     \item[$D^{p,q}$] - Unit disk in $\mathbb{R}^{p,q}$.
     \item[$S^{p,q}$] - Unit sphere in $\mathbb{R}^{p,q}$.
     \item[$X$] - $G$-CW-complex.
     \item[$\mathbf{Vect}_{(G,\phi)}(X)$] - Category of magnetic $(G,\phi)$-equivariant complex vector bundles. Also denoted $\mathbf{Vect}_G(X)$. Def. \ref{def magnetic K-theory}.
     \item[$\K_G(X)$] - Magnetic $G$-equivariant K-theory of $X$. Def. \ref{def magnetic K-theory}.
     \item[$KU_{G_0}(X)$] - $G$-equivariant complex K-theory of $X$ \cite{segal}.
     \item[$K\mathbb{R}^*$] - Atiyah's real K-theory of involution spaces \cite{atiyahreal}.
     \item[$K \mathbb{H}^*$] - Quaternionic K-theory of involution spaces.
     \item[$KO^*$] - K-theory of real vector bundles.
     \item[$KSp^*$] - Dupont's symplectic K-theory \cite{dupontsymplectic}.
     \item[$\widetilde{\K}_G(X)$] - Reduced magnetic equivariant K-theory. Prop. \ref{defireduc}.
     \item[$C_{p,q}$] - Clifford algebra over $\mathbb{R}^{p+q}$ with a negative quadratic form for $q$ variables and positive for $p$ variables. Eqn. \eqref{Cliffordpq}.
     \item[$M^{p,q}_{\mathbb{R}}(\mg{G},\phi)$]  - Grothendieck group of magnetic graded $C_{p,q}[G, \phi]$-modules. Def. \ref{magnetic Clifford}.
     \item[$(\widetilde{G}, \widetilde{\phi})$] -Magnetic central extension of $(G,\phi)$ by $A=(\mathbb{Z}/2)^k$. See \S \ref{section degree shift}.
     \item[${}^{(\widetilde{G},\chi)}\K_{G}^{*}(X)$] 
        - $(\widetilde{G},\chi)$-twisted magnetic $G$-equivariant K-theory groups of $X$ with $\chi$  a fixed character of $A$. Def. \ref{twisted magnetic K-theory}.
     \item[$(\widehat{G}, \widehat{\phi})$] - $\mathbb{Z}/2$-central magnetic extension of $(G,\phi)$ defined as the pullback $\phi^*(\mathbb{Z}/4)$. Eqn. \eqref{widehatG}.     \item[${}^{(\widehat{G},\widehat{\sigma})}\K_{G}^{*}(X)$] - Twisted magnetic equivariant K-theory which is the image of the degree shift isomorphism. Thm. \ref{degree shift}.
     \item[$\Lambda^*_E$] - Thom class of the bundle $E \to X$; it defines an element in $\K_G(E)$. Ex. \ref{Thom class}.
     \item[$H^*(X; \K_G^*)$] - Bredon's equivariant cohomology with coefficients in magnetic equivariant K-theory. Thm. \ref{ahss}.
     \item[$\widetilde{G}^{soc}$] - $\mathbb{Z}/2$ central extension of $G$ defined by the spin orbit interaction. Def. \ref{def soc group}.
     \item[${}^{\widetilde{G}^{soc}}\K^*_G((S^1)^3)$] - Twisted magnetic equivariant K-theory groups for the 3-dimensional torus in the presence of spin orbit interaction. Def. \ref{definition twisted soc}.
\end{itemize}

%%%%%%%%%%%%%%%%%%%%%%%%%%%%%%%%%%%%%%%%%%%%

%%%%%%%%%%%%%%%%%%%%%%%%%%%%%%%%%%%%%%%%%%%%

%%%%%%%%%%%%%%%%%%%%%%%%%%%%%%%%%%%%%%%%%%%%%%%

%%%%%%%%%%%%%%%%%%%%%%%%%%%%%%%%%%%%%%%%%%%%%%%

\section{Magnetic groups and their representations}

In this chapter we develop the theory of representations of magnetic point groups, namely groups that come equipped with a surjective homomorphism onto $\mathbb{Z}/2$. These groups emerge as symmetries in quantum systems, particularly in contexts involving time-reversal symmetry or combined spatial and time-reversal symmetries \cite{mpgsp}. Magnetic point groups act via complex linear or antilinear maps, depending on whether a group element lies in the kernel of the homomorphism. This introduces an added layer of complexity in their representation theory.

A key difference between the representations of magnetic point groups and ordinary groups lies in the structure of the morphisms between irreducible representations, as explained by Schur's lemma. For non magnetic (ordinary) groups, the set of homomorphisms of an irreducible representation consists solely of complex multiples of the identity matrix. However, in the case of magnetic groups, the set of morphisms of irreducible representations can take one of three forms: real, complex, or quaternion multiples of the identity.

This distinction is not merely formal; it has deep physical implications. In quantum systems, the presence of time-reversal symmetry or combined symmetries can lead to phenomena like Kramers' degeneracy, where pairs of states remain degenerate due to time-reversal symmetry, even in the absence of other symmetries \cite{TRK}. The analysis of how representations change under the action of antilinear maps is crucial for understanding such phenomena, particularly in systems with strong spin-orbit interactions or in materials that exhibit topological phases \cite{PhysRevLett.98.106803}. Thus, the study of morphisms in this more general setting of magnetic point groups opens a pathway to understanding physical systems that cannot be described by ordinary group theory alone.

Moreover, we will delve into the role of extensions of magnetic point groups by subgroups of $\U(1)$, which naturally lead to what are known as twisted representations. These twisted representations are particularly relevant in cases where the time-reversal operator no longer squares to the identity but instead to minus the identity. This occurs, for instance, in systems with Spin-Orbit Coupling, where time-reversal symmetry lifts to a group of order 4, \cite{annurev:/content/journals/10.1146/annurev-conmatphys-031115-011319}. The spatial symmetries, which often reside in $\SO(3)$, must then be lifted to its double cover, $\SU(2)$, to accommodate the spin-1/2 particles in quantum systems. This modification affects not only the algebraic structure of the symmetry group but also how representations of the group relate to physical observables, such as energy levels and degeneracies.

%%%%%%%%%%%%%%%%%%%%%%%%%%%%

%%%%%%%%%%%%%%%%%%%%%%%%%%%%
\subsection{Magnetic groups}

\begin{definition} \label{def magnetic group}

    A \textbf{magnetic group}  consists of a pair $(G,\phi_G)$ with $G$ a group and $\phi_G: {G}\longrightarrow \mathbb{Z}/2$  a surjective homomorphism. The group $G_0:=\ker \phi_G$ will be called the \textbf{core group} of the magnetic group.
This information is usually depicted in the short exact sequence of groups 
\begin{equation}
\xymatrix{ 1\ar[r]& G_0 \ar[r] &G \ar[r]^{\phi_G} & \mathbb{Z}/2 \ar[r] & 0.}
\end{equation}
A \textbf{morphism} between the magnetic groups $({G},\phi_{G})$ and $({H},\phi_{H})$  is a homomorphism $f:{G}\longrightarrow {H}$ such that $\phi_{{H}} \circ f=\phi_{{G}}$.
\end{definition}

We will denote by $\MGrp$ the category of magnetic groups and their morphisms.

A \textbf{subgroup} of a magnetic group $({G},\phi_{{G}})$ is a magnetic group $({H},\phi_{{H}})$ such that $\mg{H}\leq \mg{G}$ and this inclusion is a morphism in $\MGrp$.  It is a  \textbf{normal subgroup} if ${H}\trianglelefteq {G}$.

Note that not every subgroup of $H_0$ of $G_0$ comes from a magnetic subgroup $(\mg{H},\phi)$. For example, take the magnetic group $1\longrightarrow \mathbb{Z}/4\longrightarrow \mathbb{Z}/8\longrightarrow \mathbb{Z}/2\longrightarrow 1$ and  the subgroup $\mathbb{Z}/2$ of $ \mathbb{Z}/4$. This group
$\mathbb{Z}/2$ is not the core of any magnetic subgroup of $(\mathbb{Z}/8, \phi_{\mathbb{Z}/8})$.

In condensed matter physics, the magnetic groups that classify the symmetries of a crystal which take into account additional symmetries such as time-reversal, are called {\textbf{magnetic space groups}. Whenever the translational
symmetries are mod-out, the quotient group is called the  \textbf{magnetic point group} (see chapter 7 of \cite{bradley1972mathematical}).

\begin{example} Perhaps the most known example of a magnetic group is the \textbf{magnetic general linear group}.
For a complex vector space $V$ denote by $\MGL(V)$ the group of invertible linear or antilinear operators of $V$ where
   $\phi: \MGL(V)  \to \mathbb{Z}/2 $ maps an operator to $0$ if it is complex linear
and $1$ if it is antilinear. Its core group $\MGL(V)_0=\GL(V)$ is the group of invertible linear operators of $V$.

The general linear group with the complex conjugation automorphism is an example of what is known as a {\it Real group} \cite{atiyahsegalequivariant}. That is, 
a group $G_0$ with an automorphism $\tau: G_0\longrightarrow G_0$ such that $\tau^2=\id$. 
The group $G_0\rtimes \mathbb{Z}/2$ together with the projection $\pi_2$ to $\mathbb{Z}/2$ becomes a magnetic group.
\end{example}

\subsection{Representations of magnetic groups}

The representation theory of magnetic groups extends the classical representation theory of ordinary groups by incorporating time-reversal symmetry. Magnetic groups arise in systems where time-reversal symmetry or a combination of spatial and time-reversal symmetries plays a significant role, such as in quantum systems with magnetic or spin-orbit coupling effects.

\begin{definition}
    A \textbf{representation} of a magnetic group $({G},\phi)$ on a finite dimensional complex vector space $V$ is a morphism of magnetic groups $\rho: G \longrightarrow \MGL(V)$.
\end{definition}

\begin{example}
    Consider the magnetic group $(G_0\times \mathbb{Z}/2,\pi_2)$ where $\phi=\pi_2$ is the projection onto the second factor and the core group is $G_0$.
    
    Take any representation $\rho: G_0\times \mathbb{Z}/2 \to \MGL(V)$ and note that $D(g): = \rho(g,0)$ for $g \in G_0$ defines a complex representation of the group $G_0$. 
    Now set $J:=\rho (e,1)$, which is an antilinear operator commuting with $D$ such that $J^2=\id_V$. 
    So $D:G_0\longrightarrow \GL(V)$ is a complex representation and $J:V\longrightarrow V$ an antilinear involution which commutes with $D$. Therefore and this $(D,J)$ is a {\bf real representation} of $G_0$ in the sense of Atiyah \cite{atiyahreal}.
  
    Clearly we have the bijection
    \begin{align}
\left\{\begin{array}{c}
    \text{Real representations}\\ \text{of the group } G_0
    \end{array}\right\} \longleftrightarrow \left\{ \begin{array}{c}
    \text{Representations of the} \\ \text{magnetic point group }  (G_0\times\mathbb{Z}/2,\pi_2)\end{array}\right\}.
\end{align}

\end{example}

\begin{remark}
    Let $(D:G_0\longrightarrow \GL(V), J:V\longrightarrow V)$ be a real representation of $G_0$. Then $W=\{v\in V \,|\, J(v)=v\}$ is a vector space over the real numbers, which is closed under the action of $D$. Thus $G_0$ acts on $W$ by matrices with real coefficients. On the other hand, if $W$ is a vector space, over the real numbers and $D:G_0\longrightarrow \GL(W)$ is a real representation, then $(D\otimes 1:G_0\longrightarrow \GL(W\otimes \mathbb{C}), 1\otimes \K: W\otimes \mathbb{C}\longrightarrow W\otimes \mathbb{C})$ is a real representation of $G_0$. These constructions are inverses one to the other other.
\end{remark}

\begin{example}\label{Z4}
    Take the magnetic point group $(\mathbb{Z}/4,\operatorname{mod} 2)$ and note that there are two representations that come straight to mind.  One can have the one dimensional representation where the generator of $\mathbb{Z}/4$ is mapped to $\mathbb{K}$, the 
complex conjugation operator in complex vector spaces, or the two dimensional representations where the generator is mapped to  $ \begin{pmatrix}
        0&1\\-1&0
    \end{pmatrix}\cdot \mathbb{K}$.
\end{example}

\subsection{Morphisms of representations}

Here we give the notion of morphism of representations of magnetic groups and that of an irreducible one. Next we state and prove an appropriate version of Schur's lemma in this context. The material presented here is based on the pioneer work of  Wigner \cite{wigner} who called representations of magnetic groups {\bf corepresentations}.

\begin{definition} \label{def representation}
    A \textbf{morphism} of representations $\rho_i: {G}\longrightarrow \MGL(V_i)$ $i=1,2$ of the magnetic group $({G},\phi)$ is a (complex) linear transformation $T:V_1\longrightarrow V_2$ such that $T\rho_1(g)=\rho_2(g)T$.
\end{definition}

Two representations $\rho_1,\rho_2$ are \textbf{isomorphic} if there exists an morphism of representations which is invertible. 

\noindent
For $(\mg{G},\phi)$ a fixed magnetic group denote:
\begin{itemize}
    \item $\MRep[]{\mg{G}}$ the category of representations of magnetic groups with their morphisms.  Denote its Grothendieck group of isomorphism classes of representations  by $\MRR[]{\mg{G}}$. 
    \item $\Hom_{\MRep[]{\mg{G}}}(V,W)$ the $\mathbb{R}$-vector space of morphisms of representations and by
    $\End_{\MRep[]{\mg{G}}}(V)$ the endomorphisms of a representation $V$.
    \item $\Rep(G_0)$ the category of representations of the core group $G_0$ and by $\operatorname{R}(G_0)$ its Grothendieck ring of complex representations.
   \end{itemize}

\begin{example}
    Let $\rho:\mg{G}\longrightarrow \MGL(V)$ be a representation of the magnetic group $(\mg{G},\phi)$ and $z\in \mathbb{C}$ a complex number such that $|z|=1$. Define the map
    \begin{align}
    \rho_1(g):=\overline{z}\cdot\rho(g)\cdot z=\overline{z}\cdot \mathbb{K}^{\phi(g)}(z)\cdot\rho(g)=\begin{cases}
        \rho(g)& \text{ if } \phi(g)=0\\ \overline{z}^2\rho(g) & \text{ if } \phi(g)=1
    \end{cases}
    \end{align}
    that is, $\rho_1$ and $\rho_2$ differ only by a factor of $\overline{z}^2$ in the antilinear part.
    
    This map is a representation of the magnetic point group because
    \begin{align}
        \rho_1(gh)&=\overline{z}\cdot\rho(gh)\cdot z\\
        &= \overline{z}\cdot\rho(g)\cdot\rho(h)\cdot z\\
        &= \overline{z}\cdot\rho(g)\cdot z\overline{z} \cdot\rho(h)\cdot z\\ 
        &= \rho_1(g)\cdot \rho_1(h).
    \end{align}
    
    Now, if $T$ denotes multiplication by $z$ then
    \begin{align}
        T\rho_1(g)=z\overline{z}\rho(g)z
        = \rho(g)z
        = \rho(g)T,
    \end{align}
    and therefore the representations $\rho$ and $\rho_1$ are isomorphic. Hence the antilinear part of representations of magnetic groups could be multiplied by any complex number of norm one and the isomorphism class of the representation remains unaffected.
\end{example}

    If we have a representation $V$ of the magnetic point group $(\mg{G},\phi)$, then the \textbf{restriction} to the core group $G_0$ is a representation of $G_0$ and is denoted by $\res(V)$. This restriction map defines a functor at the level of the categories of representations. 
    \begin{equation}\label{res}\xymatrix{
    \MRep[]{\mg{G}} \ar[r]^{\hspace{0.3cm}\res}& \Rep(G_0).
    }  
    \end{equation}

    A \textbf{subrepresentation} of a representation $V$ of the magnetic point group $(\mg{G},\phi)$ is a subspace $W\subset V$ such that the restriction of the $\mg{G}$-action to $W$ is also a representation of $(\mg{G},\phi)$.
 A representation $V$ of $(\mg{G},\phi)$ is \textbf{irreducible} if its only representations are $V$ and $0$.

Now, the category of representations of finite magnetic groups is semisimple.
\begin{proposition}
    Every representation of a finite magnetic group is a sum of irreducible representations.
\end{proposition}
 Which follows from Maschke's lemma applied to representations of magnetic groups.
\begin{lemma}[Maschke]\label{maschke}
    If $W$ is a subrepresentation of $V$ of a finite magnetic group, then there exists a subrepresentation $W'$ of $V$ such that $V=W\oplus W'$.
\end{lemma}

\begin{proof}
    Write $V$ as a direct sum $V=W\oplus U$ (with $U$ not necessary a subrepresentation) and consider the projection $\pi_1: V\longrightarrow W$ given by $\pi_1(w+u)=w$. Define
    \begin{align}
     \nonumber   \pi: V&\longrightarrow W\\ v&\longmapsto \sum_{g\in \mg{G}}g\cdot \pi_1(g^{-1}\cdot v).
    \end{align}
    The morphism $\pi$ is in $ \Hom_{\MRep[]{\mg{G}}}(V,W)$ because if $h\in \mg{G}$ and $v\in V$, then
    \begin{align}
        \pi(h\cdot v)&= \sum_{g\in \mg{G}} g\cdot \pi_1(g^{-1}\cdot(h\cdot v)) \\ &= \sum_{g\in \mg{G}} g\cdot \pi_1((g^{-1}h)\cdot v) \\ &= \sum_{g\in \mg{G}} (hh^{-1}g)\cdot \pi_1((h^{-1}g)^{-1}\cdot v) \\ &= h\cdot \sum_{g\in \mg{G}} (h^{-1}g)\cdot \pi_1((h^{-1}g)^{-1}\cdot v) \\ &= h\cdot \sum_{l\in \mg{G}} l\cdot \pi_1(l^{-1}\cdot v) \\ &= h\cdot \pi(v)
    \end{align}
    The morphism $\pi$ is surjective, this follows since $\pi_1(w)=w$ for $w\in W$, so $\pi(w)=|\mg{G}|w$.
    If we denote by $W'=\ker\pi$, then we have the decomposition as a direct sum of representations of $(\mg{G},\phi)$ 
    \begin{align}
    V\cong\operatorname{Im}\pi \oplus \ker\pi=W\oplus W'.
    \end{align}
\end{proof}

One key step required to classify the irreducible representations of magnetic groups $(\mg{G},\phi)$ is to construct the induced representations of the core $G_0\leq \mg{G}$.
\begin{definition}\label{induction}
    Let $(\mg{G},\phi)$ be a finite magnetic point group and $V$ a representation of the core group $G_0$. The \textbf{induced representation} of $V$ is the complex vector space
\begin{align} 
\Ind{G_0} V:= \mathbb{C}[\mg{G},\phi]\otimes_{\mathbb{C}[G_0]} V
\end{align}
    where
       $\mathbb{C}[\mg{G},\phi]$ is the \textbf{regular representation} of $(\mg{G},\phi)$, i.e.  the complex vector space generated by $\mg{G}$ together the action of $(\mg{G},\phi)$ given by 
\begin{align}
g\cdot \left(\sum_i z_ig_i\right)=\sum_i \K^{\phi(g)}(z_i)gg_i,
\end{align}
   and  whose action of $(\mg{G},\phi)$ on $\Ind{H}V$ is given by its action on the first factor.
\end{definition}

For a fixed element $a_0\in \mg{G}\backslash G_0$, we have the following presentation for the induced representation
    \begin{align}
        \Ind{G_0}V= V \oplus a_0V
    \end{align}
    
where 
 $a_0V=\{a_0v\, |\, v\in V\}$ is a complex vector space with the scalar product given by
 \begin{align} \label{conjugaterepresentation1}
z\cdot a_0v=a_0\overline{z}v.
\end{align}
The vector space  $a_0V$ is a representation of $G_0$ with action
\begin{align} \label{conjugaterepresentation2}
g\cdot a_0v=a_0(a_0^{-1}ga_0v),
\end{align}
 and this representation is called the \textbf{conjugate representation} of $V$ by $a_0$. 

If $\rho: G_0\longrightarrow \GL(V)$ is the representation of the core, then the conjugate 
representation $a_0V$ can be written by the homomorphism $\rho': G_0\longrightarrow \GL(V)$ where 
\begin{align}
\rho'(g)=\overline{\rho(a^{-1}_0ga_0)}.
\end{align}  
Here we have implicitly assumed the isomorphism $a_0 V \cong V$, $a_0v \mapsto v$.

         The action of $(\mg{G},\phi)$ on $V\oplus a_0V$ is given as follows

        \begin{enumerate}
            \item If $g\in G_0$, then $g\cdot(v,a_0w)=\left(gv,a_0(a_0^{-1}ga_0w)\right)$. We can write this formula as
            \begin{equation}\label{formulaind0}
            g\cdot \begin{pmatrix}
                v \\ w 
            \end{pmatrix}=\begin{pmatrix}
                \rho(g) & 0 \\ 0 & \rho'(g)
            \end{pmatrix}\begin{pmatrix}
                v\\ w
            \end{pmatrix}.\end{equation}
            \item If $g\not\in G_0$, then we can write $g=a_0g'$ for some $g'\in G_0$ and 
            \begin{align}
                g\cdot (v,a_0w)&=a_0g'\cdot (v,a_0w) \notag\\
                &=a_0\cdot \left(g'v,a_0(a_0^{-1}g'a_0w)\right) \notag\\
                &= (a_0g'a_0w, a_0g'v)\notag \\
                &=\left(ga_0 w,a_0(a_0^{-1}g v)\right).
            \end{align}

        We can write the last action in the matrix form
        \begin{equation}\label{formulaforind}
        g\cdot \begin{pmatrix}
                v \\ w 
            \end{pmatrix}=\begin{pmatrix}
                0 & \rho(ga_0) \\ \rho'(ga^{-1}_0) & 0
            \end{pmatrix}\mathbb{K} \begin{pmatrix}
                v\\ w
            \end{pmatrix}.\end{equation}
 \end{enumerate}  

The restriction and induction constructions are related by the so called Frobenius reciprocity lemma:

\begin{lemma}[Frobenius reciprocity]\label{frobrec}
    Let $(\mg{G},\phi)$ be a finite magnetic point group, $V$ a representation of the core group $G_0$ and $W$ a representation of $(\mg{G},\phi)$. Then there exist natural bijection 
    \begin{align}
    \Hom_{\MRep[]{\mg{G}}}\left(\Ind{G_0}V,W\right)\overset{\cong}{\longrightarrow} \Hom_{\Rep(G_0)}\left(V,\res W\right).
    \end{align}
\end{lemma}

\begin{proof}
    Define the functions
    \begin{align}
  \nonumber      F: \Hom_{\MRep[]{\mg{G}}}\left(\mathbb{C}[\mg{G},\phi]\otimes_{\mathbb{C}[G_0]}V,W\right) & \longrightarrow \Hom_{\Rep(G_0)}(V,\res W)\\ 
f& \longmapsto F(f)(v)=f(e\otimes v) 
    \end{align}    
    \begin{align}
  \nonumber      G: \Hom_{\Rep(G_0)}(V,\res W)& \longrightarrow \Hom_{\MRep[]{\mg{G}}}\left(\mathbb{C}[\mg{G},\phi]\otimes_{\mathbb{C}[G_0]}V,W\right)\\ f& \longmapsto G(f)(g\otimes v)=gf(v) 
    \end{align}
    $F$ and $G$ are inverses one of each-other:
    \begin{enumerate}
        \item if $f\in \Hom_{\MRep[]{\mg{G}}}\left(\Ind{G_0}V,W\right)$ and $g\otimes v\in \Ind{G_0}$ is an elementary tensor, then
\begin{align} 
(G\circ F)(f)(g\otimes v)=G(F(f))(g\otimes v)=gF(f)(v)=gf(e\otimes v)=f(g\otimes w).
\end{align}
        \item If $h\in \Hom_{\Rep(G_0)}(V,\res W)$ and $v\in V$, then
\begin{align}(F\circ G)(h)(v)=F(G(h))(v)=G(h)(e\otimes v)=eh(v)=h(v).
\end{align}
    \end{enumerate}
\end{proof}

\begin{remark}
    Given any subgroup $H$ of a magnetic point group $(\mg{G},\phi)$ and a representation $V$ of $H$ we can define an induced representation $\Ind{H}V$ of $(\mg{G},\phi)$ and also prove the Frobenius reciprocity theorem. We need to be careful because there are two types of subgroups of $(\mg{G},\phi)$, hence $V$ could be an ordinary representation or a magnetic one, and of course we have two types of restrictions and two types of inductions.
\end{remark}

\subsection{Classification of irreducible representations}

The irreducible representations of finite magnetic groups split into three different types, one for each field $\mathbb{R}$, $\mathbb{C}$ and $\mathbb{H}$. Let us see how this is carried out.

\begin{lemma}\label{signodeT}
    Let $(\mg{G},\phi)$ be a finite magnetic group, $\rho:G_0\longrightarrow \GL(V)$ an irreducible representation of $G_0$, $a_0\in \mg{G}\backslash G_0$ and $\rho':G_0\longrightarrow \GL(V)$ the conjugate representation by $a_0$. Then either:

    \begin{enumerate}
        \item the representations $\rho$ and its conjugate $\rho'$ are non isomorphic $\rho'\not\cong \rho$, or
        \item the representation $\rho$ and its conjugate $\rho'$ are isomorphic $\rho\cong\rho'$ and there exist an isomorphism $T\in \Hom_{\Rep(G_0)}(\rho',\rho)$ such that
 \begin{align}
\rho(a^2_0)=\pm T\overline{T}.
\end{align}
    \end{enumerate}

We say the representation $\rho$ is of \textbf{complex type} if it satisfies the first case, \textbf{real type} if satisfies the second case with $\rho(a^2_0)=+ T\overline{T}$ and \textbf{quaternionic type} the case $\rho(a^2_0)=- T\overline{T}$.
\end{lemma}
\begin{proof}
    If the representations $\rho$ and $\rho'$ are non isomorphic, then we are done.
    If $\rho$ and $\rho'$ are isomorphic representations, then there exists a isomorphism  $S\in \Hom_{\Rep(G_0)}(a_0V,V)$ such that 
    \begin{equation}\label{isomorphisprincipal}
        \rho'(g)=S^{-1}\rho(g)S 
    \end{equation}
    for all $g\in G_0$. Let us do the following computation with respect to some basis
    \begin{align}
        \rho(a^{-2}_0)\rho(g)\rho(a^{2}_0)=& \rho(a^{-2}_0ga^{2}_0) \hspace{0.5cm}\text{ homomorphism property}\\ =& \rho(a^{-1}_0(a^{-1}_0ga_0)a_0) \\ =& \overline{\rho'(a^{-1}_0ga_0)} \hspace{0.5cm} \text{ definition of }\rho' \\
        =& \overline{S}\,^{-1}\,\overline{\rho(a^{-1}_0ga_0)}\,\overline{S} \hspace{0.5cm} \text{ complex conjugate Eqn.} \eqref{isomorphisprincipal}\\ =& \overline{S}\,^{-1} \rho'(g)\overline{S} \hspace{0.5cm} \text{ definition of }\rho' \\ =&  \overline{S}\,^{-1}S^{-1}\rho(g)S\overline{S} \hspace{0.5cm}\text{ Eqn.}\eqref{isomorphisprincipal}.
    \end{align}
    Write the last equality as:
    \begin{align}
        S\overline{S}\rho(a^{-2}_0)\rho(g)\rho(a^{2}_0)\overline{S}\,^{-1}S^{-1}&=\rho(g) \\
        \left(S\overline{S}\rho(a^{-2}_0)\right)\rho(g)\left(S\overline{S}\rho(a^{-2}_0)\right)^{-1}&=\rho(g)
    \end{align}
    so the map $S\overline{S}\rho(a^{-2}_0)\in \operatorname{End}_{\Rep(G_0)}(V)$ is an isomorphism of the representation $\rho$, and thus by the classical Schur's lemma, see for example lemma 1.7 in \cite{fultonharris}, we have
    \begin{equation}\label{schurconstnt}
        S\overline{S}\rho(a^{-2}_0)=c\operatorname{I}
    \end{equation}  
    with $c\in \mathbb{C}^*$. Now observe that replacing $g=a^{-2}_0$ in $\rho'(g)=S^{-1}\rho(g)S$ we get
\begin{align}
\rho'(a^{-2}_0)=S^{-1}\rho(a^{-2}_0)S
\end{align}
    and using the definition of $\rho'$ it becomes
    \begin{equation}\label{schuraux}
        \overline{\rho(a^{-2}_0)}=S^{-1}\rho(a^{-2}_0)S.
    \end{equation}
    
    Using these last equalities we have 
    \begin{align}
        c\operatorname{I}=&S\overline{S}\rho(a^{-2}_0) \hspace{0.5cm} \text{ Eqn.}\eqref{schurconstnt}\\ = & S\overline{S} S\overline{\rho(a^{-2}_0)} S^{-1} \hspace{0.5cm} \text{ Eqn.} \eqref{schuraux} \\ =& S\overline{c}S^{-1} \hspace{0.5cm}\text{ Eqn.}\eqref{schurconstnt} \\ =& \overline{c}\operatorname{I} 
    \end{align}
    and thus $c\in \mathbb{R}^*$. Hence there are two possibilities, either $c>0$ or $c<0$. Now we can define a new morphism by rescaling $S$ by a real number
\begin{align}  
T:=\begin{cases}
        \frac{1}{\sqrt{c}}S & \text{ if } c>0\\ \frac{1}{\sqrt{-c}}S & \text{ if } c<0.
    \end{cases}
\end{align}
    Clearly we can can rewrite Eqn. \eqref{schurconstnt} as $ T\overline{T}=\pm \rho(a^{2}_0)$.
\end{proof}

Now we state  Schur's lemma in this context.

\begin{theorem}[Schur's lemma]\label{Schurfreeb} 

Let $V,W$ be irreducible representations of the finite magnetic group $(\mg{G},\phi)$. If $V$ and $W$ are non isomorphic, then $\Hom_{\MRep[]{\mg{G}}}(V,W)=0$. If $V\cong W$, then there are three possibilities for the restriction of $V$ to a representation of the core group $G_0$:
\begin{itemize}
    \item $\End_{\MRep[]{\mg{G}}}(V)\cong \mathbb{R}$, and $\res(V)=V$ is an irreducible representation of $G_0$. 
    \item $\End_{\MRep[]{\mg{G}}}(V)\cong \mathbb{C}$, and $\res(V)\cong V_1\oplus V_2$ where $V_1,V_2$ are non isomorphic irreducible representations of $G_0$ (although they are complex conjugate of each other). 
    \item $\End_{\MRep[]{\mg{G}}}(V)\cong \mathbb{H}$, and  $\res(V)\cong V_1\oplus V_1$ where $V_1$ is an irreducible representation of $G_0$.
\end{itemize}

\end{theorem}

\begin{proof}
    Suppose $V$ and $W$ are non isomorphic and let $T:V\longrightarrow W$ be a morphism of these representations. Then $\ker(T)=V$ since the kernel and the image of $T$ are subrepresentations of $V$ and $W$, respectively. Therefore $T=0$.
    
    If $\res(V)$ is an irreducible representation of $G_0$, then 
    \begin{align}
    \End_{\MRep[]{\mg{G}}}(V)\subset \End_{\Rep(G_0)}(V)\cong \mathbb{C}.
     \end{align}
    The previous isomorphism is just the classical Schur's lemma. An endomorphism $A\in \End_{\MRep[]{\mg{G}}}(V)$ must be of the form $A=z\cdot \id_V$ with $z\in \mathbb{C}$ and satisfies
    \begin{align}
        Aa_0(v)&=a_0A(v) \\
        z\cdot \id_V a_0(v)&= a_0\cdot z\cdot \id_V(v)\\
        z\cdot a_0(v)&=\overline{z}a_0(v)
    \end{align}
    so $z\in \mathbb{R}$ and thus $\End_{\MRep[]{\mg{G}}}(V)\cong \mathbb{R}$.
    
    On the other hand, if $\res(V)$ is reducible, pick any irreducible representation $V_1$ of $G_0$ in the decomposition of $\res(V)$, and consider the induced representation $\Ind{G_0} V_1=V_1\oplus a_0V_1$. By the Frobenius reciprocity Thm. \ref{frobrec} we have a natural non zero morphism, induced by the inclusion $\iota: V_1\longrightarrow \res (V)$
    \begin{align}
   \nonumber     \Ind{G_0}V_1=V_1\oplus a_0V_1& \longrightarrow V\\ (v,a_0w)& \longmapsto v+a_0w
    \end{align}
    The image of this morphism is a subrepresentation of the irreducible representation $V$, so $V_1\oplus a_0V_1$ is isomorphic to $V$. 
    Here we have two possibilities.
    
    Case 1. $V_1\cong a_0V_1$. By Lem. \ref{signodeT}, there exist an isomorphism $T:a_0V_1\longrightarrow V_1$ such that $T\overline{T}=-a^2_0$. If we define the morphism of representations 
    \begin{align}
    \begin{pmatrix}
        \id_{V_1}&0\\0&T
    \end{pmatrix}: V_1\oplus a_0V_1\longrightarrow V_1\oplus V_1,
    \end{align}
    then we can transform the induced representation of $(\mg{G},\phi)$ on 
    $V_1\oplus\,a_0V_1$ into a representation on $V_1\oplus V_1$ given by the formulas
    \begin{itemize}
        \item if $g\in G_0$
        \begin{align}
        \begin{pmatrix}
        \rho(g)&0\\0&\rho(g)
    \end{pmatrix}= \begin{pmatrix}
        \id_{V_1}&0\\0&T
    \end{pmatrix} \begin{pmatrix}
        \rho(g)&0\\0&\rho'(g)
    \end{pmatrix} \begin{pmatrix}
        \id_{V_1}&0\\0&T^{-1}
    \end{pmatrix}
    \end{align}
    \item if $g\in \mg{G}\backslash G_0$, then $g=a_0g'$ for some $g'\in G_0$ so we only need to specify the action of $a_0$
    \begin{align}
    \begin{pmatrix}
        0&-T\\T&0
    \end{pmatrix}\mathbb{K}=\begin{pmatrix}
        \id_{V_1}&0\\0&T
    \end{pmatrix}\begin{pmatrix}
        0&\rho(a^2_0)\\\id_{V_1}&0
    \end{pmatrix}\mathbb{K} \begin{pmatrix}
        \id_{V_1}&0\\0&T^{-1}
    \end{pmatrix}.
    \end{align}
        
    \end{itemize}
    
    Now, we have the following inclusion
    \begin{align}
    \End_{\MRep[]{\mg{G}}}(V)\subset \End_{\Rep(G_0)}(V_1\oplus V_1)\cong \mathbb{C}^4
    \end{align}
    An endomorphism $A\in\End_{\MRep[]{\mg{G}}}(V)$ must be of the form $A=\begin{pmatrix}
        \lambda \id_{V_1} & \alpha \id_{V_1} \\ \gamma \id_{V_1} & \theta \id_{V_1}
    \end{pmatrix}$. The morphism $A$ satisfies
    \begin{align}
        Aa_0\begin{pmatrix}
            v_1\\v_2
        \end{pmatrix}&=a_0A\begin{pmatrix}
            v_1\\v_2
        \end{pmatrix}\\
        A\begin{pmatrix}
            -T\mathbb{K}(v_2)\\ T\mathbb{K}(v_1)
        \end{pmatrix} &= a_0\begin{pmatrix}
            \lambda v_1+\alpha v_2\\ \gamma v_1+\theta v_2
        \end{pmatrix}\\
        \begin{pmatrix}
            -\lambda T\mathbb{K}(v_2)+\alpha T\mathbb{K}(v_1)\\-\gamma T\mathbb{K}(v_2)+\theta T\mathbb{K}(v_1)
        \end{pmatrix}&= \begin{pmatrix}
           -\overline{\gamma}T\mathbb{K}(v_1)-\overline{\theta}T\mathbb{K}(v_2) \\ \overline{\lambda}T\mathbb{K}(v_1)+\overline{\alpha}T\mathbb{K}(v_2)  
        \end{pmatrix}
    \end{align}
    So we have $\lambda=\overline{\theta}$ and $\alpha=-\overline{\gamma}$ and hence
    \begin{align}
    \End_{\MRep[]{\mg{G}}}(V)=\left\{ \begin{pmatrix}
        \lambda& \alpha \\ -\overline{\alpha}& \overline{\lambda}
    \end{pmatrix}\, | \, \lambda,\alpha\in \mathbb{C}\right\}\cong \mathbb{H}
    \end{align}
    
    Case 2. $V_1\not\cong a_0V_1$, so
    \begin{align}
    \End_{\MRep[]{\mg{G}}}(V)\subset \End_{\Rep(G_0)}(V_1\oplus a_0V_1)\cong \mathbb{C}^2
    \end{align}
    An endomorphism $A\in \End_{\MRep[]{\mg{G}}}(V)$ must be of the form $A=\begin{pmatrix}
        \lambda\id&0\\0& \alpha\id
    \end{pmatrix}$ and it satisfies
    \begin{align}
        Aa_0\begin{pmatrix}
            v\\a_0w
        \end{pmatrix} &= a_0A\begin{pmatrix}
            v\\a_0w
        \end{pmatrix} \\ A\begin{pmatrix}
            a^2_0w\\a_0v
        \end{pmatrix}&=a_0\begin{pmatrix}
            \lambda v\\ \alpha a_0w
        \end{pmatrix} \\ \begin{pmatrix}
            \lambda a^2_0w\\ \alpha a_0v
        \end{pmatrix}&= \begin{pmatrix}
            \overline{\alpha}a^2_0w\\ \overline{\lambda}a_0v
        \end{pmatrix}
    \end{align}
    so $\lambda=\overline{\alpha}$ and hence
    \begin{align}
    \End_{\MRep[]{\mg{G}}}(V)=\left\{ \begin{pmatrix}
        \lambda&0\\0&\overline{\lambda}
    \end{pmatrix} \,|\, \lambda\in\mathbb{C}\right\}\cong \mathbb{C}
    \end{align}
\end{proof}

On the other hand, if we start with an irreducible representation $V$ of $G_0$, then Wigner \cite{wigner} showed us a procedure to produce irreducible representations of the magnetic group of $(\mg{G},\phi)$.

\begin{theorem}[Wigner]\label{wigner}

    If $(\mg{G},\phi)$ is a finite magnetic group and $\rho: G_0\longrightarrow \GL(\mathbb{C}^n)$ an irreducible representation of the core group $G_0$, then we can construct an irreducible representation of $(\mg{G},\phi)$, in the following way. Let $\rho': G_0\longrightarrow \GL(a_0\mathbb{C}^n)$ be the conjugate representation of $\rho$, 
    \begin{itemize}
        \item \textbf{Real type}. If $\rho$ and $\rho'$ are isomorphic representations $\rho'(g)=T^{-1} \rho(g) T$ with $T\overline{T}=\rho(a^2_0)$, then for $v\in\mathbb{C}^n$ the formula 
        \begin{align}
        g\cdot v:=\begin{cases}
            \rho(g)v & g\in G_0\\
            \rho(ga^{-1}_0)T\mathbb{K}(v) & g\in \mg{G} \backslash G_0
        \end{cases}
        \end{align}
        defines an irreducible representation of $(\mg{G},\phi)$.

        \item \textbf{Complex type}. If $\rho$ and $\rho'$ are non isomorphic representations, then for $v,w \in \mathbb{C}^n$ the formula 
        \begin{align}
        g\cdot (v,w):=\begin{cases}
            \begin{pmatrix}
                \rho(g) &  0\\ 0& \rho'(g)
            \end{pmatrix} \begin{pmatrix}
                v\\w
            \end{pmatrix}& g\in G_0\\
            \begin{pmatrix}
                0&\rho(ga_0)\\ \rho'(ga^{-1}_0)&0
            \end{pmatrix}\mathbb{K} \begin{pmatrix}
                v\\w
            \end{pmatrix}& g\in \mg{G}  \backslash G_0
        \end{cases}
        \end{align}
        defines an \textbf{irreducible representation of} $(\mg{G},\phi)$.
        
        \item \textbf{Quaternionic type}. If $\rho$ and $\rho'$ are isomorphic representations $\rho'(g)=T^{-1} \rho(g) T$ with $T\overline{T}=-\rho(a^2_0)$, then for $v,w\in \mathbb{C}^n$ the formula 
        \begin{align}
        g\cdot(v,w):=\begin{cases}
            \begin{pmatrix}
                \rho(g) &  0\\ 0& \rho(g)
            \end{pmatrix} \begin{pmatrix}
                v\\w
            \end{pmatrix} & g\in G_0\\
            \begin{pmatrix}
                0&\rho(ga^{-1}_0)T\\-\rho(ga^{-1}_0)T&0
            \end{pmatrix}\mathbb{K} \begin{pmatrix}
                v\\w
            \end{pmatrix} & g\in \mg{G} \backslash G_0
        \end{cases}
        \end{align}
        defines an \textbf{irreducible representation of} $(\mg{G},\phi)$.
    \end{itemize}
    Moreover, every irreducible representation of $(\mg{G},\phi)$ is constructed in this way.
\end{theorem}
The first proof of this theorem was given by Wigner in sections 4 and 5 of chapter 20 in \cite{wigner}. 
\begin{proof}
Consider the induced representation $\Ind{G_0}\rho$. First, suppose it is irreducible. We have $\res \Ind{G_0}\rho=\rho\oplus \rho'$, so by the Schurs's Lem. \ref{Schurfreeb} the representation $\rho$ can be of complex or quaternionic type. If $\rho$ is of complex type, then $\rho$ and $\rho'$ are non isomorphic representations and the action described at the end of Def. \ref{induction}, see Eqns. \eqref{formulaind0} and \eqref{formulaforind}, give us the formula of the theorem. If $\rho$ is of quaternionic type, then $\rho$ and $\rho'$ are isomorphic and the representation $\Ind{G_0}\rho$ take the form of described in this theorem by performing the change of basis given in the proof of Thm. \ref{Schurfreeb}.

If $\Ind{G_0}\rho$ is reducible, then $\rho$ is of real type and $\Ind{G_0}\rho$ contains to copies of the representation described in this theorem.

Finally, if $\rho$ is an irreducible representation of $(\mg{G},\phi)$, then by Thm. \ref{Schurfreeb} the representation  $\res\, \rho$ is an irreducible representation or a sum of two of irreducible representations  of the core group $G_0$. In any case, we see that we can take these pieces of $\res\, \rho$ and reconstruct $\rho$ as this theorem states.
\end{proof}

\begin{remark}
    Let $\rho$ be an irreducible representation of $(\mg{G},\phi)$. By Thm. \ref{wigner} $\rho$ comes from an 
irreducible representation of the core group $G_0$. We also say that $\rho$ is of \textbf{real}, \textbf{complex}, or \textbf{quaternion type} if it comes from an irreducible representation of the core of the respectively same type. 
\end{remark}

\begin{corollary}
    For any representation $V$ of a finite magnetic group $(\mg{G},\phi)$, there is a decomposition
    \begin{align}
    V=\left(A_1^{a_1}\oplus\ldots\oplus A_m^{a_m}\right) \oplus \left(B_1^{b_1}\oplus\ldots \oplus B_n^{b_n}\right)\oplus\left( C_1^{c_1}\oplus \ldots \oplus C_l^{c_l}\right)
    \end{align}
    where $A_i,B_j,C_k$ are non isomorphic irreducible representations of real, complex and quaternionic type, respectively. 
\end{corollary}
If we denote by $\MRF[]{G}{\mathbb{F}}$ the free abelian group generated by the irreducible representations of real, complex, quaternionic type for $\mathbb{F}=\mathbb{R},\mathbb{C},\mathbb{H}$, respectively, then the representation group of $(\mg{G},\phi)$ splits as

\begin{equation}\label{splitring}
    \MRR[]{\mg{G}}\cong \MRF[]{\mg{G}}{\mathbb{R}}\oplus \MRF[]{\mg{G}}{\mathbb{C}} \oplus \MRF[]{\mg{G}}{\mathbb{H}}
\end{equation}

and by the Thm. \ref{Schurfreeb}
\begin{align}
 \nonumber   \End_{\MRep[]{\mg{G}}}(V)&\cong \bigoplus_i\End_{\MRep[]{\mg{G}}}(A_i)^{a_i}\oplus \bigoplus_j\End_{\MRep[]{\mg{G}}}\textbf{}(B_j)^{b_j} \oplus \bigoplus_k\End_{\MRep[]{\mg{G}}}(C_k)^{c_k}\\ &\cong \bigoplus_i \mathbb{R}^{a_i}\oplus \bigoplus_j \mathbb{C}^{b_j} \oplus \bigoplus_k \mathbb{H}^{c_k}.
\end{align}

Given an irreducible representation of the core $\rho: G_0 \longrightarrow \GL(n)$, then there is a procedure to find out which type of irreducible representation of $(\mg{G},\phi)$ will produce. This is called the Schur-Frobenius indicator. Denote by $\chi_{\rho}$ the character of $\rho$ and fix an element $a_0\in \mg{G}\backslash G_0$.
\begin{definition} \label{Schurindicator}
    The \textbf{Schur-Frobenius indicator} of $\rho$ with respect to $\phi$ is 
    \begin{align}
        \SF_{\phi}(\rho)=\sum_{g\in a_0G_0} \chi_{\rho}(g^2)
    \end{align}
\end{definition}

Notice this number is well defined since for every $g\in a_0G_0$, $g^2\in G_0$.

\begin{theorem}
    Let $(\mg{G},\phi)$ be a finite magnetic group and $\rho:G_0\longrightarrow \GL(n)$ an irreducible representation of $G_0$. 
    \begin{itemize}
        \item If $\SF_{\phi}(\rho)=|G_0|$, then the irreducible representation of $(\mg{G},\phi)$ associated to $\rho$ is of real type.
        \item If $\SF_{\phi}(\rho)=0$, then the irreducible representation of $(\mg{G},\phi)$ associated to $\rho$ is of complex type.
        \item If $\SF_{\phi}(\rho)=-|G_0|$, then the irreducible representation of $(\mg{G},\phi)$ associated to $\rho$ is of quaternionic type.
    \end{itemize}
\end{theorem}
\begin{proof}
    We refer the reader to section 7 of \cite{SFindicator} for the proof of the result, which is essentially based on the orthogonally relations between characters.
\end{proof}

\begin{lemma}\label{homtensor}
    For $\mathbb{F}\cong\mathbb{R}$, $\mathbb{C}$ or $\mathbb{H}$, let $V,W\in \IrrepF[]{\mg{G}}{\mathbb{F}}$ and $E$ a finite $\mathbb{F}$-vector space. Then we have the following isomorphism of $\mathbb{F}$-vector spaces
    \begin{align}
    \Hom_{\MRep[]{\mg{G}}}(V,W\otimes_{\mathbb{F}} E)\cong \Hom_{\MRep[]{\mg{G}}}(V,W)\otimes_{\mathbb{F}} E
    \end{align}
\end{lemma}
\noindent
\begin{proof}
\begin{align}
\Hom_{\MRep[]{\mg{G}}}(V,W\otimes_{\mathbb{F}}E)\cong& \Hom_{\MRep[]{\mg{G}}}(V,W\otimes_{\mathbb{F}} \mathbb{F}^n)\\ \cong& \Hom_{\MRep[]{\mg{G}}}(V,W^n) \\ \cong & \Hom_{\MRep[]{\mg{G}}}(V,W)^n \\ \cong & \Hom_{\MRep[]{\mg{G}}}(V,W)\otimes_{\mathbb{F}} \mathbb{F}^n \\ \cong & \Hom_{\MRep[]{\mg{G}}}(V,W)\otimes_{\mathbb{F}} E
    \end{align}
\end{proof}

%%%%%%%%%%%%%%%%%%%%%%%%%%%%%%%

%%%%%%%%%%%%%%%%%%%%%%%%%%%%%%%

%%%%%%%%%%%%%%%%%%%%%%%%%%%%%%%%%%%%%%

%%%%%%%%%%%%%%%%%%%%%%%%%%%%%%%%%%%%%%

\section{Magnetic Equivariant K-theory}

In this chapter we will develop the $K$-theory for spaces with actions of magnetic groups. This is a generalization of Atiyah's Real $K$-theory \cite{atiyahreal}, Segal's equivariant $K$-theory \cite{atiyahsegalequivariant} and some twisted $K$-theories \cite{twisted}. This is not the first time where such a notion is introduced. For instance, M. Karoubi was the first to define it in \cite{10.1007/BFb0059024}. More recently, it had been articulated with the study of topological phases of matter \cite{twistedmat}, \cite{Gomi2017FreedMooreK}. In these works D. Freed, G. Moore and K. Gomi defined the so called Freed-Moore $K$-theory, that essentially contains the definition we will state next. This particular K-theory has been further developed to provide
tools for their explicit calculation \cite{PhysRevB.95.235425}, \cite{PhysRevB.106.165103}.

We carry out several computations, mainly for tori equipped with actions of the magnetic point groups studied in the previous chapter; in particular we compute the coefficients of this $K$-theory and, as one can expect, the representations of the magnetic point groups will appear. The way we compute these groups is via Clifford modules. These results can be potentially used to detect protected topological phases of matter, and in they near future, they could be used to produce periodic tables for topological insulators and superconductors generalizing the work of Kitaev \cite{kitaev}.

\subsection{Magnetic equivariant vector bundles }

For a magnetic group $(G, \phi)$ the category of $(G,\phi)$-spaces will simply be the category of $G$-spaces.
Some explicit choices of magnetic groups have been important in the literature, let us see some of them.

A $(\mathbb{Z}/2,\id)$-space $X$ is the same as a Real space in the sense of Atiyah \cite{atiyahreal}, i.e. a space with involution given by the action of the non trivial element of $\mathbb{Z}/2$. Recall, an \textit{involution} on $X$ is a homeomorphism $\tau:X\longrightarrow X$ such that $\tau^2=\id$.

    A $(G_0\rtimes \mathbb{Z}/2,\pi_2)$-space $X$ is the same as a Real $G_0$-equivariant space in the sense of Atiyah and Segal \cite{atiyahsegalequivariant}, that is, a $G_0$-space $X$ with an involution $\tau$ such that $\tau(g\cdot x)= \tau(g)\cdot \tau(x)$ for $g\in G_0$ and $x\in X$ where $\tau: G_0\longrightarrow G_0$ denotes the action of the non trivial element of $\mathbb{Z}/2$ on $G_0$.

    The following spaces are $(\mg{G},\phi)$-spaces via the homomorphism $\mg{G} \overset{\phi}{\longrightarrow} \mathbb{Z}/2$, where the non trivial element of $\mathbb{Z}/2$ acts by complex conjugation.
    \begin{itemize}
    \item $\mathbb{R}^{p,q}=\mathbb{R}^q\oplus i\mathbb{R}^p$ with the involution by complex conjugation $x+iy\longmapsto x-iy$.
    \item $D^{p,q}$ unit ball in $\mathbb{R}^{p,q}$ (with the euclidean norm on $\mathbb{R}^{p+q}$).
    \item $S^{p,q}$ unit sphere in $\mathbb{R}^{p,q}$.
\end{itemize}
The interchange on the super indices position is due to Atiyah's convention in \cite{atiyahreal}. 

    Let $V$ be a complex-vector space, then 
    \begin{align}
    \PS(V)=(V\backslash \{0\})/\mathbb{C}^*
    \end{align}
    is the projective space over $V$. If $V$ is a representation of $(\mg{G},\phi)$, then $\PS(V)$ is a $(\mg{G},\phi)$-space with the action
    $g\cdot [v]:=[gv]$.
    This action is well defined because
       $ g(\lambda v) = \mathbb{K}^{\phi(g)}(\lambda)gv$ and therefore 
     $g[\lambda v]=g[v]$.

The vector bundles relevant to our analysis in this work are described below.

\begin{definition}
    A \textbf{magnetic $(\mg{G},\phi)$-equivariant vector bundle}  is a \textit{complex vector bundle} $E \overset{p}{\longrightarrow} X$ over a compact $G$-space $X$ such that
    \begin{itemize}
        \item the total space $E$ is a $G$-space and the projection $p$ is $G$-equivariant.
        \item For $g\in \mg{G}$ and $x \in X$ the fiberwise map $g: E_x\longrightarrow E_{g\cdot x}$ is complex linear if $\phi(g)=0$ and complex antilinear if $\phi(g)=1$, for all $x\in X$.
    \end{itemize}
    \textbf{Morphisms} of magnetic $(\mg{G},\phi)$-equivariant vector bundles are morphism of complex vector bundles $f:E\longrightarrow F$ commuting  with the action of $G$.
    
\end{definition}

Denote by $\MVec_{(\mg{G},\phi)}(X)$ the \textbf{category of magnetic $(\mg{G},\phi)$-equivariant vector bundles}.
 Denote the vector space of morphisms between $E$ and $F$ by  $\Hom_{(\mg{G},\phi)}(E,F)$.

A \textbf{section} of a $(\mg{G},\phi)$-vector bundle is a $G$-equivariant section $s: X\longrightarrow E$. Denote the vector space of such sections as $\Gamma_{G}(E)$. 

\textbf{Construction of magnetic $(G,\phi)$-equivariant vector bundles.}\\
Let $E,F$ be two $(\mg{G},\phi)$-vector bundles over $X$.
\begin{itemize}
    \item The \textbf{hom bundle} $\Hom(E,F)$ with fiber over $x$ the complex vector space of linear transformations  $\Hom(E_x,F_x)$ between $E_x$ and $F_x$, and action of $(G,\phi)$ given by
    \[(g\cdot f)(v)=gf(g^{-1}v)\]
    for $g\in \mg{G}$, $f\in \Hom(E,F)$ and $v\in E_{gx}$.
    \item The \textbf{tensor bundle} $E\otimes F$ with fiber $E_x\otimes F_x$ at $x\in X$ and action of $(\mg{G},\phi)$ given by
    \[g\cdot (v\otimes w)=gv\otimes gw\]
    for $g\in \mg{G}$, $v\in E_x$ and $w\in F_x$.
    \item The \textbf{pull-back} $f^* E$ for $f: Y \longrightarrow X$  a $G$-equivariant map of $G$-spaces with the induced action of $(\mg{G},\phi)$.
    \item If $X=X_1\cup X_2$ is the union of two compact $(\mg{G},\phi)$-subspaces and $A=X_1\cap X_2$ is a compact $(\mg{G},\phi)$-subspace, the \textbf{clutching} of two $(\mg{G},\phi)$-vector bundles $E_1\longrightarrow X_1$ and $E_2\longrightarrow X_2$ via a $(\mg{G},\phi)$-isomorphism $\alpha: E_1|_A\longrightarrow E_2|_A$ is given by
    \[E_1\cup_{\alpha} E_2:=E_1\sqcup E_2\big/ \left(e\sim \alpha(e) \text{ for all }e\in (E_1)_x \text{ with }x\in A\right) \]
    
\end{itemize}
With all the previous notation we can state the following classical fact:
\begin{lemma}\label{tensorandsectios}
    Let $E,F$ be two $(\mg{G},\phi)$-vector bundle over $X$. Then we have a natural isomorphism 
    \begin{align*}
        \Hom_{(\mg{G},\phi)}(E,F) & \overset{\cong}{\longrightarrow} \Gamma_{G}(\Hom(E,F)) \\ f & \longmapsto s_f
    \end{align*}
    between morphisms of $(\mg{G},\phi)$-vector bundles and $G$-equivariant sections of the bundle $\Hom(E,F)$ over $X$. Here $s_f(x)(v)=f(v)$ for all $v\in E_x$.
\end{lemma}

\begin{example}\label{trivialbundles}
    Let $n$ be a non negative integer and $X$ a $(\mg{G},\phi)$-space. The \textbf{trivial $(\mg{G},\phi)$-vector bundle over $X$ of dimension $n$} is 
    \[\xymatrix{
    \textbf{n}:=X\times \mathbb{C}^n \ar[d]^{\pi_1} \\ X
    }\]
    with the trivial action 
    \[g\cdot(x,z)=\left(gx,\mathbb{K}^{\phi(g)}(z)\right)\]
    where $g\in \mg{G}$, $x\in X$ and $z\in \mathbb{C}^n$. More generally, let $V$ be a representation of $(\mg{G},\phi)$, then $\textbf{V}:=X\times V \overset{\pi_1}{\longrightarrow} X$ is the \textbf{trivial bundle} over $X$ with fiber $V$. 
\end{example}

\begin{definition}[Stable equivalence]
    Two $(\mg{G},\phi)$-vector bundles $E,F$ over a $G$-space $X$ are \textbf{stably equivalent} $E\sim_s E'$ if there exist trivial $(\mg{G},\phi)$-vector bundles $\textbf{V}, \textbf{V}'$ on $X$ such that $E\oplus \textbf{V}\cong E'\oplus \textbf{V}'$.  
\end{definition}

\begin{remark}
    Notice the set of stable equivalence classes of $(\mg{G},\phi)$-vector bundles is an abelian semigroup under $\oplus$, whose identity is given by the class of the trivial bundle. However, it is not clear this is an abelian group since we do not yet know if for every $(\mg{G},\phi)$-vector bundle $E$ there exist other $F$ such that $E\oplus F\cong \textbf{V}$. In fact, we will show this at the end of this section in Prop. \ref{complement}.
\end{remark}

\subsection{Definition of Magnetic Equivariant K-theory}

\begin{definition} \label{def magnetic K-theory}
    The \textbf{magnetic equivariant $K$-theory}  $\K_{G}(X)$ of the $G$-space $X$ is the Grothendieck group of the category $\MVec_{(\mg{G},\phi)}(X)$ of magnetic $(\mg{G},\phi)$-equivariant vector bundles over $X$. Its elements are formal differences $E_0-E_1$ of magnetic $(\mg{G},\phi)$-equiviariant vector bundles, modulo the equivalence relation $E_0-E_1\sim E'_0-E'_1$ if and only if $E_0\oplus E'_1\oplus F\cong E'_0\oplus E_1\oplus F$ for some magnetic $(\mg{G},\phi)$-equiviariant vector bundle $F$ on $X$.
\end{definition}

\begin{remark}
From now on, we will remove the homomorphism $\phi$ from the notation. Bold face functors
receive magnetic groups and $\phi$ will be assumed to be part of the structure of the magnetic groups.  
Hence we will denote by
 \begin{align}
 \mathbf{Rep}(G), \ \mathbf{R}(G),\ \mathbf{R}(G,\mathbb{F}), \ \mathbf{Irrep}(G,\mathbb{F}),\ \MVec_{\mg{G}}(X) \ \mathrm{and} \ \K_{\mg{G}}(X)
 \end{align}
 the category of magnetic $(G,\phi)$-representations, its Grothendieck group, the subgroup
 generated by irreducible representations of type $\mathbb{F}$ (for $\mathbb{F} \in \{\mathbb{R}, \mathbb{C}, \mathbb{H} \}$), the set of isomorphism classes of irreducible representations of type $\mathbb{F}$, the category of
 magnetic $(G,\phi)$-equivariant vector bundles and its K-theory respectively.
 \end{remark}

\begin{example}
    If the magnetic group is $(\mathbb{Z}/2,\id)$, then a $(\mathbb{Z}/2,\id)$-vector bundle is a Real vector bundle in the sense of Atiyah \cite{atiyahreal}. We have then that magnetic $(\mathbb{Z}/2,\id)$-equivariant K-theory is equivalent to Atiyah's Real K-theory,
    \begin{align}
    \K_{\mathbb{Z}/2}(X) = K \mathbb{R}(X),
    \end{align}
    and therefore if $X$ is a trivial $\mathbb{Z}/2$ space, then $\K_{\mathbb{Z}/2}(X) =KO(X)$.
\end{example}

\begin{example}\label{real and quaternionic bundles}
    Consider magnetic group 
    \begin{align}
    1\longrightarrow G_0=\mathbb{Z}/2\longrightarrow \mg{G}=\varmathbb{Z}/4 \overset{\operatorname{mod}2}{\longrightarrow} \mathbb{Z}/2\longrightarrow 1
    \end{align} together with a space $X$ with an involution. Take the induced
    action of $\mathbb{Z}/4$ on $X$ given by the homomorphism $\phi:= \operatorname{mod}2$.
    
    Since the core $G_0=\mathbb{Z}/2$ acts trivially on $X$, any magnetic $(\mathbb{Z}/4,\phi)$-equivariant
    vector bundle $E \to X$ can be split as $E\cong E_{1} \oplus E_{-1}$ where $G_0$
    acts trivially on $E_{1}$ and with the sign representation on $E_{-1}$.
    If $\tau $ denotes the generator of $\mathbb{Z}/4$, then $\tau$ acts complex antilinearly
    on $E$, with $\tau^2=1$ on $E_{1}$ and with $\tau^2=-1$ on $E_{-1}$. Therefore
    $E_{1}$ is a Real vector bundle in the sense of Atiyah \cite{atiyahreal} and $E_{-1}$ is a quaternionic bundle \cite{hartshornestable}.
    Hence for $X$ a $\mathbb{Z}/2$-space we have the canonical isomorphism
     \begin{align} \label{K-theory Z/4}
    \K_{\mathbb{Z}/4}(X) \cong K \mathbb{R}(X) \oplus K \mathbb{H}(X),
    \end{align}
    that in the case that $X$ is a trivial $\mathbb{Z}/2$ space becomes
    \begin{align}
    \K_{\mathbb{Z}/4}(X) \cong KO(X) \oplus KSp(X)
    \end{align}
    where $Ksp$ is symplect K-theory as defined by Dupont \cite{dupontsymplectic}.
\end{example}

Let us denote by $KU$ the usual complex $K$-theory.

\begin{example}\label{kteoriadelpunto}
    Suppose we take the space $X$ to be a point $*$ and $G$ finite. Then any magnetic $(G,\phi)$-equivariant vector bundle over a point is nothing else as a representation of the magnetic group $(G,\phi)$. Then we get the canonical isomorphism of groups
    \begin{align}
\K_G(*) \cong     \mathbf{R}(G).
    \end{align}
    \end{example}

    This generalizes what happens in the equivariant as shown by Segal \cite{segal}.

Now we will carry out some computations which are based on the constructions of some induced bundles.
\begin{definition}
    Let $(\mg{G},\phi)$ be a magnetic group, $H\leq \mg{G}$ a subgroup and $V$ a complex representation of the group $H$. The \textbf{induced bundle} is the $(\mg{G},\phi)$-vector bundle over the discrete space $\mg{G}/H$
    \begin{align}
    p:\mg{G}\times_H V \longrightarrow \mg{G}/H
    \end{align}
    where 
    \begin{itemize}
        \item $\mg{G}\times_H V=\mg{G}\times V/ \left( (gh^{-1},hv)\sim (g,v) \text{ for }h\in H\right)$ with the quotient topology,
        \item the projection $p$ is given by $p([g,v])=gH$,
        \item the vector space structure on each fiber is
        \begin{align}
            [g,v]+[g,w]&=[g,v+w], \ \ \ \ 
            z[g,v]=[g,\K^{\phi(g)}(z)v]
        \end{align}
        for $g\in \mg{G}$, $v,w\in V$ and $z\in \mathbb{C}$.
        \item and the action of $(\mg{G},\phi)$ on $G \times_H V$ is
        $l\cdot[g,v]=[lg,v]$.
    \end{itemize}
    Note that the action of $(\mg{G},\phi)$ on $G\times_H V$ is linear or antilinear depending on $\phi$:
     \begin{align*}
         l\cdot (z[g,v])=&  [lg,\mathbb{K}^{\phi(g)}(z)v]\\ =& \mathbb{K}^{\phi(l)}(z)[lg,v] \\ =& \mathbb{K}^{\phi(l)}(z)l[g,v].
     \end{align*}
\end{definition}

    In the particular case where $H=G_0$ and $a_0\in \mg{G}\backslash G_0$ is fixed, then we have an isomorphism 
    \begin{align*}
        G\times_{G_0} V & \longrightarrow V\sqcup a_0V \\
        [g,v]& \longmapsto \begin{cases}
            a_0 (a^{-1}_0gv) & \text{ if } g\not\in G_0\\ gv & \text{ if } g\in G_0
        \end{cases}
    \end{align*}
    where $a_0V=\{a_0 v\, | \, v\in V\}$ is the conjugated representation of $V$ introduced in Eqns. \eqref{conjugaterepresentation1} and \eqref{conjugaterepresentation2}.
\begin{lemma}
There is a natural isomorphism
\[\K_{\mg{G}}(S^{1,0})\cong KU_{G_0}(*)\cong \operatorname{R}(G_0).\]
\end{lemma}
\begin{proof}
    Let $p:E \longrightarrow \{\pm 1\}$ be a $(\mg{G},\phi)$-vector bundle with $S^{1,0}\cong\{\pm 1\}$ with the free $\mathbb{Z}/2$ and the $G$ action induced by $\phi$. Then the restriction $p^{-1}(\{1\})\longrightarrow \{1\}=*$ is a $G_0$-equivariant vector bundle.\\
    On the other hand, if $V\longrightarrow *$ is a $G_0$-equivariant vector bundle, then consider the induced bundle \[p:E=(\mg{G},\phi)\times_{G_0} V \longrightarrow \{\pm1\}\]
    These constructions are inverses one to the other. 
\end{proof}

We can generalize the previous lemma by taking any subgroup $H$ of $\mg{G}$.
\begin{lemma}\label{cocientes}
    Let $(\mg{G},\phi)$ be a finite magnetic group and $H$ a subgroup of $\mg{G}$. Then there is a canonical isomorphism
    \[\K_G(\mg{G}/H)\cong \begin{cases}
        KU_H(*) & \text{if } H\leq G_0\\
        \K_H(*) & \text{if } H\not\leq G_0
    \end{cases} \]
\end{lemma}
\begin{proof}
    The proof uses the same constructions as in the previous lemma. Let $p:E\longrightarrow \mg{G}/H$ be a $(\mg{G},\phi)$-vector bundle. Then the restriction $E|_{\{H\}}\longrightarrow \{H\}=*$ is an $H$ or $(H,\phi|_H)$-vector bundle over the point, depending on whether $H\leq G_0$ or $H\not\leq G_0$.
    
    Conversely, let $V\longrightarrow *$ be an $H$ or $(H,\phi|_H)$-vector bundle and consider the induced bundle
    \[p:E=G\times_H V \longrightarrow \mg{G}/H.\]
    These constructions are inverses one to the other.
\end{proof}

\begin{proposition}\label{decomK}
    Let  $(G,\phi)$ be a finite magnetic group acting trivially on $X$. Then we have a natural isomorphism of groups:
    \begin{align}
    \K_{G}(X)  \overset{\cong}{\longrightarrow} \left( \mathbf{R}(G, \mathbb{R}) \otimes KO(X) \right) \oplus  \left( \mathbf{R}(G, \mathbb{C}) \otimes KU(X) \right)  \oplus \left(  \mathbf{R}(G, \mathbb{H}) \otimes KSp(X) \right).
    \end{align}
\end{proposition}

\begin{proof}
Take a magnetic $(G,\phi)$-equivariant vector bundle $p:E\longrightarrow X$ and $V \in \mathbf{Irrep}(G,\mathbb{F})$. The bundle  
\begin{align}
  \operatorname{Hom}_{\MRep[]{\mg{G}}}(V,E) \to X
\end{align}
is a $\operatorname{End}_{\MRep[]{\mg{G}}}(V) \cong \mathbb{F}$ vector bundle and therefore the map
\begin{align}
 \nonumber \K_{G}(X)  &\overset{\cong}{\longrightarrow} \bigoplus_{\mathbb{F} \in \{\mathbb{R}, \mathbb{C},\mathbb{H}  \}} \mathbf{R}(G, \mathbb{F}) \otimes K\mathbb{F}(X)\\
 E & \longmapsto \bigoplus_{\mathbb{F} \in \{\mathbb{R}, \mathbb{C},\mathbb{H}  \}} \bigoplus_{V \in \mathbf{Irrep}(G, \mathbb{F})}  V \otimes \operatorname{Hom}_{\MRep[]{\mg{G}}}(V,E) 
\end{align}
provides the desired isomorphism. The inverse map takes a representation $V$ of type $\mathbb{F}$ and 
 a $\mathbb{F}$-vector bundle $E$ in $K\mathbb{F}(X)$ to $V \otimes_{\mathbb{F}}E$ in $\K_G(X)$.

\end{proof}

\subsection{Topological properties}

Now we prove some lemmas that are required to extend $\K_{\mg{G}}$ to a cohomology theory, these are extensions of the lemmas appearing in \cite{atiyahktheory}, \cite{atiyahreal} and \cite{segal}. For the rest of the section we assume
\begin{itemize}
    \item $X$ is a Hausdorff compact space.
    \item Magnetic point groups $(\mg{G},\phi)$ are finite. 
\end{itemize}

\begin{lemma}\label{veciequiv}
    Let $X$ be a $G$-space, $A\subset X$ a closed invariant subspace and $W$ an open neighborhood of $A$. Then there exist an open invariant set $U$ such that $A\subset U \subset W$.
\end{lemma}
\begin{proof}
    Define the open set \[U:= q^{-1}\left(X/G\,\backslash\,  q(X\backslash W)\right)\] where $q: X \longrightarrow X/G$ is the canonical projection. 
    First note that $q(A)\cap q(X\backslash W)=\emptyset$: Suppose $[x]\in q(A)\cap q(X\backslash W)=\emptyset$, then there exist $g,h\in \mg{G}$ such that $g\cdot x\in A$ and $h\cdot x\in X\backslash W$. So $h\cdot (hg^{-1})\cdot g\cdot x=\in A$ since $A$ is invariant. Then $h\cdot x\in A\cap X\backslash W$ which is a contradiction.
    So we have $q(A)\subset X/G \, \backslash \, q(X\backslash W)$ and thus $A=q^{-1}(q(A))\subset q^{-1}(X/G \backslash q(X\backslash W)=U$ since $A$ is invariant. $U\subset W$: suppose there exist $x\in 
    U$ such that $x\not\in W$, then $[x]\in p(U)=X/G\,\backslash \, q(X\backslash W)$ and $[x]\in q(X\backslash W)$ which is a contradiction. 
\end{proof}
This proof is based in the one given by R. Palais in 1.1.14 of \cite{classification_g_spaces}.
\begin{lemma}[Extension of sections]\label{extofsec}
If $E$ is a $(\mg{G},\phi)$-vector bundle on a $G$-space $X$, and $A$ is a closed $G$-subspace of $X$, then any section of the $(\mg{G},\phi)$-vector bundle $E|_A$ can be extended to a section of the $(\mg{G},\phi)$-vector bundle $E$. \end{lemma}
\begin{proof}
    Let $s:A \longrightarrow E|_A$ be such a section and extend it arbitrarily to $E$, as in 1.4.1 of \cite{atiyahktheory} or by using the Tietze extension theorem to extend it locally and then add this local extension to a global section via a partition of the unity. Finally, take the averages of $s$ over $\mg{G}$
    \[s(x)=\frac{1}{|\mg{G}|}\sum_{\mg{G}}(g\cdot s)(x) \]
    as in 1.1 of \cite{segal} or 1 of \cite{atiyahreal}.
\end{proof}

\begin{lemma}\label{extensionofiso}
    Let $E,F$ be two $(\mg{G},\phi)$-vector bundles on $X$, $A$ a closed $G$-subspace of $X$ and $f:E|_A \longrightarrow F|_A$ an isomorphism. Then there is a $G$-neighborhood of $A$ in $X$ and an isomorphism $f: E|_U\longrightarrow F|_U$ of $(\mg{G},\phi)$-vector bundles extending $f$.
\end{lemma}

\begin{proof}
    By Lem. \ref{tensorandsectios}, $f|_A$ can be understood as a section of the $(\mg{G},\phi)$-vector bundle $\Hom_{(\mg{G},\phi)}(E|_A,F|_A)=\Hom_{(\mg{G},\phi)}(E,F)|_A$ over $A$. Using Lem. \ref{extofsec}, extend this section to a section $f$ of $\Hom_{(\mg{G},\phi)}(E,F)$. So we have a morphism from $E$ to $F$. The subset $W$ of $X$ where $f$ is an isomorphism, is an open set which contains $A$ by the continuity of the determinant. By Lem. \ref{veciequiv} there exists a $G$-equivariant open neighborhood $U$ of $A$ contained in $W$, and of course, $f$ is still an isomorphism on $E|_U$.
\end{proof}

\begin{lemma}\label{invhom}
    Let $p:E\longrightarrow X$ be a $(\mg{G},\phi)$-vector bundle, $f_0,f_1: Y\longrightarrow X$ be two $G$-homotopic $G$-maps and $Y$ be a compact space. Then we have an isomorphism $f_0^* E\cong f_1^* E$ of $(\mg{G},\phi)$-vector bundles.
\end{lemma}
\begin{proof}
    Suppose $H: Y\times I\longrightarrow X$ is such an equivariant homotopy between $f_0$ and $f_1$. For any $t\in I$, we have an isomorphism of $(\mg{G},\phi)$-vector bundles 
    \begin{align}
    H^* E|_{Y\times \{t\}} \longrightarrow \pi_t^* H_t^* E|_{Y\times \{t\}}=H^*_tE 
    \end{align}
    over $Y\times \{t\}$, where $\pi_t: Y\times I \longrightarrow Y\times \{t\}$ is given by $\pi_t(y,s)=(y,t)$ and $H_t=H|_{y\times \{t\}}$. By Lem. \ref{extensionofiso} there is an open neighborhood of $Y\times \{t\}$, namely $Y\times \delta(t)$, where $H^* E$ and $\pi_1^* H_t^* E$ are isomorphic. If $r\in \delta(t)$, then we have a commutative diagram
    \begin{align}
    \xymatrix{ (H^*E)|_{Y\times \delta(t)} \ar[r]^{\cong}  & (\pi^*_{t}(H_t^*E))|_{Y\times \delta(t)} \\ H_r^*E=(H^*E|_{Y\times \delta(t)})|_{Y\times \{r\}} \ar@{^(->}[u] \ar@{-->}[r]^{\cong} & (\pi^*_{t}(H_t^*E)|_{Y\times \delta(t)})|_{Y\times \{r\}} \ar@{^(->}[u]}
    \end{align}
    The vector bundle $\pi^*_{t}H_t^*E|_{Y\times \delta(t)}|_{Y\times \{r\}}$ is the same vector bundle as $\theta_{t,r}^*H^*_t E$ where $\theta_{t,r}:Y\times \{r\}\longrightarrow Y\times \{t\}$ is given by $\theta_{t,r}(y,r)=(y,t)$ and also we have a natural isomorphism $\theta_{t,r}^*H^*_t E\cong H^*_t E$. Hence the isomorphism class of $H_t^*E$ is locally constant function on $t$. Since the interval $I$ is connected, this implies it is constant. This is the same argument given in Lem. 1.4.3 in \cite{atiyahktheory}.
\end{proof}

\begin{corollary}\label{trivial}
    
         If $f: X \longrightarrow Y$ is a $G$-homotopy equivalence, then the pull-back
         \begin{align}
         f^*: \MVec_{(\mg{G},\phi)}(Y)/\cong\, \longrightarrow \MVec_{(\mg{G},\phi)}(X)/\cong
         \end{align}
          is s bijection between the isomorphism classes of $(\mg{G},\phi)$-vector bundles over $Y$ and $X$.
        In particular,  if $X$ is $G$-contractible, then every $(G,\phi)$-vector bundle over $X$ is isomorphic to a trivial $(\mg{G},\phi)$-vector bundle.
\end{corollary}

Let $X=X_1\cup X_2$, $A=X_1\cap X_2$, all spaces being compact and $G$-invariant. Assume that $E_i$ is a $(\mg{G},\phi)$-vector bundles over $X_i$ and that $\alpha: E_1|_A\longrightarrow E_2|_A$ is an $(\mg{G},\phi)$-isomorphism. The \textbf{clutching} of $E_1$ and $E_2$ is the following $(\mg{G},\phi)$-vector bundle over $X$
\begin{align}
E_1\cup_{\alpha} E_2:=E_1\sqcup E_2/(v\sim \alpha(v) \text{ for all }v\in E_1).
\end{align}
\begin{corollary}\label{isoofclut}
    The isomorphism class of a clutching $(\mg{G},\phi)$-vector bundle $E_1\cup_{\alpha}E_2$ depends only on the homotopy class of the isomorphism $\alpha: E_1|_A \longrightarrow E_2|_A$.
\end{corollary}
\begin{proof}
    Suppose $\alpha_{t}: E_1|_{A}\longrightarrow E_2|_{A}$ is a homotopy of $(\mg{G},\phi)$-isomorphisms. Then we can construct an $(\mg{G},\phi)$-isomorphism of $(\mg{G},\phi)$-vector bundles over $A\times I$ as follows
    \begin{align}
     \nonumber   \alpha:(\pi_1^* E_1)|_{A\times I} & \longrightarrow (\pi_1^*E_2)|_{A\times I} \\ (a,t,v)& \longmapsto (a,t,\alpha_t(v)).
    \end{align}
    where $\pi_1: X\times I \longrightarrow X$ is the projection onto the first factor. So we can construct the clutching of $\pi_1^* E_1$ and $\pi_1^* E_2$
    \begin{align}
    \xymatrix{ \pi_1^* E_1\cup_{\alpha} \pi_1^* E_2 \ar[d] \\ X\times I.}
    \end{align}
    For $t\in I$, define the function 
    \begin{align}
     \nonumber   f_t: X&\longrightarrow X\times I \\ x&\longmapsto (x,t).
    \end{align}
    Note for $t\in I$ we have an $(\mg{G},\phi)$-isomorphism of $(\mg{G},\phi)$-vector bundles
    over $X$
    \begin{align}
    E_1\cup_{\alpha_t} E_2 \overset{\cong}{\longrightarrow} f^*_t(\pi_1^* E_1\cup_{\alpha} \pi_1^* E_2).
    \end{align}
    So apply Lem. \ref{invhom} to the $(\mg{G},\phi)$-vector bundle $\pi_1^* E_1\cup_{\alpha} \pi_1^* E_2$ and the $(\mg{G},\phi)$-homotopic functions $f_0,f_1$. 
\end{proof}
Let $(X,A)$ be a closed $G$-pair with $A$ $G$-contractible. If $E$ is a $(\mg{G},\phi)$-vector bundle over $X$, then by Cor. \ref{trivial} there exist a $(\mg{G},\phi)$-isomorphism $\alpha: E|_A \longrightarrow A\times V$ with $V$ a representation of $(\mg{G},\phi)$. We refer to $\alpha$ as a \textbf{trivialization} of $E$ over $A$. Let $\pi_2: A\times V \longrightarrow V$ be the projection onto the second factor and define an equivalence relation on $E|_A$ by
\begin{align}
e\sim e' \text{ if and only if } \pi_2(\alpha(e))=\pi_2(\alpha(e'))
\end{align}
and extend it by the identity on $E|_{X\backslash A}$. Denote by $E/\alpha$ the quotient of $E$ by this equivalence relation. Then $E/\alpha \longrightarrow X/A$ is a $(\mg{G},\phi)$-vector bundle, one just need to verify the local triviality around the point $A/A$ but this is an application of Lem. \ref{extensionofiso}. Moreover:
\begin{corollary}\label{homotopyinvoftriv}
    Let $(X,A)$ be a closed $G$-pair with $A$ $G$-contractible, 
    $E\longrightarrow X$ a $(\mg{G},\phi)$-vector bundle over $X$ and $\alpha: E|_A\longrightarrow A\times V$ a trivialization of $E$ over a closed $G$-subspace $A$. Then the isomorphism class of the collapsing construction $E/\alpha$ depends only on the homotopy class of $\alpha$.
\end{corollary}
\begin{proof}
    Similar to the proof of Cor. \ref{isoofclut}, a $G$-homotopy of between trivializations $\alpha_0$ and $\alpha_1$ induces an isomorphism $\beta: (E\times I)|_{A\times I} \longrightarrow A\times I\times V$ of $(\mg{G},\phi)$-vector bundles over $A\times I$. Consider the natural $G$-map
    \begin{align}
   \nonumber     f: (X/A)\times I &\longrightarrow (X\times I)/(A\times I)\\ ([x],t) & \longrightarrow [x,t]
    \end{align}
    This a $G$-homotopy between $f_0=f|_{(X/A)\times \{0\}}$ and $f_1=f|_{(X/A)\times \{1\}}$. For $i=0,1$, we have  $(\mg{G},\phi)$-isomorphisms of $(\mg{G},\phi)$-vector bundles over $(X/Y)\times I$
    \begin{align}
    E/\alpha_i \longrightarrow f^*\left((E\times I)/\beta\right)|_{(X/Y)\times \{i\}}.
    \end{align}
    So apply Lem. \ref{invhom} to the $(\mg{G},\phi)$-vector bundle $(E\times I)/\beta$ and the maps $f_0$ and $f_1$.
\end{proof}
For $X,Y$ two $G$-spaces, denote by $[X,Y]^{\mg{G}}$ the set of $G$-homotopy classes of $G$-maps.
\begin{lemma}\label{calculodepi0}
    Let $A$ be a $G$-contractible $G$-space, $V$ a representation of $(\mg{G},\phi)$ and suppose it  decomposes as $V\cong \left(A_1^{a_1}\oplus\ldots\oplus A_m^{a_m}\right) \oplus \left(B_1^{b_1}\oplus\ldots \oplus B_n^{b_n}\right)\oplus\left( C_1^{c_1}\oplus \ldots \oplus C_l^{c_l}\right)$
    where $A_i,B_j,C_k$ are non isomorphic irreducible representations of real, complex and quaternionic type, respectively. Then we have a bijection
    \begin{align}
    [A,\GL(V)]^{\mg{G}} \overset{\cong}{\longrightarrow} \Pi_i\{\pm 1\}^{a_i}.
    \end{align}
    where the action of $G$ on $\GL(V)$ is given by $(g\cdot T)(v)=g\cdot T(g^{-1}\cdot v)$ with $g\in \mg{G}$, $T\in \GL(V)$ and $v\in V$. The set $\Pi_i\{\pm 1\}^{a_i}$ is the discrete space of $(\sum_i a_i)$-tuples in $\{\pm 1\}$.
\end{lemma}
\begin{proof}\label{homotopytype}
    Suppose $*\in A$ is a fixed point of the group $G$. A $G$-homotopy $H:A\times I \longrightarrow A$, between the identity map $\id_A$ and a the constant map $A\longrightarrow \{*\}$, gives rise to a bijection 
    \begin{align}
    [A,\GL(V)]^{\mg{G}} \longrightarrow [*,\GL(V)]^{\mg{G}}.
    \end{align}
    Now we have the equality 
    \begin{align}
    [*,\GL(V)]^{\mg{G}}=\pi^{\mg{G}}_0(\operatorname{Iso}_{\MRep[]{\mg{G}}}(V))
    \end{align}
    where $\pi^{\mg{G}}_0(X)$ is the set of $G$-path components of $X$. Finally, by Schur's lemma in  Lem. \ref{Schurfreeb} we have the bijection
    \begin{align}
        \pi^{\mg{G}}_0 \left( \operatorname{Iso}_{\MRep[]{\mg{G}}}(V)\right) & \cong \pi^{\mg{G}}_0 \left(\bigoplus_i (\mathbb{R}^*)^{a_i}\oplus \bigoplus_j (\mathbb{C}^*)^{b_j} \oplus \bigoplus_k (\mathbb{H}^*)^{c_k} \right) \\
        & \cong \pi_0^{\mg{G}}\left(\bigoplus_i (\mathbb{R}^*)^{a_i}\right)\times \pi^{\mg{G}}_0\left(\bigoplus_j (\mathbb{C}^*)^{b_j}\right) \times \pi_0^{\mg{G}}\left(\bigoplus_k (\mathbb{H}^*)^{c_k}\right)\\ & \cong \Pi_i \{\pm 1\}^{a_i}.
    \end{align}
\end{proof}
\begin{remark}\label{convention}
    As we have seen, Lem. \ref{homotopytype} shows that the $G$-homotopy classes $[A,\GL(V)]^{\mg{G}}$ depend only on the irreducible representations of real type in the decomposition of $V$. However, we always can choose classes in $\mathbb{R}^+=\{x\in \mathbb{R}\, |\, x>0\}$ or in $\mathbb{R}^-=\{x\in \mathbb{R}\, |\, x<0\}$ because they always exist, and \textbf{by convention we choose the positive part $\mathbb{R}^+\subset \mathbb{R}^*$}.

    An example of this choice, and possibly the most important, is the following: Given a $(\mg{G},\phi)$-vector bundle $E$ over $X$, then there exist trivializations $\alpha: E|_A \longrightarrow A\times V$. Two of these trivializations $\alpha_0,\alpha_1:E_A\longrightarrow A\times V$ define a map $\alpha^{-1}_1\alpha_0: A\times V \longrightarrow A\times V$ which is of the form $(x,v)\longmapsto (x,\Psi_x(v))$. The map $\Psi:A\longrightarrow \GL(V)$ is $\mg{G}$-equivariant, so its class is in $[\Psi]\in [A,\GL(V)^{\mg{G}}]\cong \pi^{\mg{G}}_0(\operatorname{Iso}_{\MRep[]{\mg{G}}}(V))=\{\pm 1\}^{\oplus_i a_i}$. We choose $\alpha_0,\alpha_1$ such that $[\Psi]$ is represented by $(1,1,\ldots,1)\in\{\pm 1\}^{\oplus_i a_i}$. 
\end{remark}
\begin{corollary}\label{trivializationiso}
    Let $(X,A)$ be a $(\mg{G},\phi)$-pair with $A$ $(\mg{G},\phi)$-contractible. Then the canonical projection $f: X \longrightarrow X/A$ induces a bijection 
    \begin{align}
    \xymatrix{f^*: \MVec_{(\mg{G},\phi)}(X/A)/\cong \, \ar[r] & \MVec_{(\mg{G},\phi)}(X)/\cong.}
    \end{align}
\end{corollary}
\begin{proof}
    Given a $(\mg{G},\phi)$-vector bundle $E$ over $X$, choose a trivialization $\alpha:E|_A \longrightarrow A\times V$ and then consider the $(\mg{G},\phi)$-vector bundle $E/\alpha \longrightarrow X/A$. So we have map 
    \begin{align}
    \MVec_{(\mg{G},\phi)}(X)/ \longrightarrow \MVec_{(\mg{G},\phi)}(X/A)/\cong.
    \end{align}
    Let us see that this map only depends on the isomorphism class of $E$. Let $F$ be $(\mg{G},\phi)$-vector bundle over $X$ isomorphic to $E$, choose a trivialization $\beta:F|A\longrightarrow A\times V$ such that $\beta^{-1}\alpha\pi$ produces a $(\mg{G},\phi)$-map $A\longrightarrow GL(V)$ represented by a class of the form $(1,1\ldots, 1)\in [A,\GL(V)]^{\mg{G}}$, following our convention in Rem. \ref{convention}. Then by Cor. \ref{homotopyinvoftriv} $E/\alpha$ is isomorphic to $F/\beta$.
\end{proof}

    \begin{remark}
    For equivariant complex vector bundles the quoting map $f: X \longrightarrow X/A$ induces a bijection between the corresponding isomorphism classes of equivariant complex vector bundles over $X$ and $X/A$, see for example 2.10 in \cite{segal} or 1.4.8 in \cite{atiyahktheory}.
\end{remark}

The usual \textbf{suspension} of a $G$-space $X$ is $S X:=C_{+}(X)\cup C_{-}(X)$ where $C_{-}(X)=X\times [0,\frac{1}{2}]/(x,0)\sim (y,0)$ and $C_{+}(X)=X\times [\frac{1}{2},1]/(x,1)\sim (y,1)$. We put the trivial $\mg{G}$ action on $[0,1]$, so $C_{-}(X)$, $C_{+}(X)$ and $SX$ are $G$-spaces.  The two cones $C_{-}(X)$ and $C_{+}(X)$
are $G$-contractible.

Denote by $\MVec^n_{(\mg{G},\phi)}(S X)$ the set of $(\mg{G},\phi)$-vector bundles over $X$ of dimension $n$.
\begin{corollary}\label{kofsuspension}
    There is a natural bijection 
    \begin{align}
    \MVec^n_{(\mg{G},\phi)}(S X)/ \cong\, \overset{\cong}{\longrightarrow} \bigcup_{\scalebox{0.75}{\begin{tabular}{c}
        Isomorphism classes of \\
         $n$-dimensional \\ representations $V$ of $(\mg{G},\phi)$           
    \end{tabular}}}[X, \GL(V)]^{\mg{G}}
    \end{align}
\end{corollary}

\begin{proof}
    If $E$ is a $n$-dimensional $(\mg{G},\phi)$-vector bundle over $SX$, then the restrictions of $E$ to $C_{+}(X)$ and $C_{-}(X)$ are trivial,  and by Cor. \ref{trivial} we have trivializations $\alpha_{\pm}:E|_{C_{\pm}(X)} \longrightarrow C_{\pm}(X)\times V$. We can restrict the trivializations to $X=X\times \{\frac{1}{2}\}$ and get an isomorphism of $(\mg{G},\phi)$-vector bundles 
    \begin{align}
    \alpha_{+}\circ\alpha^{-1}_-: X\times V \longrightarrow X\times V
    \end{align}
    This map is of the form $(x,e)\longmapsto (x,\psi_x(e))$ where $\psi_x\in  \GL(V)$ so we have a map
    \begin{align}
    \psi: X\longrightarrow \GL(V).
    \end{align}
    This map commutes with the action of $(\mg{G},\phi)$:
    \begin{align}
        \alpha_{+}\circ\alpha^{-1}_-(g\cdot (x,v))&=g \cdot \alpha_{+}\circ\alpha^{-1}_-(x,v) \\ \alpha_{+}\circ\alpha^{-1}_-(g\cdot x,g\cdot v)&= g\cdot (x,\psi_x(v)) \\ (g\cdot x,\psi_{g\cdot x}(g\cdot v))&= (g\cdot x,g\cdot \psi_{x}(v))
    \end{align}
    so $\psi_{g\cdot x}(g\cdot v)= g\cdot \psi_{x}(v)$. If we put $w=gv$ we get
    \begin{align}
        \psi_{g\cdot x}(w)&= g\cdot \psi_{x}(g^{-1}w)\\
        \psi_{g\cdot x}(w)&= (g\cdot \psi_x)(w),
    \end{align}
    hence $\psi(g\cdot x)=g\cdot\psi(x)$ i.e. $\psi\in [X,\GL(V)]^{\mg{G}}$. So we have a map

    \begin{align}
    \MVec^n_{(\mg{G},\phi)}(S X) \longrightarrow \bigcup_{\scalebox{0.75}{\begin{tabular}{c}
        Isomorphism classes of \\
         $n$-dimensional \\ representation $V$ of $(\mg{G},\phi)$           
    \end{tabular}}}[X, \GL(V)]^{\mg{G}}.
    \end{align}
    We can show this function descends to isomorphism classes $\MVec^n_{(\mg{G},\phi)}(S X)/\cong$ as in the proof of Cor. \ref{trivializationiso}.
    
    The clutching construction defines a map
    \begin{align}
    \bigcup_{\scalebox{0.75}{\begin{tabular}{c}
        Isomorphism classes of \\
         $n$-dimensional \\ representation $V$ of $(\mg{G},\phi)$           
    \end{tabular}}}[X, \GL(V)]^{\mg{G}} \longrightarrow \MVec^n_{(\mg{G},\phi)}(S X)/\cong
    \end{align}
    and both maps are inverses of each other. 
\end{proof}

We need the existence of invariant Hermitian metrics.

\begin{lemma}\label{metrica}
    Let $E$ be a $(\mg{G},\phi)$-vector bundle on a $G$-space $X$. Then there exists an Hermitian metric on $E$ invariant under the action of $(\mg{G},\phi)$.
\end{lemma}

\begin{proof}
    Given a finite open cover $\{U\}_i$ of $X$ where $E$ is locally trivial, choose Hermitian metrics on each $E|_U\cong U\times \mathbb{C}^n$ and then add them together with a partition of unity, as in Lemma 1.4.10 of \cite{atiyahktheory}.  So we have a metric $\langle \cdot,\cdot \rangle$ on $E$. Now we just average the metric to obtain an equivariant one, namely for $x\in X$ and $v,w\in E_x$ 
    \begin{align}
    \langle v,w\rangle_{\mg{G}}:=\frac{1}{|\mg{G}|}\sum_{\mg{G}} \mathbb{K}^{\phi(g)}\left(\langle gv,gw\rangle\right) .
    \end{align}
    The complex conjugation is needed to make the first entry linear and the second entry antilinear.
\end{proof}

\begin{lemma}\label{embetriv}
    If $p:E\longrightarrow X$ is a $(\mg{G},\phi)$-vector bundle, then there exist a representation $V$ of $(\mg{G},\phi)$ and a $(\mg{G},\phi)$-embedding of $(\mg{G},\phi)$-vector bundles over $X$
    \begin{align}
    \xymatrix{E \ar@{^(->}[rr] \ar[dr]_{p} && X\times V\ar[dl]^{\pi_1} \\ &X.&}
    \end{align}
\end{lemma}
\begin{proof}
    Let $x\in X$. Notice that the sum 
    \begin{align}
    \bigoplus_{y\in \operatorname{orb}_{\mg{G}}(x)}F_y, 
    \end{align}
    of the fibers of $E$ over the points in the orbit $\operatorname{orb}_{\mg{G}}(x)$, is a representation of $(\mg{G},\phi)$. We have a natural $(\mg{G},\phi)$-embedding of $(\mg{G},\phi)$-vector bundles over $\operatorname{orb}_{\mg{G}}(x)$
    \begin{align}
      \Psi_x:E|_{\operatorname{orb}_{\mg{G}}(x)} \longrightarrow \operatorname{orb}_{\mg{G}}(x)\times \bigoplus_{y\in \operatorname{orb}_{\mg{G}}(x)}F_y.
      \end{align}
    By Lem. \ref{extensionofiso} there exist an invariant open neighborhood $U_x$ of $\operatorname{orb}_{\mg{G}}(x)$ and an extension of $\Psi_{U_x}$ which is still an embedding 
    \begin{align}
     \Psi_{U_x}:E|_{U_x} \longrightarrow U_x\times \bigoplus_{y\in \operatorname{orb}_{\mg{G}}(x)}F_y.
     \end{align}
    We can cover $X$ with a finite collection of trivializing open sets $\{U_\alpha\}$ since $X$ is a compact space. Let $\{f_\alpha\}$ be a partition of unity with $\operatorname{supp}(f_\alpha)\subset U_\alpha$. Denote by $V$ the sum $\oplus_\alpha \left(\bigoplus_{y\in \operatorname{orb}(\alpha)}F_{\alpha}\right)$ of representations of $(\mg{G},\phi)$. Define the maps
    \begin{align}
        \widetilde{\Psi}_\alpha: E &\longrightarrow X\times V \\
        v & \longmapsto \begin{cases}
            f_\alpha(p(x)) \Psi_{U_{\alpha}}(v) & \text{ if } p(v)\in U_\alpha \\
            0 & \text{ if } p(v)\not\in U_\alpha.
        \end{cases}
    \end{align}
    The $(\mg{G},\phi)$-maps $\widetilde{\Psi}_{\alpha}$ define an embedding $E\longrightarrow X\times V$.
\end{proof}

\begin{proposition}\label{complement}
    If $E$ is a $(\mg{G},\phi)$-vector bundle on a $G$-space $X$, then there is a trivial $(\mg{G},\phi)$-vector bundle $\textbf{V}$ and a $(\mg{G},\phi)$-vector bundle $E^\bot$ such that $E\oplus E^{\bot}\cong \textbf{V}$.
\end{proposition}

\begin{proof}
    By Lem. \ref{embetriv} there exist a $(\mg{G},\phi)$-embedding $E\longrightarrow X\times V$. By Lem. \ref{metrica} there exist an Hermitian metric on $X\times V$ invariant under the action of $(\mg{G},\phi)$. Define $E^{\bot}$ as the orthogonal complement of $E$ in $X\times V$.
\end{proof}

\begin{proposition}\label{defireduc}
    The set of isomorphism classes of stable $(\mg{G},\phi)$-vector bundles is an abelian group $\widetilde{\K}_{\mg{G}}(X)$. It can be identified naturally with a quotient group of $\K_{\mg{G}}(X)$.
\end{proposition}
\begin{proof}
    Prop. \ref{complement} guaranties the existence of inverses in $\widetilde{\K}_{\mg{G}}(X)$.

    Now we must prove there exists an epimorphism $\K_{\mg{G}}(X)\longrightarrow \widetilde{\K}_{\mg{G}}(X)$. Take an element of $\K_{\mg{G}}(X)$ represented by $E-F$. Choose a $(\mg{G},\phi)$-vector bundle $F'$ such that $F\oplus F'\cong \textbf{V}$ is trivial. We have
    \begin{align}
        E-F=& E\oplus F'-F\oplus F'\\ =& H-\textbf{V}
    \end{align}
    so any element $E-F\in \K_{\mg{G}}(X)$ can be represented by $H-\textbf{V}$. Define the homomorphism
    \begin{align}
     \nonumber   \K_{\mg{G}}(X) &\longrightarrow \widetilde{\K}_{\mg{G}}(X) \\ H-\textbf{V} &\longmapsto H.
    \end{align}
    Clearly this is an epimorphism and the kernel is generated by elements of the form $\textbf{V}'-\textbf{V}$ which is isomorphic to $\mathbf{R}(G)$.
\end{proof}

\begin{theorem}\label{freeaction}
    If $N\leq G_0$ is a normal subgroup which acts freely on $X$, then the projection $\pi: X \longrightarrow X/N$ induces an isomorphism
    \begin{align}
    \pi^*:         \K_{\mg{G}/N}(X/N)\longrightarrow \K_{\mg{G}}(X)
    \end{align}
\end{theorem}

\begin{proof}
    An inverse to $\pi^*$ is given by the homomorphism $E\longmapsto E/N$. It is the same as 2.1 of \cite{segal}.
\end{proof}

\subsection{Magnetic K-theory as a cohomology theory}
In this section we are going to enhance the  magnetic equivariant $K$-theory to a  cohomology theory for a fixed magnetic group $(\mg{G},\phi)$.

\begin{definition}
    Let $X$ be a compact $G$-space with a base point $x_0$. The inclusion $x_0\longrightarrow X$ induces a restriction homomorphism $\psi: \K_{\mg{G}}(X) \longrightarrow \K_{\mg{G}}(x_0)$ and we define the \textbf{reduced $K$-theory} of $X$ as 
    \begin{align}
    \widetilde{\K}_{\mg{G}}(X):=\ker\psi .
    \end{align}
    The $K$-theory of a compact $(\mg{G},\phi)$-pair $(X,Y)$ is defined as 
    \begin{align}
    \K_{\mg{G}}(X,Y):=\widetilde{\K}_{\mg{G}}(X/Y).
    \end{align}
    
\end{definition}
In case $Y=\emptyset$ we set $X/Y=X_+:=X\sqcup \{*\}$ to get $\K_{\mg{G}}(X)=\K_{\mg{G}}(X,\emptyset)$.

\begin{remark}
    This definition of reduced $\K$-theory is naturally isomorphic with the one given in Prop. \ref{defireduc}.
\end{remark}
\begin{definition}
    We define the $(p,q)$-\textbf{suspension groups} of the $G$-pair $(X,Y)$ as
    \begin{align}
    \K^{p,q}_{\mg{G}}(X,Y)=\K_{\mg{G}}(X\times B^{p,q}, X\times S^{p,q}\cup Y\times B^{p,q})
    \end{align}
\end{definition}
The usual suspension groups $\K^{-q}_{\mg{G}}$ are given by 
\begin{align} \label{suspension K-theory}
\K^{-q}_{\mg{G}}(X,Y):=\K^{0,q}_{\mg{G}}(X,Y).
\end{align}
In particular for a base point $x_0\in X$, we have $\rK^{-q}_{\mg{G}}(X)=\rK_{\mg{G}}(\Sigma^{q} X)$ where $\Sigma X=X\times I/ (X\times \{0,1\}\cup \{x_0\}\times I)$ is the \textbf{reduced suspension} of $(X,x_0)$.

\begin{example}\label{Hopfbundle}
    Consider the inclusion $\iota: \mathbb{R}^2=\mathbb{R}^{0,2}\longrightarrow \mathbb{R}^{2,2}=\mathbb{C}^2$ which produces a map
    \begin{align}
    \PS(\mathbb{R}^2) \longrightarrow \PS(\mathbb{C}^2)
    \end{align}
    and induces a morphism
    \begin{align}
    \xymatrix{
    \widetilde{\K}_{\mg{G}}(\operatorname{P}(\mathbb{C}^2)) \ar[r]^{\iota^*} \ar@{=}[d] & \widetilde{\K}_{\mg{G}}(\operatorname{P}(\mathbb{R}^2)) \ar@{=}[d] \\ \widetilde{\K}_{\mg{G}}(B^{1,1}/S^{1,1}) \ar@{=}[d]
     \ar[r]^{\iota^*} & \widetilde{\K}_{\mg{G}}(B^{0,1}/S^{0,1}) \ar@{=}[d]
     \\  \K^{1,1}_{\mg{G}}(*)
     \ar[r]^{\iota^*} & \K^{0,1}_{\mg{G}}(*)
    }
    \end{align}
    Consider the Hopf line bundle (dual to the tautological $(\mg{G},\phi)$-vector bundle)
    \begin{align}
     H^*=\left\{(L,p)\in \operatorname{P}(\mathbb{C}^2)\times \mathbb{C}^2\, |\, p\in L\right\}
     \end{align}
     over $\operatorname{P}(\mathbb{C}^2)$). This is a $(\mg{G},\phi)$-vector bundle with the action induced by the conjugation action and the homomorphism $\mg{G}\overset{\phi}{\longrightarrow} \mathbb{Z}/2$.
      If we subtract a trivial line bundle, then the \textbf{reduced Hopf bundle}, also called the \textbf{Bott element} $[H]-\textbf{1}$, is in the reduced $K$-group $\widetilde{\K}_{\mg{G}}(\operatorname{P}(\mathbb{C}^2))$ and so $\eta:=\iota^*(H-\textbf{1})$ is the \textbf{reduced Hopf bundle} over $\operatorname{P}(\mathbb{R}^2)$. Of course we have a $\mathbb{Z}/2$-homotopy equivalence $B^{1,1}/S^{1,1}\longrightarrow \PS(\mathbb{R}^{2,2})$ given by stereographic projection and the Hopf fibration.
\end{example}

\begin{lemma}\label{exactade3}
    For $(X,Y)$ a $G$-pair we have an exact sequence
    \begin{align}
    \xymatrix{ \K_{\mg{G}}(X,Y) \ar[r]^{j^*} & \K_{\mg{G}}(X) \ar[r]^{i^*} & \K_{\mg{G}}(Y)}
    \end{align}
    where $i:Y\longrightarrow X$ and $j: (X,\emptyset) \longrightarrow (X,Y)$ are the inclusions.
\end{lemma}

\begin{proof}
    The composition $i^*j^*$ is induced by the composition $ji:(Y,\emptyset)\longrightarrow (X,Y)$, which factors through $(Y,Y)$, and so the $i^*j^*$ factors through the zero group $\K_{\mg{G}}(Y,Y)$ and thus $\operatorname{im} j^* \subset \ker i^*$.

    Now, suppose that $\xi\in \ker i^*$. Represent $\xi$ as $E-\textbf{V}$ where $E$ is a $(\mg{G},\phi)$-vector bundle over $X$ and $\textbf{V}$ is a trivial $(\mg{G},\phi)$-vector bundle. Since $i^*(\xi)=0$ it follows that $[E|_Y]=[\textbf{V}]$ in $\K_{\mg{G}}(Y)$. This implies that for some trivial $(\mg{G},\phi)$-vector bundle $W$ we have
    \begin{align}
    (E\oplus \textbf{W})|_Y=\textbf{V}\oplus \textbf{W}.
    \end{align}
    Thus we have a trivialization $\alpha$ of $(E\oplus \textbf{W})|_Y$ and hence we have the $(\mg{G},\phi)$-vector bundle $(E\oplus \textbf{W})/\alpha \longrightarrow X/Y$. If we define the element $\eta= (E\oplus \textbf{W})/\alpha- \textbf{V}\oplus \textbf{W}\in \widetilde{\K}_{\mg{G}}(X/Y)=\K_{\mg{G}}(X,Y)$, then
     \begin{align}
     j^*(\eta)=E\oplus \textbf{W}- \textbf{V}\oplus \textbf{W}=E-\textbf{V}=\xi .
     \end{align}
\end{proof}
\begin{corollary}\label{reduses}
    Let $(X,Y)$ be a $G$-pair and $y\in Y$ a base point. Then we have an exact sequence
    \begin{align}
    \xymatrix{ \K_{\mg{G}}(X,Y) \ar[r]^{j^*} & \widetilde{\K}_{\mg{G}}(X) \ar[r]^{i^*} & \widetilde{\K}_{\mg{G}}(Y)}
    \end{align}
    where $i:Y\longrightarrow X$ and $j: (X,\emptyset) \longrightarrow (X,Y)$ are the inclusions.
\end{corollary}
\begin{proof}
    It follows from the next commutative diagram 
    \begin{align}
    \xymatrix{ 0 \ar[r] & \K_{\mg{G}}(\{y\}) \ar@{=}[r] & \K_{\mg{G}}(\{y\}) \\ \K_{\mg{G}}(X,Y) \ar[r]^{j^*} \ar[u] & \K_{\mg{G}}(X) \ar[r]^{i^*} \ar[u] & \K_{\mg{G}}(Y) \ar[u] \\ \K_{\mg{G}}(X,Y) \ar[r]^{j^*} \ar@{=}[u]& \widetilde{\K}_{\mg{G}}(X) \ar[r]^{i^*} \ar@{^(->}[u]& \widetilde{\K}_{\mg{G}}(Y). \ar@{^(->}[u] }
    \end{align}
\end{proof}
\begin{lemma}\label{sucesionlarga}
    For a $G$-pair $(X,Y)$ we have a long exact sequence
    \begin{align}
    \dots \longrightarrow \K^{-1}_{\mg{G}}(X)\longrightarrow \K^{-1}_{\mg{G}}(Y) \longrightarrow \K_{\mg{G}}(X,Y) \longrightarrow \K_{\mg{G}}(X)\longrightarrow \K_{\mg{G}}(Y)
    \end{align}
\end{lemma}

\begin{proof} 
    It is enough to show that we have a five term exact sequence
    \begin{align}
    \widetilde{\K}^{-1}_{\mg{G}}(X)\longrightarrow \widetilde{\K}^{-1}_{\mg{G}}(Y) \longrightarrow \K_{\mg{G}}(X,Y) \longrightarrow \widetilde{\K}_{\mg{G}}(X)\longrightarrow \widetilde{\K}_{\mg{G}}(Y).
    \end{align}
    In fact, if this has been established, then replacing $(X,Y)$ by $(S^nX,S^nY)$ for $n=1,2,3,\ldots $ we obtain a family of exact sequences that can be concatenated to produce a long exact sequence. Then replacing $(X,Y)$ by $(X_+,Y_+)$ we get the long exact sequence of the statement.
    For a base point $x_0\in X$, denote by $\widetilde{C}(X)=X\times  I/(X\times \{0\}\cup \{x_0\}\times I)$ the reduced cone. So consider the following commutative diagram of $G$-spaces
    \begin{align}
    \xymatrix{
  Y \ar@{^(->}[r] & X \ar[r] & X/Y & \Sigma Y & \Sigma X \\ & X \ar@{^(->}[r] \ar@{=}[u]
  & X\cup \widetilde{C}Y \ar[u]_{\simeq} \ar[r] & X\cup \widetilde{C}Y/X \ar[u]_{\cong} & \Sigma X \ar@{=}[u] \\
  & & X\cup \widetilde{C}Y \ar@{^(->}[r] \ar@{=}[u] & (X\cup \widetilde{C}Y)\cup \widetilde{C}X=\widetilde{C}X\cup \widetilde{C}Y \ar[r] \ar[u]_{\simeq} & \widetilde{C}X \cup \widetilde{C}Y / X\cup \widetilde{C}Y \ar[u]_{\cong}  \\ &&&\widetilde{C}X\cup \widetilde{C}Y \ar@{=}[u] \ar@{^(->}[r]& (\widetilde{C}X\cup \widetilde{C}Y)\cup \widetilde{C}(X\cup \widetilde{C}Y) \ar[u]_{\simeq}}
  \end{align}
  By Cor. \ref{reduses}, when we apply $\widetilde{\K}_{\mg{G}}(\mathunderscore )$, the first three rows become exact sequences. All spaces, maps and homotopies involved here are in the category of $G$-spaces, and the cones are $G$-contractible. So the same arguments of 2.6 in \cite{segal} applies to this context, but of course with the appropriate results about the homotopy properties of $(\mg{G},\phi)$-vector bundles, $G$-maps and $G$-homotopies.    
\end{proof}

If $X$ is a locally compact $G$-space, then the one point compactification $X\cup\{\infty\}$ is still a $G$-space with the new point $\infty $ being a fixed point. This allows us to define the $\K_{\mg{G}}$-theory of locally compact $G$-spaces $X$ as 
\begin{align}
\K_{\mg{G}}(X):=\ker\left(\K_{\mg{G}}(X\cup \{\infty\}) \longrightarrow \K_{\mg{G}}(\infty)\right).
\end{align}

If $(X,Y)$ is a compact $G$-pair, then we have natural a $G$-homeomorphism
\begin{align}
X \slash Y \longrightarrow (X\backslash Y)\cup \{\infty\}
\end{align}
 so we have an isomorphism
 \begin{align} \label{K(X/Y)=K(X,Y)}
 \K_{\mg{G}}(X,Y)\cong \K_{\mg{G}}(X\slash Y).
 \end{align}

\begin{corollary}\label{exactforlocally}
    For a locally compact $G$-pair $(X,Y)$ we have a long exact sequence
    \begin{align}
    \dots \longrightarrow \K^{-1}_{\mg{G}}(X)\longrightarrow \K^{-1}_{\mg{G}}(Y) \longrightarrow \K_{\mg{G}}(X,Y) \longrightarrow \K_{\mg{G}}(X)\longrightarrow \K_{\mg{G}}(Y)
    \end{align}
\end{corollary}

\begin{proof}
    The same as II.4.7 in \cite{karoubi}, we have a commutative diagram
\begin{align}
\scalebox{0.8}{\xymatrix{
0 & 0 & 0 & 0 & 0 \\
\K^{-1}_{\mg{G}}(\infty) \ar[r] \ar[u] & \K^{-1}_{\mg{G}}(\infty) \ar[r] \ar[u] & 0 \ar[r] \ar@{=}[u] & \K_{\mg{G}}(\infty) \ar[r] \ar[u] & \K_{\mg{G}}(\infty) \ar[u] \\
\K^{-1}_{\mg{G}}(X\cup \{\infty\}) \ar[r] \ar[u] & \K^{-1}_{\mg{G}}(Y\cup \{\infty\}) \ar[r] \ar[u] & \K_{\mg{G}}(X\cup \{\infty\}\backslash Y\cup\{\infty\}) \ar[r] \ar[u] & \K_{\mg{G}}(X\cup \{\infty\}) \ar[r] \ar[u] & \K_{\mg{G}}(Y\cup \{\infty\})  \ar[u] \\
\K^{-1}_{\mg{G}}(X) \ar[r] \ar[u] & \K^{-1}_{\mg{G}}(Y) \ar[r] \ar[u] & \K_{\mg{G}}(X\backslash Y) \ar[r] \ar[u]^{\cong }  & \K_{\mg{G}}(X) \ar[r] \ar[u] & \K_{\mg{G}}(Y)  \ar[u] \\ 0 \ar[u]& 0 \ar[u]& 0 \ar[u]& 0 \ar[u]& 0 \ar[u]
}}
\end{align}
    where the middle row is exact Lem. \ref{sucesionlarga} and the columns are just the definition of $\K$-theory for locally compact spaces.
\end{proof}

\begin{corollary}\label{sucesiondetres}
    Given a locally compact $G$-triple $(X,Y,Z)$ there is a long exact sequence
\begin{align}    
    \dots \longrightarrow \K^{-1}_{\mg{G}}(X,Z) \longrightarrow \K^{-1}_{\mg{G}}(Y,Z) \longrightarrow \K_{\mg{G}}(X,Y) \longrightarrow \K_{\mg{G}}(X,Z) \longrightarrow \K_{\mg{G}}(Y,Z) 
    \end{align}
\end{corollary}

\begin{proof}
    As in II.4.14 of \cite{karoubi}, apply Cor. \ref{exactforlocally} to the $G$-pair $(X\backslash Z,Y\backslash Z)$. 
\end{proof}

\begin{corollary}
    Given a locally compact $G$-triple $(X,Y,Z)$ there is a long exact sequence
    \begin{align}
    \dots \longrightarrow \K^{p,1}_{\mg{G}}(X,Z) \longrightarrow \K^{p,1}_{\mg{G}}(Y,Z) \longrightarrow \K^{p,0}_{\mg{G}}(X,Y) \longrightarrow \K^{p,0}_{\mg{G}}(X,Z) \longrightarrow \K^{p,0}_{\mg{G}}(Y,Z) 
    \end{align}
    for each integer $p\geq 0$.
\end{corollary}

\begin{proof}
    As in page 373 of \cite{atiyahreal}, apply Cor. \ref{sucesiondetres} to the $G$-triple $(X\times B^{p,0},X\times S^{p,0}\cup Y\times B^{p,0},X\times S^{p,0})$. 
\end{proof}

\subsection{Products and Bott periodicity}

Let $E$ be a $(\mg{G},\phi)$-vector bundle over a $G$-space $X$. Recall that $\PS(E)$ the \textbf{projective bundle} of $E$, which fiber at $x\in X$ is given by the projective space $\PS(E_x)$, is also a $G$-space. The tautological line-bundle $H^*$ is given by
\begin{align}
\xymatrix{H^* \ar[d]_{\pi_1}\\ \operatorname{P}(E),}\hspace{-0.4cm} =\{(l,v)\in \PS(E)\times E \mid l\in\operatorname{P}(E) \text{ and }v\in l\} 
\end{align}
it is naturally a $(\mg{G},\phi)$-vector bundle. Denote by $H$ its dual line bundle, known as the standard line $(\mg{G},\phi)$-vector bundle over $\operatorname{P}(E)$. 

The projection map $\operatorname{P}(E)\longrightarrow X$ induces a ring homomorphism $\K_{\mg{G}}(X) \longrightarrow \K_{\mg{G}}(\operatorname{P}(E))$ so that $\K_{\mg{G}}(\operatorname{P}(E))$ becomes a $\K_{\mg{G}}(X)$-algebra. For a particular $(\mg{G},\phi)$-vector bundle $E$ we can determine the structure of the algebra:
\begin{theorem}[Bott periodicity]\label{pediocicitytheorem}
    Let $L$ be a line $(\mg{G},\phi)$-vector bundle over a $G$-space $X$, $H$ the standard line $(\mg{G},\phi)$-vector bundle over the $G$-space $\operatorname{P}(L\oplus \textbf{1})$. Then, as a $\K_{\mg{G}}(X)$-algebra, $\K_{\mg{G}}(\operatorname{P}(L\oplus \textbf{1}))$ is generated by $[H]$, subject to the single relation
    \begin{align}
    \left([H]-[\textbf{1}]\right) \left([L][H]-[\textbf{1}]\right)=0.
    \end{align}
\end{theorem}

\begin{proof}
As Atiyah points out in theorem 2.1 of \cite{atiyahreal}, we can follow the proof given in \cite{periodicitytheorem} and modify the necessary parts. That proof is divided into four parts, we will state the results and sketch the proof in the appropriated equivariant way. To simplify notation set $\PS:=\PS(L\oplus \textbf{1})$.

\vspace{0.3cm}

\textbf{Part 1. $(\mg{G},\phi)$-vector bundles over $\PS$ as clutching bundles by Laurent polynomials.}

By Lem. \ref{metrica}, choose an $G$-invariant metric in $L$ and let $S\subset L$ be the unit circle bundle in this metric. We identify $L$ with a subspace of $\PS$ and we can write
\begin{align}
\PS= \PS^0 \cup \PS^{\infty}, \hspace{1cm} S=\PS^0\cap \PS^{\infty},
\end{align}
where $\PS^0$ is the closed disc $(\mg{G},\phi)$-bundle interior to $S$ and $\PS^{\infty}$ is the closed disc $(\mg{G},\phi)$-bundle exterior to $S$. So we have the $(\mg{G},\phi)$-bundles
\begin{align}
\xymatrix{\PS^0 \ar[rd]_{\pi_0}& S \ar[d]^{\pi} \ar@{^(->}[r] \ar@{_(->}[l]& \PS^{\infty} \ar[ld]^{\pi_{\infty}} \\ &X.&}
\end{align}
Suppose now that $E^0$, $E^{\infty}$ are two $(\mg{G},\phi)$-vector bundles over $X$ and $f:\pi^*E^0\longrightarrow \pi^*E^{\infty}$ is an isomorphism of $(\mg{G},\phi)$-vector bundles over $S$. Then  we can form the $(\mg{G},\phi)$-vector bundle
\begin{align}
\xymatrix{\pi^*_0E^0\cup_f \pi^*_{\infty}E^{\infty} \ar[d] \\ \PS.}
\end{align}
We shall denote this $(\mg{G},\phi)$-vector bundle for brevity by $(E^0,f,E^{\infty})$ and we shall say that $f$ is a clutching function for $(E^0,E^{\infty})$.
\begin{lemma}\label{vecasclut}
    Let $E$ be any $(\mg{G},\phi)$-vector bundle over $\PS$ and let $E^0$, $E^{\infty}$ be the $(\mg{G},\phi)$-vector bundles over $X$ induced by the 0-section and the $\infty$-section respectively. Then there exist an $(\mg{G},\phi)$-isomorphism $f:\pi^*E^0\longrightarrow \pi^* E^{\infty}$ such that 
    \begin{align}
    E\cong (E^0,f,E^{\infty}),
    \end{align}
    the isomorphism being the natural one on the 0-section and the $\infty$-section.
\end{lemma}
\begin{proof}
    Let $s_0: X\longrightarrow \PS^0$ be the 0-section. Then $s_0\circ\pi_0([v:1])=[0,1]$ is $G$-homotopic to the identity, for example consider the $G$-homotopy $H_s([v:1])=[sv:1]$. So there exist an $(\mg{G},\phi)$-isomorphism
    \[f_0: E|_{\PS^0}\longrightarrow \pi^*_0E^0=\pi^*_0s^*_0 (E|_{\PS^0}).\]
    We can use a similar procedure to obtain an $(\mg{G},\phi)$-isomorphism 
    \[f_{\infty}: E|_{\PS^{\infty}}\longrightarrow \pi^*_{\infty}E^{\infty}=\pi^*_{\infty}s^*_{\infty} (E|_{\PS^{\infty}}).\]
    If we define $f=f_{\infty}f^{-1}_0$, then we are done.
    \end{proof}
    Now we will focus on the clutching functions. The natural inclusion of $(\mg{G},\phi)$-bundles $\iota:S\longrightarrow L$ defines a $G$-section $z\in \Gamma_{(\mg{G},\phi)}(\pi^* L)$ of $\pi^* L$ given by $z(v):=(v,\iota(v))$. This is depicted as:
    \begin{align}
    \xymatrix{\pi^*L \ar[r] \ar[d]& L \ar[d]^{\pi} \\ S \ar[r]_{\pi} \ar@(lu,dl)[u]^{z}& X.}
    \end{align}
    Using the canonical isomorphisms 
    \begin{align}
    \pi^*(L)\cong \pi^*\Hom(\textbf{1},L)
    \end{align}
    we may also regard $z$ as a $G$-section of the $(\mg{G},\phi)$-vector bundle $\pi^*\Hom(\textbf{1},L)$ and, as such, it has an inverse $z^{-1}$ which is a $G$-section of $\pi^*\Hom(L,\textbf{1})\cong \pi^*(L^{-1})$. More generally, for any integer $k$, we may regard $z^k$ as a $G$-section of $\pi^* L^k$. If $a_k\in \Gamma_{(\mg{G},\phi)}(L^{-k})$ then
    \begin{align}
    a_kz^k:=\pi^*(a_k)\otimes z^k\in \Gamma_{(\mg{G},\phi)}(\pi^* \textbf{1}),
    \end{align}
    i.e. it is a function on $S$. Finally suppose that $E^0$, $E^{\infty}$ are two $(\mg{G},\phi)$-vector bundles on $X$ and that 
    \begin{align}
    a_k\in \Gamma_{(\mg{G},\phi)}\Hom(L^k\otimes E^0,E^{\infty}),
    \end{align}
    then 
    \begin{align}
    a_kz^k:=\pi^*(a_k)\otimes z^k\in \Gamma_{(\mg{G},\phi)}\Hom(\pi^* E^0,\pi^*E^{\infty}).
    \end{align}
    If $f=\sum^n_{-n}a_kz^k: \pi^* E^0\longrightarrow \pi^*E^{\infty}$ is an $(\mg{G},\phi)$-isomorphism, then it defines a clutching function and we call this a \textbf{Laurent clutching function} for $(E^0,E^\infty)$.
    One can prove that there is a $(\mg{G},\phi)$-isomorphism
    \begin{align}
    H^{-k}\cong (\textbf{1},z^{-k},L^{-k}).
    \end{align}
    Suppose now that $f\in \Gamma_{(\mg{G},\phi)}\Hom(\pi^*E^0,\pi^*E^{\infty})$ is any $G$-section, then we can define its \textbf{Fourier coefficients}
    \begin{align}
    a_k(x)=\frac{1}{2\pi i}\int_{S_x}f_xz_x^{-k-1}dz_x
    \end{align}
    where $f_x$ and $z_x$ denote the restrictions of $f$, $z$ to $S_x$ and $dz_x$ is therefore a differential on $S_x$. The section $a_k$ is $G$-equivariant:
    \begin{align}
    (g\cdot a_k)(x)=& ga_k(g^{-1}x)\\
    =& g\left(\frac{1}{2\pi i}\int_{S_{g^{-1}x}}f_{g^{-1}x}z^{-k-1}_{g^{-1}x}dz_{g^{-1}x}\right) \\
    =& \frac{(-1)^{\phi(g)}}{2\pi i} \int_{S_{g^{-1}x}}(gf_{g^{-1}x})(gz_{g^{-1}x})^{-k-1}(gdz_{g^{-1}x}) \\ =& \frac{1}{2\pi i}\int_{S_x} f_xz^{-k-1}_x dz_x \\
    =& a_k(x).
\end{align}
Let $s_n$ be the partial sum 
\begin{align}
s_n=\sum^n_{-n}a_kz^k
\end{align}
and define the \textbf{Cesaro means}
\begin{align}
f_n=\frac{1}{n}\sum^n_{k=n}s_k.
\end{align}
\begin{lemma}
        Let $f$ be any clutching function for $(E^0,E^\infty)$, $f_n$ the sequence of Cesaro  means of the Fourier series of $f$. Then $f_n$ converges uniformly to $f$ and hence is a Laurent clutching function for all sufficiently large $n$.
    \end{lemma}
    \begin{proof}
        This is a extension of the Fej\'er's theorem, see chapter 2 of \cite{hoffman2014banach} for a proof.
    \end{proof}

\textbf{Part 2. Linearization of Laurent clutching functions}

All the isomorphisms presented here are $(\mg{G},\phi)$-isomorphisms of $(\mg{G},\phi)$-vector bundles, since they are in terms of $z^k$ and $a_k$ which are $G$-sections.

\vspace{0.3cm}

\textbf{Part 3. Linear clutching functions.}

Let $p$ be a linear clutching function for $(E^0,E^\infty)$ and let the endomorphisms $Q^0:E^0\longrightarrow E^0$ and $Q^\infty:E^\infty\longrightarrow E^\infty$  be given by
\begin{align}
Q_x^0=\frac{1}{2\pi i}\int_{S_x}p_x^{-1}dp_x, \hspace{1cm} Q^{\infty}_x=\frac{1}{2\pi i}\int_{S_x}dp_x p^{-1}_x.
\end{align}
These are $(\mg{G},\phi)$-endomorphisms, in fact for $g\in \mg{G}$
\begin{align}
    (g\cdot Q^0)_x&= gQ^0_{gx}\\ &= g\left(\frac{1}{2\pi i}\int_{S_{gx}} p^{-1}_{gx}dp_{gx}\right) \\ &= \frac{(-1)^{\phi(g)}}{2\pi i}\int_{S_{gx}} gp^{-1}_{gx}gdp_{gx} \\ &= \frac{1}{2\pi i}\int_{S_{x}} p^{-1}_{x}dp_{x},
\end{align}
and similar for $Q^{\infty}$. Again, all the results in this part follow quite formally.

\vspace{0.3cm}

\textbf{Part 4. Construction of an inverse.}

Consider the homomorphism
\begin{align}
 \nonumber   \mu: \K_{\mg{G}}(X)[t]/(t-1)([L]t-1) & \longrightarrow \K_{\mg{G}}(\PS) \\ t& \longmapsto [H].
\end{align}
Then using the previous parts one can construct an inverse for $\mu$.
As we mention before, the same lines and the rest of the proof in \cite{periodicitytheorem} applies quite formally.
\end{proof}

\begin{corollary}
    We have an isomorphism:
    \begin{align}
    \K_{\mg{G}}(\operatorname{P}(\mathbb{C}^2))\cong \frac{\mathbf{R}(G) [H]}{([H]-[\textbf{1}])^2}.
    \end{align}
\end{corollary}

\begin{proof}
    Apply the Thm. \ref{pediocicitytheorem} to the space $X=*$ and $L=\textbf{1}$. In this case $\operatorname{P}(L\oplus \textbf{1})\cong \operatorname{P}(\mathbb{C}^2)$ and by Ex. \ref{kteoriadelpunto} we know $\K_{\mg{G}}(*)=\mathbf{R}(G)$ is the representation ring of $(\mg{G},\phi)$.
\end{proof}

\begin{corollary}\label{onegenerator}
    The $\mathbf{R}(G)$-module $\widetilde{\K}_{\mg{G}}(\operatorname{P}(\mathbb{C}^2))$ is free with generator $[H]-[\textbf{1}]$.
\end{corollary}

In order to present the most classical version of the periodicity theorem we need to talk about external products.

\begin{definition} \label{boxtimes}
    Let $E\longrightarrow X$ and $F\longrightarrow Y$ be $(\mg{G},\phi)$-vector bundles over $G$-spaces $X$ and $Y$. The \textbf{external tensor product} of $E$ and $F$ over $X\times Y$ is defined as:
    \begin{align}
    E\boxtimes F:= \pi_1^*(E)\otimes \pi_2^*(F)
    \end{align}
    The map $(E,F)\longmapsto E\boxtimes F$ induces a pairing 
    \begin{align}
      \nonumber   \K_{\mg{G}}(X)\otimes \K_{\mg{G}}(Y) &\longrightarrow \K_{\mg{G}}(X\times Y) \\
         \alpha\otimes \beta & \longmapsto \alpha * \beta
    \end{align}
    which is called \textbf{external product}.
\end{definition}

If $\Delta: X\longrightarrow X\times X$ denotes the diagonal map, then the composition 
\begin{align}
         \K_{\mg{G}}(X)\otimes \K_{\mg{G}}(X)\longrightarrow \K_{\mg{G}}(X\times X) \overset{\Delta^*}{\longrightarrow} \K_{\mg{G}}(X) 
    \end{align}
    is called the \textbf{internal product} and we will denoted it by $\alpha\beta=\Delta^*(\alpha *\beta)$.
We need also the reduced external product and to this end we need the following classical results.

\begin{corollary}\label{cartesianproduct}
     If $(X,Y)$ is a $G$-pair with $Y$ a $G$-retract of $X$, then for all $n\geq 0$, the sequence 
    \begin{align}
    0\longrightarrow \K^{-n}_{\mg{G}}(X,Y)\longrightarrow \K^{-n}_{\mg{G}}(X) \longrightarrow \K^{-n}_{\mg{G}}(Y) \longrightarrow 0
    \end{align}
    is a split short exact sequence, so $\K^{-n}_{\mg{G}}(X)\cong \K^{-n}_{\mg{G}}(X,Y)\oplus \K^{-n}_{\mg{G}}(Y)$.
    
     If $(X,Y)$ is a $G$-pair of based spaces, the projection maps $\pi_1: X\times Y \longrightarrow X$, $\pi_2:X\times Y \longrightarrow Y$ induces an isomorphism for all $n\geq 0$
    \begin{align}
    \widetilde{\K}^{-n}_{\mg{G}}(X\times Y)\cong \widetilde{\K}^{-n}_{\mg{G}}(X\wedge Y) \oplus \widetilde{\K}^{-n}_{\mg{G}}(X) \oplus \widetilde{\K}^{-n}_{\mg{G}}(Y)
    \end{align}
     
\end{corollary}

\begin{proof}
    For the first we see that the long exact sequence of Lem. \ref{sucesionlarga} splits using the retract.

    For the second observe that $X$ is a $G$-retract of $X\times Y$, and $Y$ is a $G$-retract of $X\times Y/X$. The result follows by applying twice the first result.
\end{proof}
By the definition of reduced $\K$-theory we have the following isomorphism
\begin{align}
    \K_{\mg{G}}(X)\otimes \K_{\mg{G}}(Y)\cong &\left( \widetilde{\K}_{\mg{G}}(X)\oplus \K_{\mg{G}}(*)\right)\otimes \left(\widetilde{\K}_{\mg{G}}(X)\oplus \K_{\mg{G}}(*)\right) \\
    \cong & \left( \widetilde{\K}_{\mg{G}}(X)\otimes \widetilde{\K}_{\mg{G}}(Y)\right)\oplus \left(\widetilde{\K}_{\mg{G}}(X)\otimes \K_{\mg{G}}(*)\right) \oplus \left(\K_{\mg{G}}(*)\otimes \widetilde{\K}_{\mg{G}}(Y)\right)\\ & \oplus \left(\K_{\mg{G}}(*)\otimes \K_{\mg{G}}(*)\right)
\end{align}

Using the splitting and the Lem. \ref{cartesianproduct} we can see the external tensor product $\K_{\mg{G}}(X)\otimes \K_{\mg{G}}(Y) \longrightarrow \K_{\mg{G}}(X\times Y)$ restrict to four maps:

\begin{align}
    \widetilde{\K}_{\mg{G}}(X)\otimes \widetilde{\K}_{\mg{G}}(Y) & \longrightarrow \widetilde{\K}_{\mg{G}}(X\wedge Y), \\
    \widetilde{\K}_{\mg{G}}(X)\otimes \K_{\mg{G}}(*) & \longrightarrow \widetilde{\K}_{\mg{G}}(X),\\
    \K_{\mg{G}}(*)\otimes \widetilde{\K}_{\mg{G}}(Y) & \longrightarrow \widetilde{\K}_{\mg{G}}(Y), \\
    \K_{\mg{G}}(*)\otimes \K_{\mg{G}}(*) & \longrightarrow \K_{\mg{G}}(*).
\end{align}

The first map $\widetilde{\K}_{\mg{G}}(X)\otimes \widetilde{\K}_{\mg{G}}(Y) \longrightarrow \widetilde{\K}_{\mg{G}}(X\wedge Y)$ is call \textbf{reduced external tensor product}, the following two maps induce a structure of $\K_{\mg{G}}(*)=\mathbf{R}(G)$-module on reduced $\K_{\mg{G}}$-theory. Using this reduced tensor product and the natural homeomorphism 
\begin{align}
B^{p,q}\times  B^{s,t}\cong B^{p+s,q+t}
\end{align}
we have a natural product
\begin{align}
\K_{\mg{G}}^{p,q}(X,Y)\otimes \K^{s,t}_{\mg{G}}(X',Y')\longrightarrow \K^{p+s,q+t}_{\mg{G}}(X\times X',X\times Y'\cup Y\times X').
\end{align}

Recall the Bott element given in Ex. \ref{Hopfbundle}
\begin{align}
b=[H]-\textbf{1}\in \K^{1,1}_{\mg{G}}(*)=\K_{\mg{G}}(B^{1,1},S^{1,1})=\widetilde{\K}_{\mg{G}}(\operatorname{P}(\mathbb{C}^2)).
\end{align}
The \textbf{Bott homomorphism} is given by 
\begin{align}
\nonumber    \beta: \K^{p,q}_{\mg{G}}(X,Y) &\longrightarrow \K^{p+1,q+1}_{\mg{G}}(X,Y)\\ a & \longmapsto a*b.
\end{align}

\begin{theorem}[$(1,1)$-periodicity]\label{1,1periodicity}
    The Bott homomorphism $\beta: \K^{p,q}_{\mg{G}}(X,Y) \longrightarrow \K^{p+1,q+1}_{\mg{G}}(X,Y)$ is an isomorphism.
\end{theorem}

\begin{proof}
    We can assume $p=q=0$ replacing $(X,Y)$ by $(X\times B^{p,q},X\times S^{p,q}\cup Y\times B^{p,q})$. Also, we can assume $Y$ to be a point, since $\K_{\mg{G}}(X,Y)=\widetilde{\K}_{\mg{G}}(X/Y)=\widetilde{\K}_{\mg{G}}(X/Y,\{Y\})$.

    In this situation $\K_{\mg{G}}(X,*)=\widetilde{\K}_{\mg{G}}(X)$, $\K^{1,1}_{\mg{G}}(X,*)=\K_{\mg{G}}(X\wedge \operatorname{P}(\mathbb{C}^2))$ and $\beta$ can be written as the composition
    \begin{align}
    \xymatrix{\widetilde{\K}_{\mg{G}}(X) \ar[r]^{\otimes b\hspace{1.5cm}} & \widetilde{\K}_{\mg{G}}(X)\otimes_{\mathbf{R}(G)} \widetilde{\K}_{\mg{G}}(\operatorname{P}(\mathbb{C}^2)) \ar[r]^{\hspace{1cm}*} & \widetilde{\K}_{\mg{G}}(X\wedge \operatorname{P}(\mathbb{C}^2))
    }\end{align}
    By Cor. \ref{onegenerator} the first map is an isomorphism. So we need to show the second map is an isomorphism. By the previous morphism relating the external product and the reduced external product we just have to prove the map 
    \begin{align}
   \K_{\mg{G}}(X)\otimes_{\mathbf{R}(G)} \K_{\mg{G}}(\operatorname{P}(\mathbb{C}^2)) \longrightarrow \K_{\mg{G}}(X\times \operatorname{P}(\mathbb{C}^2))
    \end{align}
    is an isomorphism. But by Thm. \ref{pediocicitytheorem} we have
    \begin{align}
    \K_{\mg{G}}( \operatorname{P}(\mathbb{C}^2))\cong \frac{\MRR[]{\mg{G}}[H]}{([H]-[\textbf{1}])^2}
    \end{align}
    and of course there is a natural isomorphism
    \begin{align}
    \K_{\mg{G}}(X)\otimes_{\mathbf{R}(G)} \frac{\mathbf{R}(G)[H]}{([H]-[\textbf{1}])^2} \cong \frac{\K_{\mg{G}}(X)[H]}{([H]-[\textbf{1}])^2}
    \end{align}
    which commutes with the internal tensor product.
\end{proof}

\begin{definition} Let the positive degree magnetic equivariant K-theory groups be:
\begin{align}
\K^{p}_{\mg{G}}(X,Y):=\K^{p,0}_{\mg{G}}(X,Y) \hspace{1cm} \text{ for }p > 0.
\end{align}
\end{definition}
\begin{corollary}
    There is a natural isomorphism
    \begin{align}
    \K^{p,q}_{\mg{G}}\cong \K^{p-q}_{\mg{G}}
    \end{align}
    for all $p,q\geq 0$.
\end{corollary}
\begin{proof}
    If $p<q$, apply  $p$ times Thm. \ref{1,1periodicity}  and then the definition of the suspension of Eqn. \eqref{suspension K-theory}. If $p\geq q$, apply $q$ times Thm. \ref{1,1periodicity}.
\end{proof}

\subsection{Clifford Modules} 

In this section we introduce bigraded Clifford modules in the context of magnetic groups. These are used in section \ref{sectionofcoeff} to compute the $K$-theory of spheres with or without involution.

Let $\Cliff{p}{q}$ denote the Clifford algebra over $\mathbb{R}^{p,q}$ of the quadratic form
\begin{align}
-\left(\sum_j^qx^2_j+\sum_i^p y_i^2\right)
\end{align}
A more explicit description is the following
\begin{align}
\Cliff{p}{q}=\left\langle 1,e_1,e_2,\ldots,e_q,e_{q+1},e_{q+2},\ldots, e_{p+q} \,\,\Big|
{\begin{tabular}{c}
   $e^2_j=-1$  \\
   $e_je_k=-e_ke_j$  \\
\end{tabular}} \right\rangle
\end{align}

This algebra has a natural action of $(\mg{G},\phi)$ given by
\begin{align}
g\cdot e_j=\begin{cases}
    e_j & \text{ if } 1\leq j \leq q \\
    (-1)^{\phi(g)}e_j & \text{ if } q+1\leq j \leq q+p.
\end{cases} 
\end{align}

\begin{remark}
    There is another natural Clifford algebra over $\mathbb{R}^{q+p}=\mathbb{R}^q\oplus \mathbb{R}^p$ ($=\mathbb{R}^{p,q}$ as vector spaces over the real numbers) induced by the quadratic form 
    \begin{align}
    -\sum_i^{q}x^2+\sum_j^{p}y^2.
    \end{align}
    Denote this Clifford algebra by $C_{p,q}$. More explicitly we have
    \begin{align} \label{Cliffordpq}
    C_{p,q}=\left\langle 1,e_1,e_2,\ldots,e_q,e_{q+1},e_{q+2},\ldots, e_{p+q}\,\,\Big|
{\begin{tabular}{cc}
   $e^2_j=-1$  & $1\leq j\leq q$ \\
   $e^2_j=1$  & $q+1\leq j\leq q+p$ \\
   $e_je_k=-e_ke_j$ & \\
\end{tabular}} \right\rangle.
\end{align}
    Endow $C_{p,q}$  with the \emph{trivial involution} and hence with the trivial action of $\mg{G}$. Later we will compare the categories of Clifford modules over $\operatorname{Cliff}(\mathbb{R}^{p,q})$ and $C_{p,q}$. Denote by $C_q$ the Clifford algebra $C_{0,q}$.
\end{remark}

\begin{definition} \label{magnetic Clifford}
    A \textbf{magnetic $\Cliff{p}{q}[\mg{G},\phi]$-module} is a $\mathbb{Z}_2$-graded complex vector space $M=M_0\oplus M_1$ together with:
    \begin{itemize}
        \item a graded $\mathbb{C}$-linear action of $\Cliff{p}{q}$,
        \item a representation of the magnetic point group $(\mg{G},\phi)$ of degree zero such that
        \begin{align}
        g\cdot(zm)=(g\cdot z)(g\cdot m)
        \end{align}
        for $g\in \mg{G}$, $z\in \Cliff{p}{q}$ and $m\in M$.
    \end{itemize}
Denote by $M^{p,q}(\mg{G},\phi)$ the Grothendieck group of isomorphism classes of $\Cliff{p}{q}[\mg{G},\phi]$-modules.
\end{definition}

In a similar way, for the Clifford  algebra $C_{p,q}$, we can define the magnetic $C_{p,q}[\mg{G},\phi]$-modules as $\mathbb{Z}_2$-graded complex vector spaces $M=M_0\oplus M_1$ together with:
\begin{itemize}
    \item a graded $\mathbb{C}$-linear action of $C_{p,q}$,
    \item a representation of the magnetic point group $(\mg{G},\phi)$ of degree zero such that 
    \begin{align}
    g\cdot(zm)=z(g\cdot m)
    \end{align}
    for $g\in \mg{G}$, $z\in C_{p,q}$ and $m\in M$ .
\end{itemize}
The difference with respect to the $\Cliff{p}{q}[\mg{G},\phi]$-modules is that the action of $\mg{G}$ is trivial on $C_{p,q}$ and not trivial on $\Cliff{p,q}{}$. Denote by $\widehat{M}^{p,q}(\mg{G},\phi)$ the Grothendieck group of the isomorphism classes of $C_{p,q}[\mg{G},\phi]$-modules.

\begin{remark}
    We will use both groups $M^{p,q}(\mg{G},\phi)$ and $\widehat{M}^{p,q}(\mg{G},\phi)$, the first one to produce a well defined morphism to the $K$-theory groups, and the second one to do calculations of the coefficients.
\end{remark}
\begin{example}
    Let $(\mg{G},\phi)=(\mathbb{Z}/2,\id)$ and $M$ be a $\Cliff{p}{q}[\mg{G},\phi]$-module. Then
    \begin{itemize}
        \item the non-trivial element of $\mg{G}=\mathbb{Z}/2$ acts on $M$ in an antilinear way, let us denote such action by $m\longmapsto \overline{m}$,
        \item the Clifford algebra $\Cliff{p}{q}$ acts on $M$ and these two actions interact as follows
        \begin{align}
            \overline{zm}=& 1\cdot (zm)\\ =& 1\cdot(z) 1\cdot (m) \\ =& \overline{z}\cdot \overline{m}
        \end{align}
        where $\overline{z}:= 1\cdot z$ denotes the action of $\mathbb{Z}/2$ on $\Cliff{p}{q}$.
    \end{itemize}
    
    So $M$ is a real $\mathbb{Z}/2$-graded $\Cliff{p}{q}$-module in the sense of \cite{atiyahreal}.
\end{example}

\begin{example}
    Let $(\mg{G},\phi)=(G_0\rtimes \mathbb{Z}/2,\pi_2)$ and $M$ a magnetic $\Cliff{0}{q}[\mg{G},\phi]$-module. Then $M$ is a $\mathbb{Z}/2$-graded complex vector space together with:
    \begin{itemize}
        \item a $\mathbb{C}$-linear graded action of $\Cliff{0}{q}=C_{0,q}=C_q$,
        \item an antilinear action given by the element $(e,1)\in G_0\rtimes \mathbb{Z}/2$ of degree zero, denoted by $m\longmapsto \overline{m}$. This map commutes with the action of the Clifford algebra because
        \begin{align}
            \overline{zm}=& (e,1)\cdot (zm)\\ =& ((e,1)\cdot z)((e,1)\cdot m) \\ =& z\cdot \overline{m} \hspace{1cm}(\Cliff{0}{q} \text{ has the trivial involution})
        \end{align}
        for $z\in\Cliff{0}{q}$ and $m\in M$.
        \item The elements of the form $(g,0)\in G_0\rtimes \mathbb{Z}/2$ act linearly on $M$, $m\mapsto g\cdot m$. This action is of degree zero, commutes with the action of $C_q$, and is such that 
        \begin{align}
            \overline{g\cdot m}=& (e,1)(g,0)\cdot m\\ =& (e,1)(g,0)(e,1)^{-1}(e,1)\cdot m \\ =& (1\cdot g,1)(e,1)\cdot m \\ =& (\overline{g},0)\cdot m \\ =& \overline{g}\cdot \overline{m}
        \end{align}
         where $\overline{g}:=1\cdot g$ denotes the action of $\mathbb{Z}/2$ on $G_0$.
    \end{itemize}
     So $M$ is a real graded $C_q[G_0]$-module in the sense of \cite{atiyahsegalequivariant}.
\end{example}

As in the classical (complex) case \cite{abs}, Real case \cite{atiyahreal}, or Real equivariant case \cite{atiyahsegalequivariant}, there is an \textbf{ Atiyah-Bott-Shapiro homomorphism}
\begin{align}
 \nonumber   \alpha: M^{p,q}(\mg{G},\phi) & \longrightarrow \K_{\mg{G}}(D^{p,q},S^{p,q})\\
    M_0\oplus M_1& \longmapsto (D^{p,q}\times M_0, D^{p,q}\times M_1,\sigma)
\end{align}
where $\sigma: S^{p,q}\times M_0 \longrightarrow S^{p,q}\times M_1$
is given by $\sigma(z,m)=(z,zm)$. Of course, the space $D^{p,q}\times M_i$ is a $\mg{G}$-space and this action commutes with the projection onto the first coordinate 
\begin{align}
\pi_1:D^{p,q}\times M_i\longrightarrow D^{p,q}.
\end{align}
So the vector bundles $D^{p,q}\times M_i \overset{\pi_1}{\longrightarrow} D^{p,q}$ are $(\mg{G},\phi)$-vector bundles and the map $\sigma$ is a morphism of $(\mg{G},\phi)$-vector bundles.

There is a restriction morphism 
\begin{align}
\res: M^{p,q+1}(\mg{G},\phi)\longrightarrow M^{p,q}(\mg{G},\phi)
\end{align}
induced by the inclusion $\mathbb{R}^{p,q}\subset \mathbb{R}^{p,q+1}$. So we have a sequence
\begin{align}
\xymatrix{
M^{p,q+1}(\mg{G},\phi)\ar[r]^{\res} & M^{p,q}(\mg{G},\phi) \ar[r]^{\alpha} & \K_{\mg{G}}(D^{p,q},S^{p,q}) \ar[r] & 0
}
\end{align}
which will be proven to be exact in the case $p=0$.
\begin{lemma}
    The composition $\alpha\circ \res $ is zero.
\end{lemma}
\begin{proof}
If $M=M_0\oplus M_1$ is a $\Cliff{p}{q+1}[\mg{G},\phi]$-module, then as said before we have a morphism 
\begin{align}
\sigma: S^{p,q+1}\times M_0 \longrightarrow S^{p,q+1}\times M_1.
\end{align}
We have the decomposition $S^{p,q+1}=S^{p,q_1}_+\cup S^{p,q+1}_-$ where $S^{p,q+1}_{\pm}=\{(x_1,\ldots,x_{q+1},y_1,\ldots,y_p)\in S^{p,q+1}\, | \, \pm x_{q+1}\geq0\}$ and this decomposition respects the involution. Notice we have a homeomorphism
\begin{align}
  \nonumber  \Psi: S^{p,q+1}_+& \overset{\cong}{\longrightarrow} D^{p,q} \\ (x_1,\ldots,x_{q+1},y_1,\ldots,y_p)& \longmapsto (x_1,\ldots,x_{q},y_1,\ldots,y_p).
\end{align}
Clearly $\Psi$ is a map which commutes with the involution in both spaces. So we get a isomorphism of $\mg{G}$-vector bundles 
\begin{align}
\xymatrix{
D^{p,q}\times M_0 \ar[rr]^{(\Psi\times \id)\circ\sigma \circ(\Psi^{-1}\times \id)} \ar[rd]_{\pi_1} && D^{p,q}\times M_1 \ar[ld]^{\pi_1} \\ &D^{p,q}& 
}
\end{align}
which restricts to the isomorphism $\sigma: S^{p,q}\times M_0 \longrightarrow S^{p,q}\times M_1$.
\end{proof}
So we have an Atiyah-Bott-Shapiro homomorphism
\begin{align}
\xymatrix{M^{p,q}(\mg{G},\phi)/\res (M^{p,q+1}(\mg{G},\phi)) \ar[r]^{\hspace{1cm}\alpha} & \K_{\mg{G}}(D^{p,q},S^{p,q})}.
\end{align}
Now it is time to compare the two types of Clifford modules. 

\begin{lemma}\label{moritaequiv}
    There exist an equivalence of categories 
    \begin{align}
    \left\{ {\begin{tabular}{c}
         Category of  \\
          $\Cliff{p}{q}[\mg{G},\phi]$-modules
    \end{tabular}}\right\} \longleftrightarrow \left\{{\begin{tabular}{c}
          Category of \\ $C_{p,q}[\mg{G},\phi]$-modules
    \end{tabular}}\right\}.
    \end{align}
    In particular we have a natural isomorphism of groups
    \begin{align}
    M^{p,q}(\mg{G},\phi)\cong \widehat{M}^{p,q}(\mg{G},\phi.)
    \end{align}
\end{lemma}
The proof is just a generalization of the Atiyah's proof in \cite{atiyahreal} at the end of section 4.
\begin{proof}
    Let $M$ be a $C_{p,q}[\mg{G},\phi]$-module. We are going to define an action of $\Cliff_{p,q}$ on $M$

    \begin{align}
    \begin{cases}
        [e_j]m:=e_jm & 1\leq j \leq q\\
        [e_j]m:=i(e_jm) & q+1\leq j \leq q+p
    \end{cases}
    \end{align}
    where $[e_j]$ denotes a basis element in $\operatorname{Cliff}(\mathbb{R}^{p,q})$ and $i(e_jm)$ refers to the complex vector space structure of $M$. Then
    \begin{align}
    [e_j] ^2m&= e_je_jm=e^{2}_jm=-m \text{ for } 1\leq j \leq q\\
    [e_j]^2m&= ie_jie_jm=-e^{2}_jm=-m \text{ for } q+1\leq j \leq q+p.
    \end{align}
    So we have a well defined action of $\Cliff{p}{q}$ on $M$. Now let us analyze the behavior of the action of $(\mg{G},\phi)$ with respect to this new action. For $g\in \mg{G}$, $q+1\leq j\leq q+p$ and $m\in M$, we have 
    \begin{align}
        g([e_j]m)=&g(ie_jm) \\ 
        =&(-1)^{\phi(g)}ig(e_jm) \\
        =& (-1)^{\phi(g)}ie_jgm \quad(\text{the action of }(\mg{G},\phi) \text{ commutes with } C_{p,q})\\
        =& (-1)^{\phi(g)}[e_j]gm\\
        =& g[e_j]gm,
    \end{align}
    and for $g\in \mg{G}$, $1\leq j\leq q$ and $m\in M$, we have
    \begin{align}
        g([e_j]m)=&g(e_jm) \\ 
        =& e_jgm\\
        =& g[e_j]gm.
    \end{align}
     So we have a functor
    \begin{align}
    \left\{ {\begin{tabular}{c}
         Category of  \\
          $\Cliff{p}{q}[\mg{G},\phi]$-modules
    \end{tabular}}\right\} \longleftarrow \left\{{\begin{tabular}{c}
          Category of \\ $C_{p,q}[\mg{G},\phi]$-modules
    \end{tabular}}\right\}
    \end{align}
    We can reverse the process, let $M$ be a $\Cliff{p}{q}[\mg{G},\phi]$-module and define the action of $C_{p,q}$ by 

    \begin{align}
    \{e_j\}m=\begin{cases}
        e_jm & 1\leq j\leq q\\
        -ie_jm & q+1 \leq j \leq q+p.
    \end{cases}
    \end{align}

    The negative sign make these functors inverses of one another, in fact:
    \begin{align}
        [\{e_j\}]m=\{e_j\}m&=e_jm \text{ for }1\leq j \leq q\\
        [\{e_j\}]m=i\{e_j\}m=-i^2e_jm&=e_jm \text{ for } q+1\leq j\leq q+p.
    \end{align}
\end{proof}

The importance of this result is that we can pass from algebras with non-trivial involution to algebras with trivial involution. An application is the following theorem which let us compute the $K$-theory of the spheres $S^{p,q}$.

\begin{theorem}
    The following diagram commutes

    \begin{align}
    \xymatrix{M^{p,q+1}(\mg{G},\phi) \ar[rr]^{\res} \ar[d]_{\cong} && M^{p,q}(\mg{G},\phi)  \ar[d]_{\cong} \ar[rr]^{\alpha}&& \K_{\mg{G}}(D^{p,q},S^{p,q}) \ar[rr]&& 0 \\ \widehat{M}^{p,q+1}(\mg{G},\phi) \ar[rr]^{\res} && \widehat{M}^{p,q}(\mg{G},\phi) \ar@{-->}[rru] && &&  
    }
    \end{align}
    Moreover, if $p=0$ then the upper row is exact. 
\end{theorem}
\begin{remark}
    The exactness of the upper row can be proven for every $p$. For this to happen one has to show, for instance, that the groups $\widehat{M}^{p,q}/\res(\widehat{M}^{p,q+1})$ have period 8 on $p$ and $q$.
\end{remark}
\begin{proof}
    The isomorphisms given in Lem. \ref{moritaequiv} are compatible with the restriction homomorphism because we have a commutative diagram of vector spaces
    
    \begin{align}
    \xymatrix{
    \mathbb{R}^{q+1}\oplus i\mathbb{R}^{p} && \mathbb{R}^{q}\oplus i\mathbb{R}^{p}  \ar@{_(->}[ll]\\ \mathbb{R}^{q+1}\oplus \mathbb{R}^{p}  \ar[u]&& \mathbb{R}^q\oplus \mathbb{R}^{p} \ar[u] \ar@{_(->}[ll]
    }
    \end{align}

    where the morphisms $\mathbb{R}^{q}\oplus \mathbb{R}^p\longrightarrow \mathbb{R}^q\oplus i\mathbb{R}^p$ are given by
    \begin{align}
    e_j\longmapsto{\begin{cases}
        ie_j& 1\leq j\leq p\\ e_j & p+1 \leq j\leq p+q.
    \end{cases}}
    \end{align}
    The exactness for $p=0$ will be shown in the following section since we need a particular decomposition of Clifford modules.
\end{proof}
\subsection{Coefficients of Magnetic Equivariant K-theory }\label{sectionofcoeff}
Now we restrict ourselves to the Clifford modules over $C_q$ to make calculations of the $\K_{\mg{G}}$-theory of the spheres $S^q$ without involution. 
We adopt the following notation from \cite{abs}:
    \begin{align}
    M^q_{\mathbb{R}}(\mg{G},\phi):=M^{0,q}(\mg{G},\phi)=\widehat{M}^{0,q}(\mg{G},\phi)
    \end{align}

One defines $M^q_{\mathbb{C}}(\mg{G},\phi)$ and $M^q_{\mathbb{H}}(\mg{G},\phi)$ similarly, replacing $C_q$ by $C_q\otimes_{\mathbb{R}}\mathbb{C}$ and  $C_q\otimes_{\mathbb{R}}\mathbb{H}$ respectively.

Denote by $M^q_{\mathbb{F}}$ the classical Clifford modules defined in \cite{abs} for $\mathbb{F}=\mathbb{R},\mathbb{C},\mathbb{H}$.
\begin{proposition}\label{decomc}
    There exists an isomorphism
    \begin{align}
    M^q_{\mathbb{R}}(\mg{G},\phi)\cong \left(\mathbf{R}(G,\mathbb{R})\otimes M^q_{\mathbb{R}}\right) \oplus \left(\mathbf{R}(G,\mathbb{C})\otimes M^q_{\mathbb{C}}\right) \oplus \left(\mathbf{R}(G,\mathbb{H})\otimes M^q_{\mathbb{H}}\right).
    \end{align}
\end{proposition}

\begin{proof}
    Let $M=M_0\oplus M_1\in M^{q}_{\mathbb{R}}(\mg{G},\phi)$. Consider the following decomposition, similar to the one given in Prop. \ref{decomK}
    \begin{align}
        M_0\oplus M_1 \cong \bigoplus_{V\in \mathbf{Irrep}(G)} \left(V\otimes_{\End_{\mathbf{Rep}(G)}(V)}\Hom_{\mathbf{Rep}(G)}(V,M_0\oplus M_1)\right)
    \end{align}
    By the Schur lemma of Lem. \ref{Schurfreeb}, the set $\Hom_{\mathbf{R}(G)}(V,M_0\oplus M_1)$ is a vector space over the field of real, complex or quaternionic numbers depending on the type of the representation $V$ with trivial action of $(\mg{G},\phi)$, so $\Hom_{\mathbf{R}(G)}(V,M_0\oplus M_1)\in M^q_{\End_{\mathbf{R}(G)}}$.
\end{proof}

\begin{proposition}\label{split}
    There exist an exact sequence
    \begin{align}
    \xymatrix{
    M^{q+1}_{\mathbb{R}}(\mg{G},\phi) \ar[r] & M^q_{\mathbb{R}}(\mg{G},\phi) \ar[r]^{\alpha\hspace{0.5cm}} & \K_{\mg{G}}(D^q,S^{q-1}) \ar[r]& 0
    }
    \end{align}
\end{proposition}

\begin{proof}
    Using Prop. \ref{decomK} and Prop. $\ref{decomc}$ we see the sequence splits in tree sequences
    \begin{align}
    \xymatrix{
    M^{q+1}_{\mathbb{F}} \ar[r] & M^q_{\mathbb{F}} \ar[r]^{\alpha\hspace{0.5cm}} & KF(D^q,S^{q-1}) \ar[r]& 0
    }
    \end{align}
    for $(\mathbb{F},KF)=(\mathbb{R},KO),(\mathbb{C},KU),(\mathbb{H},KSp)$. Atiyah and Segal proved the exactness of the last tree sequences in \cite{atiyahsegalequivariant}. 
\end{proof}
So we have a way to compute the $K$-theory of the point in terms of the Clifford modules

\[\K_{\mg{G}}(D^q, S^{q-1})\cong M^q_{\mathbb
{R}}(\mg{G},\phi)/M^{q+1}_{\mathbb
{R}}(\mg{G},\phi).\]

Now we are going to compute the coefficients of the $\K^*_{\mg{G}}$-theory, where the representations of the magnetic groups will appear.
\begin{theorem}[Coefficients of $\K^*_{\mg{G}}$]\label{coefofK}
    Let $(\mg{G},\phi)$ be a finite magnetic group and denote by $n_{\mathbb{F}}$ the number of isomorphism classes of irreducible representations of $(\mg{G},\phi)$ of real, complex or quaternionic type for $\mathbb{F}=\mathbb{R},\mathbb{C},\mathbb{H}$ . Then the coefficients of the $\K^*_{G}$-theory are
    \begin{align}
    \begin{array}{|c|c|}
         \hline q & \K^{-q}_{G}(*)   \\ \hline
          0& \mathbb{Z}^{n_{\mathbb{R}}}\oplus \mathbb{Z}^{n_{\mathbb{C}}} \oplus \mathbb{Z}^{n_{\mathbb{H}}} \\ 
         1 & (\mathbb{Z}/2)^{n_{\mathbb{R}}} \\
          2 & (\mathbb{Z}/2)^{n_{\mathbb{R}}}\oplus \mathbb{Z}^{n_{\mathbb{C}}} \\
          3 & 0 \\ 
          4& \mathbb{Z}^{n_{\mathbb{R}}}\oplus \mathbb{Z}^{n_{\mathbb{C}}} \oplus \mathbb{Z}^{n_{\mathbb{H}}} \\
           5 &  (\mathbb{Z}/2)^{n_{\mathbb{H}}} \\
           6 & \mathbb{Z}^{n_{\mathbb{C}}}\oplus (\mathbb{Z}/2)^{n_{\mathbb{H}}} \\
           7& 0\\ \hline
    \end{array}
    \end{align}
    and these groups are periodic of period 8 with respect to $q$. Alternatively, we have the canonical isomorphism:
    \begin{align} \label{decomposition K_G}
    \K_G^*(*) \cong KO^*(*)^{\oplus n_\mathbb{R}} \oplus KU^*(*)^{\oplus n_\mathbb{C}} \oplus KSp^*(*)^{\oplus n_\mathbb{H}}        
    \end{align}
\end{theorem}
\begin{proof}
    Using the decomposition of Prop. \ref{decomc} of $M^q_{\mathbb{R}}(\mg{G},\phi)$ we have the following commutative diagram:

    \begin{align}
    \xymatrix{
    M^{q+1}(\mg{G},\phi) \ar[d]^{\res} \ar[r]^{\cong\hspace{1.75cm}} & \bigoplus_{\mathbb{F}\in \{\mathbb{R},\mathbb{C},\mathbb{H}\}}\left(\mathbf{R}(G,\mathbb{F})\otimes M^{q+1}_{\mathbb{F}}\right) \ar[r]^{\hspace{1cm}\cong} \ar[d]^{\oplus \res} & \bigoplus_{\mathbb{F}\in \{\mathbb{R},\mathbb{C},\mathbb{H}\}}\left(\mathbb{Z}^{n_{\mathbb{F}}}\otimes M^{q+1}_{\mathbb{F}}\right)  \ar[d]_{\oplus\res} \\ M^q(\mg{G},\phi) \ar[r]_{\cong\hspace{1.5cm}} & \bigoplus_{\mathbb{F}\in \{\mathbb{R},\mathbb{C},\mathbb{H}\}} \left(\mathbf{R}(G,\mathbb{F})\otimes M^{q}_{\mathbb{F}}\right) \ar[r]_{\hspace{1cm}\cong} & \bigoplus_{\mathbb{F}\in \{\mathbb{R},\mathbb{C},\mathbb{H}\}}\left(\mathbb{Z}^{n_{\mathbb{F}}}\otimes M^{q}_{\mathbb{F}}\right) 
    }
    \end{align}
    so  we can conclude that
    \begin{align}
    M^q(\mg{G},\phi)/\res (M^{q+1}(\mg{G},\phi))\cong \bigoplus_{\mathbb{F}\in \{\mathbb{R},\mathbb{C},\mathbb{H}\}}\left(\mathbb{Z}^{n_{\mathbb{F}}}\otimes M^q_{\mathbb{F}}/\res(M^{q+1}_{\mathbb{F}})\right)=\bigoplus_{\mathbb{F}\in \{\mathbb{R},\mathbb{C},\mathbb{H}\}}\left(M^q_{\mathbb{F}}/\res(M^{q+1}_{\mathbb{F}})\right)^{n_{\mathbb{F}}}.
    \end{align}
    The groups $M^q_{\mathbb{F}}/\res (M^{q+1}_{\mathbb{F}})$ give the complex, real and quaternionic $K$-theory and   have been computed in \cite{abs}, \cite{atiyahreal}. The decomposition follows.
    \end{proof}

\subsection{Twistings and degree shift} \label{section degree shift}

In quantum mechanics the symmetry groups act projectively, therefore it is imperative
that we understand  how this type of actions relate to the magnetic equivariant K-theory groups.
We will focus only on projective actions that come from sign representations of copies of the group $\mathbb{Z}/2$. 

The setup is the following. Start with a magnetic group $(G,\phi)$ and consider a magnetic central extension
$(\widetilde{G}, \widetilde{\phi})$ of $(G,\phi)$ by an abelian group $A \cong (\mathbb{Z}/2)^k$.

 If $\widetilde{G}_0$ denotes the core
of $\widetilde{G}$ we have the following diagram of group extensions:
\begin{align}
\xymatrix{
 A \ar@{=}[r] \ar[d]& A \ar[d] &  \\ 
 \widetilde{G}_0 \ar[r] \ar[d]^{p_0} & \widetilde{G} \ar[r]^{\widetilde{\phi}} \ar[d]^p & \mathbb{Z}/2 \ar@{=}[d]  \\ 
  G_0 \ar[r]  & G \ar[r] ^\phi & \mathbb{Z}/2.  \\
}
\end{align}
Such extensions are classified
by the cohomology group $H^2(G,A)$ but instead of working with the cocycle representing
the extension, we will work with the group extension itself.

\begin{definition} \label{twisted magnetic K-theory}
Let $X$ be a $G$-space, $(\widetilde{G}, \widetilde{\phi})$ a magnetic group
extension of $(G,\phi)$ by the abelian group $A \cong  (\mathbb{Z}/2)^k$ as above, 
and $\chi \in \widehat{A}:=\rm{Hom}(A,U(1))$ an irreducible representation of $A$.

Let ${}^{(\widetilde{G},\chi)}\K_{G}^{*}(X)$, the
\textbf{$(\widetilde{G},\chi)$-twisted magnetic $G$-equivariant K-theory groups of $X$ } be the direct summand
of the magnetic $\widetilde{G}$-equivariant K-theory of $X$ 
\begin{align}
{}^{(\widetilde{G},\chi)}\K_{G}^{*}(X)  \subset \K_{\widetilde{G}}^{*}(X)
\end{align}
generated by magnetic
$(\widetilde{G}, \widetilde{\phi})$-equivariant vector bundles where the action of $A$ on the fibers matches the one of the character $\chi$. 
\end{definition}

\begin{remark}
Since $X$ is a $G$-space, the induced action of $\widetilde{G}$ on $X$ has the property that $A$ acts trivially.
The group $A$ being abelian allows us to decompose any magnetic $(\widetilde{G}, \widetilde{\phi})$-equivariant vector bundle by the action of $A$ on the fibers 
\begin{align}
 \K_{\widetilde{G}}(X) \cong \bigoplus_{\chi \in \widehat{A}} {}^{(\widetilde{G},\chi)}\K_{G}^{*}(X).
\end{align}

\end{remark}

Accordingly to the Def. \ref{twisted magnetic K-theory} we define
 \begin{align}
 {}^{(\widetilde{G},\chi)}\mathbf{Rep}(G), \ {}^{(\widetilde{G},\chi)}\mathbf{R}(G),\ {}^{(\widetilde{G},\chi)}\mathbf{R}(G,\mathbb{F}), \ {}^{(\widetilde{G},\chi)}\mathbf{Irrep}(G,\mathbb{F}),\ {}^{(\widetilde{G},\chi)}\MVec_{\mg{G}}(X) 
 \end{align}
as the category of $(\widetilde{G},\widetilde{\phi})$ representations where $A$ acts by the character $\chi \in \widehat{A}$,
its Grothendieck group, the ones of type $\mathbb{F}$, the irreducible ones, the irreducible ones of type $\mathbb{F}$ and the magnetic $(\widetilde{G},\widetilde{\phi})$-equivariant vector bundles over $X$ where $A$ acts by the character $\chi$, respectively.

\begin{example}
Of particular relevance is the case on which $(G,\phi)=(\mathbb{Z}/2,id)$ and $(\widetilde{G},\widetilde{\phi})=(\mathbb{Z}/4, mod_2)$. Here $A\cong \mathbb{Z}/2$ and we have two characters.  If $X$ is a $\mathbb{Z}/2$-space, any magnetic $(\mathbb{Z}/4, mod_2)$-equivariant
vector bundle over $X$ decomposes as a direct sum of two bundles, one on which $A$ acts trivially, and therefore could be thought as a magnetic $(\mathbb{Z}/2,id)$-equivariant vector bundle, and 
one on which $A\cong \mathbb{Z}/2$ acts by multiplication by $-1$. The first type are known as Atiyah's Real vector bundles and  the second type
are equivalent to what is known in the literature as quaternionic vector bundles (Dupont \cite{dupontsymplectic} calls them symplectic bundles but we will leave the name symplectic
for the case on which the space $X$ has no action of $\mathbb{Z}/2$), 
, hence:
\begin{align}
\K_{\mathbb{Z}/4}^*(X) \cong K\mathbb{R}^*(X) \oplus K\mathbb{H}^*(X),
\end{align}
see Ex. \ref{real and quaternionic bundles}. 
If we denote by $\sigma$ the sign representation of $A=\mathbb{Z}/2$, 
then the $(\mathbb{Z}/4,\sigma)$-twisted magnetic $\mathbb{Z}/2$-equivariant K-theory over $\mathbb{Z}/2$-spaces
is clearly equivalent to quaternionic K-theory:
\begin{align}
{}^{(\mathbb{Z}/4,\sigma)}\K_{\mathbb{Z}/2}^*(X) \cong K\mathbb{H}^*(X).
\end{align}
\end{example}

The degree shift in magnetic equivariant K-theory comes from the 
$\mathbb{Z}/2$-central extension of any magnetic group $(G,\phi)$ defined by the pullback of $\mathbb{Z}/4$ under $\phi$. This particular
choice of magnetic group is 
\begin{align}
(\widehat{G}, \widehat{\phi}) :=(G \times_{\mathbb{Z}/2}\mathbb{Z}/4,mod_2 \circ \pi_2)
\end{align}
 where we have that:
\begin{align} \label{widehatG}
\widehat{G}=\phi^*(\mathbb{Z}/4)=G \times_{\mathbb{Z}/2}\mathbb{Z}/4 = \{(g,m) \in G \times \mathbb{Z}/4  \colon \phi(g)& =mod_2(m) \},\\
\mathrm{and} \ \ \ \widehat{\phi}(g,m)=mod_2 \circ \pi_2(g,m)= mod_2(m),&
\end{align}
and whose core is $\widehat{G}_0 \cong G_0 \times \mathbb{Z}/2$.
Diagrammatically we have:

\begin{align}
\xymatrix{
 G_0 \ar[r] & \mg{G} \ar[r]^{\phi} & \mathbb{Z}/2  \\
 G_0 \ar@{=}[r] \ar[d] \ar[u] & \widehat{G}=G \times_{\mathbb{Z}/2}\mathbb{Z}/4 \ar[r]^{\hspace{1cm}\pi_2} \ar[u]^{\pi_1} \ar[ur]^{\widehat{\phi}} & \mathbb{Z}/4  \ar[u]_{mod_2} \\ 
 \widehat{G}_0=G_0 \times \mathbb{Z}/2  \ar[ur] & \ar[l] \mathbb{Z}/2 \ar@{=}[r] \ar[u] & \mathbb{Z}/2. \ar[u]  \\
}
\end{align}

Denote by $\widehat{\sigma}: =\phi^*\sigma$ the sign representation on $\mathbb{Z}/2 = \mathrm{Ker}(\widehat{G} \to G)$ coming from $\sigma$ on $\mathbb{Z}/2 = \mathrm{Ker}(\mathbb{Z}/4 \to \mathbb{Z}/2)$. We claim the following degree shift isomorphism.

\begin{theorem}[Degree shift] \label{degree shift}
    For $X$ a $G$-space and $(G,\phi)$ a magnetic group, then there is a natural isomorphism
   \begin{align}
\K^{*}_{G}(X)  \stackrel{\cong}{\longrightarrow} {}^{(\widehat{G}, \widehat{\sigma})}\K^{*+4}_{G}(X).
\end{align}
\end{theorem}
\begin{proof}
Let us start denoting the twisted equivariant magnetic K-theory by ${}^{\widehat{G}}\K^{*}_{G}(X)$
since there is only one non-trivial representation of the group $\mathbb{Z}/2 = \mathrm{Ker}(\widehat{G} \to G)$.

We will generalize the proof of the degree shift isomorphism that was shown by Dupont in \cite{dupontsymplectic}
to the equivariant setup. We will follow closely the steps of his original proof.

Take $(B^{3,0},S^{3,0})$  the pair of the unit ball and the unit sphere in $\mathbb{R}^{3,0}$. Take $\mathbb{H}$ the quaternions and denote by $\mathbb{P}(\mathbb{H})$ its complex projectivization together
with the canonical complex bundle $\mathrm{H}_{sp} \to \mathbb{P}(\mathbb{H})$ over it. Note that
the magnetic group $\mathbb{Z}/4$ acts on $\mathbb{H} \cong \mathbb{C} \oplus \mathbb{C}$ via the matrix $\big(\begin{smallmatrix}
  0 & 1\\
  -1 & 0
\end{smallmatrix}\big) \mathbb{K}$
and therefore the complex bundle $\mathrm{H}_{sp} \to \mathbb{P}(\mathbb{H})$ becomes
magnetic $\mathbb{Z}/4$-equivariant over $\mathbb{P}(\mathbb{H})$. The induced action 
of $\mathbb{Z}/4$ on $\mathbb{P}(\mathbb{H})$ is trivial on $\mathrm{Ker}( \mathbb{Z}/4 \to \mathbb{Z}/2)$
and moreover this  induced $\mathbb{Z}/2$ action on $\mathbb{P}(\mathbb{H})$ is free. 

Therefore $\mathbb{P}(\mathbb{H})$ as
a $\mathbb{Z}/2$ space is equivalent to $S^{3,0}$ and $\mathrm{H}_{sp}$ is a $\mathbb{Z}/4$-twisted
magnetic $\mathbb{Z}/2$-equivariant line bundle over $S^{3,0}$:
\begin{align}
[\mathrm{H}_{sp}] \in {}^{\mathbb{Z}/4}\K^0_{\mathbb{Z}/2}(S^{3,0}) = K\mathbb{H}^0(S^{3,0}).
\end{align}

The long exact sequence associated to the pair $(X \times B^{3,0}, X\times S^{3,0})$ delivers by \cite[Cor. 3.8]{atiyahreal} the short exact sequence
\begin{align}
0 \longrightarrow \K_{G}^*(X) \stackrel{\pi^*}{\longrightarrow} \K_{G}^*(X \times S^{3,0}) \stackrel{\delta}{\longrightarrow} \K_{G}^{*+4}(X) \to 0
\end{align}
where $G$ is acting on $B^{3,0}$ via $\phi$, on the product $X \times B^{3,0}$ diagonally,  and
$\pi:  X \times S^{3,0} \to X$ is the projection.

This short exact sequence applied on ${}^{\mathbb{Z}/4}\K_{\mathbb{Z}/2}$ delivers the  sequence \cite[eqn. 5]{dupontsymplectic}
\begin{align}
0 \longrightarrow{}^{\mathbb{Z}/4}\K_{\mathbb{Z}/2}^0(*)\stackrel{\pi^*}{\longrightarrow}{}^{\mathbb{Z}/4}\K_{\mathbb{Z}/2}^0(  S^{3,0}) \stackrel{\delta}{\longrightarrow} {}^{\mathbb{Z}/4}\K_{\mathbb{Z}/2}^{4}(*) \to 0
\end{align}
where $d_0 := \delta \mathrm{H}_{sp}$ is the generator of ${}^{\mathbb{Z}/4}\K_{\mathbb{Z}/2}^{4}(*) = KSp^{4}(*)=\mathbb{Z}$, and $d_0^2 \in \K_{\mathbb{Z}/2}^8(*)=K\mathbb{R}^8(*)$ is also the generator.

Consider now the map
 \begin{align}
 \K_{G}(X) \times {}^{\mathbb{Z}/4}\K_{\mathbb{Z}/2}(B^{4,0},S^{4,0})& \to {}^{\widehat{G}}\K_{G}(X\times B^{4,0},X \times S^{4,0})\nonumber \\
 (E,\delta \mathrm{H}_{sp}) & \mapsto \pi^*_1 E \otimes \pi_2^* \delta \mathrm{H}_{sp} \label{tensor with Hsp}
 \end{align}
where $\pi_1$ and $\pi_2$ are the projections and $\delta \mathrm{H}_{sp}$ denotes the generator
of ${}^{\mathbb{Z}/4}\K_{\mathbb{Z}/2}^{4}(*)$ defined above. Note that the bundle
$\pi^*_1 E \otimes \pi_2^* \delta \mathrm{H}_{sp}$ is $\widehat{G}$-equivariant since
both $E$ and $\delta \mathrm{H}_{sp}$ are also $\widehat{G}$-equivariant, the former 
through the induced action of $G$ and the latter through the induced action of $\mathbb{Z}/4$. Note
moreover that $\mathbb{Z}/2 = \mathrm{ker} (\widehat{G} \to G)$ acts by multiplication by $-1$ on $\pi^*_1 E \otimes \pi_2^* \delta \mathrm{H}_{sp}$ and therefore it belongs to 
the $\widehat{G}$-twisted magentic $G$-equivariant K-theory group.

Summarize the construction of Eqn. \eqref{tensor with Hsp} as the  desired degree shifting map
\begin{align}
\K_{G}^0(X) & \to  {}^{\widehat{G}}\K_{G}^4(X) \nonumber \\
E &\mapsto E \cdot d_0.
\end{align}
This homomorphism is an isomorphism because when carried twice it induces the 8-periodicity isomorphism:
\begin{align}
\K_{G}^0(X)  & \stackrel{\cong}{\to} \K_{G}^8(X) \nonumber \\
 E &  \mapsto E \cdot d_0^2,
\end{align}
where $d_0^2$ is the generator of $\K_{\mathbb{Z}/2}^8(*)$.

\end{proof}
Following the notation of the degree shift isomorphism of Thm. \ref{degree shift},  carrying it twice 
delivers the isomorphism
\begin{align}
{}^{(\widetilde{G}, \widehat{\sigma})}\K^{*+4}_{G}(X) \stackrel{\cong}{\to} 
{}^{(\widehat{\widehat{G}}, \widehat{\widehat{\sigma}}+\widehat{\sigma})}\K^{*+8}_{G}(X),
\end{align}
where $(\widehat{\widehat{G}},\widehat{\widehat{\phi}})$ is the magnetic group
\begin{align}
\widehat{\widehat{G}}=\{(g,m,n) \in G \times \mathbb{Z}/4 \times \mathbb{Z}/4  \colon \phi(g) =mod_2(m) = mod_2(n) \}
\end{align}
with $\widehat{\widehat{\phi}}(g,m,n)=mod_2(n)$, and $\widehat{\widehat{\sigma}}+\widehat{\sigma}$  the irreducible representation of 
\begin{align}
( \mathbb{Z}/2)^2 \cong \langle (1_G,2,0),(1_G,0,2) \rangle = \mathrm{Ker}(\widehat{\widehat{G}} \to G )
\end{align}
that maps both $(1_G,2,0)$ and $(1_G,0,2)$ to $-1$. Then we have the following simple result.

\begin{lemma}
There is a natural isomorphism
\begin{align}
\K^{*}_{G}(X) \stackrel{\cong}{\to} 
{}^{(\widehat{\widehat{G}}, \widehat{\widehat{\sigma}}+\widehat{\sigma})}\K^{*}_{G}(X).
\end{align}
\end{lemma}
\begin{proof}
Take a magnetic $G$-equivariant vector bundle $E$ over $X$ and endow this same vector bundle $E$ but with the following action of $\widehat{\widehat{G}}$. For $e \in E$ and $(g,m,n) \in 
\widehat{\widehat{G}}$ let
\begin{align}
(g,m,n) \cdot e := (-1)^{\frac{m+n}{2}} g \cdot  e.
\end{align}
The action is well defined because $m+n$ is even, and note that both $(1_G,2,0)$ and $(1_G,0,2)$
act by multiplication by $-1$. With this action of $\widehat{\widehat{G}}$ the vector bundle
$E$ is a $(\widehat{\widehat{G}}, \widehat{\widehat{\sigma}}+\widehat{\sigma})$-twisted
magnetic $G$-equivariant vector bundle and therefore we get the following isomorphism of categories
\begin{align}
\mathbf{Vec}_G(X) \stackrel{\cong }{\to } {}^{(\widehat{\widehat{G}}, \widehat{\widehat{\sigma}}+\widehat{\sigma})}\mathbf{Vec}_G(X).
\end{align}
The isomorphism in K-theories follows.
\end{proof}

Therefore the  degree shift homomorphism gives us isomorphisms 
\begin{align}
\K_G^*(X) \stackrel{\cong}{\to}  {}^{(\widehat{G},\widehat{\sigma})}\K^{*+4}_{G}(X) \stackrel{\cong}{\to} \K_G^{*+8}(X)
\end{align}
whose composition $\K_G^*(X) \stackrel{\cong}{\to} \K_G^{*+8}(X)$ is multiplying by $d_0^2$, or alternatively, the canonical isomorphism given by the Morita equivalence of the Clifford algebras $C_{p,q}$ and $C_{p+8,q}$.

An interesting consequence of the degree shift isomorphism of Thm. \ref{degree shift} is the following.

\begin{proposition} \label{4-periodic}
Whenever the short exact sequence $\mathbb{Z}/2 \to \widehat{G} \to G$ splits, namely $\widehat{G} \cong G \times \mathbb{Z}/2$ is a trivial extension of $G$, then the magnetic equivariant K-theory
$\K_{G}^*(X)$ of the $G$-space $X$ is $4$-periodic.
\end{proposition}

\begin{proof}
Let $\alpha: G \to \widehat{G}$ be the splitting and $G \times \mathbb{Z}/2 \to \widehat{G} $, $(g, m) \mapsto \alpha(g) \circ (1_G,m)$ the induced isomorphism. Then any $(\widehat{G},\widehat{\sigma})$-twisted  magnetic$G$-equivariant
vector bundle translates to a magnetic $(G \times \mathbb{Z}/2, sgn)$-equivariant vector bundle
where the $\mathbb{Z}/2$ component acts on the fibers by the sign representation (multiplication by $-1$). Forgetting
this action of $\mathbb{Z}/2$ gives the desired isomorphism
\begin{align}
{}^{(\widehat{G}, \widehat{\sigma})}\K_{G}^*(X) = {}^{(G \times \mathbb{Z}/2, sgn)}\K_G^*(X) \cong \K_G^*(X).
\end{align}
The degree shift isomorphism of Thm. \ref{degree shift} provides the desired periodicity of degree $4$:
\begin{align}
\K_G^*(X) \stackrel{\cong}{\to }   {}^{(\widehat{G}, \widehat{\sigma})}\K_{G}^{*+4}(X)  \stackrel{\cong}{\to} \K_G^{*+4}(X).
\end{align}
\end{proof}

\begin{example}
Take the cyclic magnetic group $\mathbb{Z}/2n$ for $n>1$. Then the extension 
\begin{align}
\widehat{\mathbb{Z}/2n}  = \{ (a,b) \in \mathbb{Z}/2n \times \mathbb{Z}/4 \colon a+b = 0  \ mod_2 \}
\end{align}
splits with isomorphism
\begin{align}
\widehat{\mathbb{Z}/2n} \cong \mathbb{Z}/2n \times \mathbb{Z}/2,  \ \ \  ((a,b) \mapsto (a, a+b).
\end{align}
Therefore the  magnetic $\mathbb{Z}/2n$-equivariant K-theory $\K_{\mathbb{Z}/2n}^*(X)$ is $4$-periodic. 
\end{example}

\subsection{Complexes of vector bundles and Thom isomorphism}

In this section, we elaborate on another familiar construction of the groups $\K_{\mg{G}}(X,Y)$, this time in terms of complexes of $(\mg{G},\phi)$-vector bundles. This formulation is crucial for the computation of the coefficients $\K^*_{\mg{G}}$ and for proving the Thom isomorphism. The construction and proof align with the corresponding results for $KU_{G_0}(X)$, as detailed in Sections 3 and Appendix A of \cite{segal}. The antilinear structure is naturally incorporated into the complex framework, following the same path.

\begin{definition}
    Let $(X,A)$ be a $G$-pair of locally compact spaces. The category $C_{G}(X,A)$ of \textbf{complexes on $X$ acyclic on} $A$ has
    \begin{itemize}
        \item Objects: sequences 
\begin{align}E^* : \ldots \overset{d}{\longrightarrow}E^{i-1}\overset{d}{\longrightarrow} E^i \overset{d}{\longrightarrow} E^{i-1} \overset{d}{\longrightarrow} \ldots
\end{align}
        of $(\mg{G},\phi)$-vector bundles on $X$ such that $E^i=0$ when $|i|$ is large, homomorphisms $d$ such that $d\circ d=0$, and 
\begin{align}
\operatorname{supp}(E^*):=\{x\in X\, |\, \text{the sequence of vector spaces } E^*_x \text{ is not exact}\}
\end{align}
 is a compact subset of $X-A$.
    \item Morphisms:  $f:E^*\longrightarrow F^*$ are sequences of morphisms $f^i:E^i\longrightarrow F^i$ such that $f^id=df^i$.
    \end{itemize}
\end{definition}

The direct sum of vector bundles gives a natural sum of complexes
\[(E^*\oplus F^*)^n=E^n\oplus F^n\]
given by the direct sum of vector bundles on each level.

\begin{definition}

   Denote the set of isomorphism classes of complexes on $X$ acyclic on $A$ by $\Lm_{\mg{G}}(X,A)$. This is naturally a semigroup under the sum of complexes.
\begin{itemize}    
\item Two elements $E^*,F^*\in \Lm_{\mg{G}}(X,A)$ are \textbf{homotopic}, $E^*\simeq F^*$, if there is an element $H^*\in \Lm_{\mg{G}}(X\times I, A\times I)$ such that $E^*=H^*|_{X\times \{0\} }$ and $F^*=H^*|_{X\times \{1\}}$.
    \item Two elements $E^*,F^*\in \Lm_{\mg{G}}(X,A)$ are \textbf{equivalent}, $E^*\sim F^*$, if there are elements $F^*_0,F^*_1\in \Lm_{\mg{G}}(X,X)$ such that $E^*\oplus F^*_0\simeq F^*\oplus F^*_1$ are homotopic. 
\end{itemize}
     
\end{definition}

\begin{theorem}
    There is an isomorphism
    \begin{align}
\Lm_{\mg{G}}(X,A)/\!\sim\,\,\overset{\cong}{\longrightarrow} \K_{\mg{G}}(X,A)
\end{align}
\end{theorem}
\begin{proof}
    The proof is equivalent to the one given in Prop. 3.1 of \cite{segal}. The antilinear action of the magnetic groups do not play an important role. When $X$ is compact and $A=\emptyset$, then the isomorphism is simply $E^*\longmapsto \sum_k (-1)^k E^k$.
\end{proof}
Now we are going to describe the product in terms of complexes.

\begin{definition}
    The product of two complexes $E^*,F^*$ on $X$ acyclic on $A$ is 
    \begin{align}
(E^*\otimes F^*)^n=\oplus_{i+j=n} E^i\otimes F^j.
\end{align}
\end{definition}
We have $\operatorname{supp}(E^*\otimes F^*)=\operatorname{supp}(E^*)\cap \operatorname{supp}(F^*)$, see page 139 of \cite{segal}. This defines a product
\begin{align}
\Lm_{\mg{G}}(X,A)\otimes \Lm_{\mg{G}}(X,B) \longrightarrow \Lm_{\mg{G}}(X,A\cup B).
\end{align}

If $A=B=\emptyset$, then this induces the usual product in $\K_{\mg{G}}(X)$.

The product in $\K_{\mg{G}}(X)$ extends to $\K^*_{\mg{G}}(X)$ making it a graded ring. If $\xi_i\in \K^{-p_i}_{\mg{G}}(X)$ (for $i=1,2$) is represented by a complex $E^*_i\in \Lm_{\mg{G}}(X\times \mathbb{R}^{p_i})$ with compact support, then the product $\xi_1\cdot \xi_2\in \K^{-p_1-p_2}_{\mg{G}}(X)$, is represented by the complex $pr_1^*E^*_1\otimes pr_2^*E^*_2$ on $X\times \mathbb{R}^{-p_1-p_2}$, where $pr_i:X\times \mathbb{R}^{p_1}\times \mathbb{R}^{p_2}\longrightarrow X\times \mathbb{R}^{p_i}$.

Now we define the Koszul complex, the key construction to define the Thom isomorphism.
\begin{definition}
    Let $E$ be a $(\mg{G},\phi)$-vector bundle on $X$ and $s$ a $G$-section of $E$ one can form the \textbf{Koszul complex}
    \begin{align}
\cdots \longrightarrow 0 \longrightarrow \mathbb{C} \overset{d}{\longrightarrow} \Lambda^1 E \overset{d}{\longrightarrow} \Lambda^2 E \overset{d}{\longrightarrow} \cdots
\end{align}
    where $d$ is defined by $d(\xi)=\xi\wedge s(x)$ if $\xi \in \Lambda^iE_x$.
\end{definition}
This complex is acyclic at all points $x$ at which $s(x)\not=0$: Let $x\in X$ and suppose $e_1,e_2,\ldots, e_m:=s(x)\in E_x$ is a basis, then if $v=\sum_{I=(i_1, i_2,\ldots,i_k)} \alpha_I e_{i_1}\wedge e_{i_2}\wedge \ldots \wedge e_{i_k}\in \Lambda^{k}E_x $ is such that $v\wedge s(x)=0$, we have 
\begin{align}
0=&\left(\sum_{I=(i_1, i_2,\ldots, i_k)} \alpha_I e_{i_1}\wedge e_{i_2}\wedge \ldots \wedge e_{i_k}\right)\wedge e_m \\ =& \sum_{I=(i_1, i_2,\ldots, i_k)} \alpha_I e_{i_1}\wedge e_{i_2}\wedge \ldots \wedge e_{i_k}\wedge e_m \\ =& \sum_{I=(i_1, i_2,\ldots, i_k), i_K\not= m} \alpha_I e_{i_1}\wedge e_{i_2}\wedge \ldots \wedge e_{I_k}\wedge e_m.
\end{align}
So $\alpha_I=0$ for every $I=(i_1, i_2,\ldots, i_k)$ such that $I_k\not=m$, that is 
\begin{align}
v=\sum_{I=(i_1, i_2,\ldots, m)} \alpha_I e_{i_1}\wedge e_{i_2}\wedge \ldots \wedge e_{m}\in d(\Lambda^{k-1}E_x).
\end{align}
The most important example of a Koszul complex in our case is the following:

\begin{example} \label{Thom class}
    Let $p:E\longrightarrow X$ be a $(\mg{G},\phi)$-vector bundle, the pull-back $p^*E$ on $E$ has a natural section which is the diagonal map $\Delta:E \longrightarrow E\times_X 
    E=p^* E$

    \begin{align}
\xymatrix{
    p^*E \ar[r] \ar[d]^{p} & E \ar[d]^{p} \\ E \ar[r] \ar@(ul,dl)[u]^{\Delta} & X
    }
\end{align}
    This section is zero on the zero-section of $E$. Let us denote by $\Lambda^*_E$ the Koszul complex of $p^*E$ and $\Delta$. If $X$ is compact, $\Lambda^*_E\in \K_{\mg{G}}(E)$ has compact support and is called the \textbf{Thom class}.
\end{example}

\begin{example}
	Let $V$ be a representation of $(\mg{G},\phi)$. The exterior algebra $\Lambda^*V$ of $V$ induces a Koszul complex $\lambda_V$ over $V$:
	\begin{align}
\lambda_V: \xymatrix{V\times \{0\}\ar[r]^{d}& V\times \mathbb{C} \ar[r]^{d} & V\times \Lambda^1 V \ar[r]^{d}& V\times \Lambda^2 V \ar[r]^{d} & \ldots } 
\end{align}
	where $d(v,w)=v\wedge w$. Of course, this complex is exact on $V\backslash \{0\}$.
\end{example}

\begin{definition}
    Let $p:E\longrightarrow X$ be a $(\mg{G},\phi)$-vector bundle over a compact $G$-space $X$. If $F^*$ is a complex with compact support on $X$ then $p^*F^*$ is a complex on $E$ with support $p^{-1}(\operatorname{supp}(F^*))$, and $\Lambda^*_E\otimes p^*F^*$ is a complex with compact support on $E$. The homomorphism 
    \begin{align}
\varphi_*: \K_\mg{G}(X) \longrightarrow \K_{\mg{G}}(E)
    \end{align}
 induced by the assignment $F^*\longmapsto \Lambda^*_E\otimes p^*F^*$ is called the \textbf{Thom homomorphism}.
\end{definition}
One can replace $X$ by $X\times \mathbb{R}^{p}$ and $E$ by $E\times \mathbb{R}^p$ and induce a Thom homomorphism
\[\varphi_*: \K^{-p}_{\mg{G}}(X) \longrightarrow \K^{-p}_{\mg{G}}(E).\]
The main theorem of this section is the following

\begin{theorem}[Thom isomorphism]\label{tiso}
    The Thom 
 \begin{align}
\varphi_*: \K^{-p}_{\mg{G}}(X) \longrightarrow \K^{-p}_{\mg{G}}(E)
\end{align}
 is an isomorphism for any $(\mg{G},\phi)$-vector bundle $E$ on a compact $G$-space $X$.
\end{theorem}

The proof is based on the one given in \cite{bottandthom} which uses elliptic differential operators. The guideline of the proof is the following. For every representation $V$ of $(\mg{G},\phi)$:

\begin{itemize}
    \item Construct a morphism $\alpha:\K^{*}_{\mg{G}}(X\times V) \longrightarrow \K^*_{\mg{G}}(X)$
    which will be the inverse of the Thom homomorphism $-\otimes\lambda_{V^*}: \K^*_{\mg{G}}(X) \longrightarrow \K^*_{\mg{G}}(X\times V)$. The isomorphism $\lambda_{V^*}$ generalizes the Bott isomorphism given in Thm. \ref{1,1periodicity}.
    \item In the previous isomorphism replace $(\mg{G},\phi)$ by $(\operatorname{U}(n)\times \mg{G}, \phi)$, $X$ by $\mathcal{F}(E)$, the frame bundle associated to $E$, and $V=\mathbb{C}^{\operatorname{dim}E}$ to obtain the Thom isomorphism
    $\K_{\mg{G}}(X) \longrightarrow \K_{\mg{G}}(E)$.
\end{itemize}
The complex, equivariant and Real cases follow these lines, see 3, 4,and 5 of \cite{bottandthom}.

\begin{remark}
    We adopt the proof of \cite{bottandthom}, based on differential operators and their indexes,  for two principal reasons:
    \begin{itemize}
        \item It gives the most general form of the theorem.
        \item The proofs of the Thom isomorphism given in \cite{atiyahktheory}, \cite{segal} and \cite{atiyahreal} assume \emph{every vector bundle is locally the sum  of line vector bundles}. \emph{If $G$ is an abelian compact Lie group, then every $G$-equivariant vector bundle is locally the sum of line bundles}. This assumption is not true for $(\mg{G},\phi)$-vector bundles because abelian magnetic groups $(\mg{G},\phi)$ could have complex or quaternion type irreducible representations which are of complex dimension two; for example $(\mathbb{Z}/4, mod_2)$. 
    \end{itemize}
\end{remark}

\subsection{Generalized Bott isomorphism}
Here we are going to construct an inverse for $\lambda_{V^*}: \K^*_{\mg{G}}(X) \longrightarrow \K^*_{\mg{G}}(X\times V)$. To construct this  we are going to use the following factorization
\begin{align}
\xymatrix{
\K_{\mg{G}}(X\times V) \ar[rr]^{j} \ar@(ur,ul)[rrrr]^{\alpha} && \K_{\mg{G}}(X\times \operatorname{P}(V\oplus \textbf{1})) \ar[rr]^{\hspace{1cm}\operatorname{index}(D)} && \K_{\mg{G}}(X)
}
\end{align}
Let us describe the first map $j$ and its value at $\lambda_V$. Note we have a natural map
\begin{align}
	\PS(V\oplus \textbf{1}) & \longrightarrow  V^+\\ [v,z]& \longmapsto \begin{cases}
v/z & \text{ if }z\not=0\\ + & \text{ if }z=0.
\end{cases}
\end{align}
It descends to a homeomorphism
\begin{align}
\mathcal{H}: \PS(V\oplus \textbf{1})/\PS(V)\longrightarrow V^+
\end{align}
so we have a homomorphism
\begin{align}
\xymatrix{\K_{\mg{G}}(V) \ar[r]^{\mathcal{H}^*\hspace{3cm}} \ar@(dl,dr)[rr]_{j} & \K_{\mg{G}}(\PS(V\oplus\textbf{1})-\PS(V))\cong\K_{\mg{G}}(\PS(V\oplus\textbf{1}),\PS(V)) \ar[r] & \K_{\mg{G}}(\PS(V\oplus\textbf{1})). }
\end{align}
Let us consider the following commutative diagram
\begin{align}
\xymatrix{V^+\times\Lambda^{k}V \ar[d]& \Lambda^k(\textbf{1}\otimes V) \ar[l] \ar[d] & \Lambda^k(H\otimes V) \ar[l] \ar[d] \\ V^+  & \PS(V\oplus\textbf{1})/\PS(V)) \ar[l]_{\mathcal{H}\hspace{1cm}} & \PS(V\oplus \textbf{1}) \ar[l]_{\hspace{1cm}q} }
\end{align}
where $H$ is the dual of the tautological line bundle on $\operatorname{P}(V\oplus \textbf{1})$. It can be demonstrated that this is a pull-back diagram. More is true, it induces a diagram of Koszul complex on each space, so we get
\begin{align}
j(\lambda^*_{V^*})=\sum_i (-1)^iH^{i} \Lambda^i(V^*).
\end{align}

\begin{theorem}\label{bottgeneral}
    For any compact $G$-space $X$ and any representation $V$ of $(\mg{G},\phi)$, multiplication by $\lambda_{V^*}$ induces an isomorphism
    \begin{align}
\K_{\mg{G}}(X)\longrightarrow \K_{\mg{G}}(X\times V).
\end{align}
\end{theorem}
\begin{proof}
    
As we said before, we are going to construct an inverse $\alpha$. By the previous comments we just focus on the construction of the homomorphism $\operatorname{index}(D)$.

As its name suggests, $\operatorname{index}(D)$ is the index of some Fredholm operator. The construction is a generalization of the one given in sections 2 and 3 of \cite{bottandthom}.

Consider the Dolbeaut complex on the smooth manifold $\operatorname{P}:=\operatorname{P}(V\oplus \textbf{1})$ 
\begin{align}
\xymatrix{
0\ar[r] & \Omega^{0,0}(\PS)  \ar[r]^{\overline{\partial}} & \Omega^{0,1}(\PS)  \ar[r]^{\overline{\partial}} & \Omega^{0,2}(\PS)  \ar[r]^{\overline{\partial}} & \Omega^{0,3}(\PS)  \ar[r]^{\overline{\partial}} & \cdots
}
\end{align}

where $\Omega^{a,b}(\PS)$ is the complex vector space of $(a,b)$-differential forms, i.e. sections of the complex vector bundle $\Lambda^aT^*\PS\otimes \overline{\Lambda^bT^*\PS}$, for the details of this complex see page 23 of \cite{PAG}.

The action of $(\mg{G},\phi)$ on $V\oplus \textbf{1}$ induces a smooth action by diffeomorphisms on $P$. So we have an action on $\Omega^{0,b}(\PS)$, the sections of a $(\mg{G},\phi)$-vector bundle. Of course, the action commutes with $\overline{\partial}$. Choose a hermitian metric on $\PS$, by Lem. \ref{metrica} there is an hermitian $(\mg{G},\phi)$-invariant metric on $\Omega^{0,b}(\PS)$, so we can define the formal adjoint operator 
\begin{align}
\overline{\partial}^*: \Omega^{0,b}(\PS) \longrightarrow \Omega^{0,b-1}(\PS)
\end{align}
of $\overline{\partial}$ which also commutes with $(\mg{G},\phi)$; for a detailed construction see page 76 of \cite{Palais}.

Denote by $\Omega^+(\PS)=\oplus_k \Omega^{0,2k}(\PS)$ and $\Omega^-(\PS)=\oplus_k \Omega^{0,2k+1}(\PS)$ and define the Dirac operator

\begin{align}
D=\overline{\partial}+\overline{\partial}^*: \Omega^+(\PS)\longrightarrow \Omega^-(\PS).
\end{align}

Define the Laplace-Beltrami operator $\square=\overline{\partial}\overline{\partial}^*+\overline{\partial}^*\overline{\partial}=D^2$. This last operator is important for two reasons. One is that it allows to prove that $D$ has a well define index and two, is that it provides  a Hodge decomposition of the Dolbeaut complex. 
The subspace of elements $\omega\in \Omega^{0,b}(\PS)$ such that $\square\omega=0$ will be denoted by $B^{0,b}(\PS)$ and is called the space of $(\mg{G},\phi)$-complex harmonic forms. It is easy to show that $\square\omega=0$ if and only if $\overline{\partial}\omega=0=\overline{\partial}^*\omega$.

By standard arguments, the differential operator $\square=D^2$ is elliptic and hence $D$ is elliptic too, see for example chapter V in \cite{Palais}. So we have a well defined index
\begin{align}
\operatorname{index}(D):=\ker D-\operatorname{coker}D\in \mathbf{R}(G)=\K_{\mg{G}}(*).
\end{align}

As in theorem 15.4.2 in \cite{TMAG} we have a Hodge decomposition
\begin{align}
\Omega^{0,b}(\PS)=\overline{\partial}\Omega^{0,b-1}(\PS)\oplus \overline{\partial}^*\Omega^{0,b+1}(\PS)\oplus B^{0,b}(\PS).
\end{align}
From this decomposition we have
\begin{align}
\operatorname{index}(D)\cong \oplus_k H^{2k}(\PS;\mathbb{C})-\oplus_m H^{2m+1}(\PS;\mathbb{C}).
\end{align}

Now if $Q$ is other $(\mg{G},\phi)$-vector bundle over $\PS$, then $D$ has a natural extension $D_Q$ over the Dolbeaut complex with coefficients in $Q$ and we have a natural $(\mg{G},\phi)$-isomorphism
\begin{align}
\ker D_Q\cong & \oplus_k H^{2k}(\PS;\mathcal{O}(Q))\\
\operatorname{coker}D_Q\cong  & \oplus_k H^{2k+1}(\PS;\mathcal{O}(Q)).
\end{align}
Finally if $Q$ is a $(\mg{G},\phi)$-vector bundle over $\PS\times X$, then we can define the index for each restriction $Q|_{\PS\times \{x\}}$ of $(\mg{G},\phi)$-vector bundles $\forall x\in X$ and produce an element 
\begin{align}
\operatorname{index}(D_Q)\in \K_{\mg{G}}(X).
\end{align}
So we have defined a homomorphism
\begin{align}
\operatorname{index}(D): \K_{\mg{G}}(X\times \PS) \longrightarrow \K_{\mg{G}}(X).
\end{align}
The rest of the proof follows the same steps as the one in section 4 of \cite{bottandthom}. 
\end{proof}

\begin{proof}[Proof of Thm. \ref{tiso}]
    We can reduce the case when $p=0$. Let us recall the classical constructions of the frame and associated bundles. 

    Let $p:E\longrightarrow X$ be a $(\mg{G},\phi)$-vector bundle of rank $n$. Let $x\in X$ and consider the set of linear or antilinear isomorphisms
    \begin{align}
\MGL(\mathbb{C}^n,E_x
    ).
\end{align}
    It possesses a left action of $(\mg{G}\times\U(n),\phi)$ given by 
    \begin{align}
(g,h)\cdot(f)(z)=gf(h^{-1}z).
\end{align}
    The \textbf{frame bundle} is
    \begin{align}
\xymatrix{
    \mathcal{F}(E):= \bigsqcup_x \MGL(\mathbb{C}^n,E_x) \ar[d] \\
    X.
    }
\end{align}

    It is a principal fiber bundle over $X$ with structure group $\GL(n,\mathbb{C})$ and an extra action of the magnetic point group $(\mg{G},\phi)$, see for example section 7 of chapter 8 in \cite{Husemoller}. The action of the normal subgroup $\U(n)$ on $\mathcal{F}(E)$ is free and we have an $(\mg{G},\phi)$-isomorphism
    \begin{align}
  \nonumber      \Psi:\mathcal{F}(E)\times_{\U(n)} \mathbb{C}^n & \longrightarrow E \\ [(f,z)]& \longmapsto f(z).
    \end{align}
    This isomorphism is $(\mg{G},\phi)$-equivariant, indeed
    \begin{align}
        g\Psi([f,z])&=g\cdot f(z)\\
        &=((g,1)\cdot f)(z)\\
        &=\Psi(g\cdot [f,z]).
    \end{align}
Thus we have the following commutative diagram of isomorphisms
    \begin{align} \label{diagram Thom iso}
\xymatrix{
    \K_{\mg{G}\times \U(n)}(\mathcal{F}(E)) \ar[rr]^{\cong} && \K_{\mg{G}\times \U(n)}(\mathcal{F}(E)\times \mathbb{C}^n) \\
 \K_{\mg{G}}(X) \ar@{-->}[rr]^{\cong} \ar[u]^{\cong}&& \K_{\mg{G}}(E) \ar[u]^{\cong}
    }
\end{align}
  where the vertical arrows are given by Thm. \ref{freeaction} and the top arrow is given by Thm. \ref{bottgeneral}.
The Thom Isomorphism of Thm. \ref{tiso}
\begin{align}
\varphi_*: \K^{*}_{\mg{G}}(X) \longrightarrow \K^{*}_{\mg{G}}(E)
\end{align}
then follows from the lower horizontal isomorphism of diagram \eqref{diagram Thom iso}.

\end{proof}

\subsection{Atiyah-Hirzebruch Spectral Sequence} \label{section AHSS}

In this section we describe the Atiyah-Hirzebruch spectral sequence (AHSS) for a $G$-CW complex $X$.
 This is an important tool for computing the groups $\K^*_{G}(X)$ in from 
equivariant Bredon cohomology with coefficients in $\K^*_{G}(*)$. 
The following definition is basically due to G. E. Bredon  \cite{bredon}.

\begin{definition}
    Let $G,\phi)$ be a finite magnetic group. A \textbf{$G$-CW-complex} is a CW-complex $X$ of finite type (finite number of cells on each dimension) with an action of $G$ by cellular maps such that 
    \begin{align}
\{x\in X\,|\, gx=x\}
\end{align}
 is a subcomplex of
    $X$ for every $g\in \mg{G}$.
\end{definition} 
A more concrete description is the following: a $G$-CW-complex $X$ is the union of $G$-subspaces $X^{(n)}$ such that
\begin{itemize}
    \item $X^{(0)}$ is a disjoint union of a finite number of orbits $G/H$ and
    \item $X^{(n+1)}$ is obtained from $X^{(n)}$ by attaching a finite number of $G$-cells $G/H\times D^{n+1}$ along attaching $G$-maps $G/H\times S^{n}\longrightarrow X^{(n)}.$
\end{itemize}
The attaching maps are determined by their restrictions $S^n\longrightarrow (X^{(n)})^H$: Denote by $\Map_{G}(X,Y)$ the set of $G$-equivariant maps from $X$ to $Y$ and by $X^H$ the subspace of fixed points of $X$ under the action of $H\leq G$. Then
\begin{align}
    \Map_{\mg{G}}(\mg{G}/H\times S^n, X^{(n)}) & \longrightarrow \Map\left(S^n, \left(X^{(n)}\right)^H\right) \\ f& \longmapsto \widehat{f}(x):=f(H,x)\\
\end{align}
and
\begin{align}
\Map\left(S^n, \left(X^{(n)}\right)^H\right)& \longrightarrow \Map_{\mg{G}}(\mg{G}/H\times S^n, X^{(n)}) \\ f& \longmapsto \widetilde{f}(gH,x):=g\cdot f(x)
\end{align}
are inverses one of the the other.

Thus we have 
\begin{align}
X^{(0)}=\sqcup_i \mg{G}/H_{D^0_i}
\end{align}
and
\begin{align}
\xymatrix{ \sqcup_j \mg{G}/H_{D^n_j}\times S^{n-1}_j \ar[r] \ar[d]& X^{(n-1)} \ar[d]\\  \mg{G}/H_{D^n_j}\times D^n_j \ar[r] & X^{(n)}. }
\end{align}
Now we see the groups $\widetilde{\K}^*_{G}(X)$ form a $G$-cohomology theory in the sense of T. Matumoto \cite{matumoto}.
\begin{definition}
    A \textbf{reduced $G$-cohomology theory} on the category of $G$-finite $G$-CW complexes with base point and base point preserving $G$-homotopy classes of base point preserving $G$-maps is a sequence of contravariant functors $\{\widetilde{h}^n_{G}\}_{n\in \mathbb{Z}}$ into the category of abelian groups, together with natural transformations 
\begin{align}
\partial^n: \widetilde{h}^n_{G}(A)\longrightarrow \widetilde{h}^{n+1}_{G}(X,A)
\end{align}
for every $G$-pair $(X,A)$
 satisfying the following axioms
    \begin{itemize}
        \item The inclusion $(X,X \cap Y) \to (X \cup Y, Y)$ induces an isomorphism
\begin{align}
\widetilde{h}^{n}_{G} (X \cup Y, Y) \stackrel{\cong}{\longrightarrow}  \widetilde{h}^{n}_{G}(X,X \cap Y)
\end{align}
        \item If $(X,A)$ is a pair of $G$-spaces the following long sequence is exact 
\begin{align}        
\cdots \longrightarrow  \widetilde{h}^{n-1}_{G}(A) \stackrel{\partial^{n-1}}{\longrightarrow }\widetilde{h}^n_{G}(X,A)\longrightarrow \widetilde{h}^n_{G}(X) \longrightarrow \widetilde{h}^n_{G}(A) 
\stackrel{\partial^{n}}{\longrightarrow } \widetilde{h}^{n+1}_{G}(X,A) \longrightarrow  \cdots.
\end{align}
    \end{itemize}

\end{definition}
\begin{theorem}\label{cohomologytheory}
    Let $(\mg{G},\phi)$ be a finite magnetic group. The magnetic $(\mg{G},\phi)$-equivariant K-theory
$\widetilde{\K}^*_{G}$
 is a reduced $G$-cohomology theory.
\end{theorem}
\begin{proof}
  The isomorphism of the inclusion follows from the isomorphism $\K_G(X,Y) \cong \K_G(X \slash Y)$ shown in Eqn. \eqref{K(X/Y)=K(X,Y)} and the exactness of the long sequence is shown in Lem. \ref{exactade3}.
\end{proof}

\begin{remark}
        As T. Matumoto points out in p. 54 of \cite{matumoto}, this reduced $G$-cohomology theory produces a \textbf{$G$-cohomology theory} by $\K^n_{\mg{G}}(X,A):=\widetilde{\K}^n_{\mg{G}}(X/A)$.
\end{remark}
\begin{theorem}[Mayer-Vietoris long exact sequence]\label{mayervietoris}
    Let $X$ be a $G$-CW-complex and $A,B$ two sub $G$-CW-complexes such that $X=A\cup B$ and $A\cup B$ is also a sub $G$-CW-complex. Then there exist a long exact sequence
    \begin{align}
\xymatrix{
    \ldots \ar[r] & \K^n_{\mg{G}}(X) \ar[r] & \K^{n}_{\mg{G}}(A)\oplus \K^{n}_{\mg{G}}(B) \ar[r] & \K^{n}_{\mg{G}}(A\cap B) \ar[r] & \K^{n+1}_{\mg{G}}(X) \ar[r]& \ldots.
    }
\end{align}
\end{theorem}

\begin{proof}
    Consider the following commutative diagram 

    \begin{align}
\xymatrix{
    \K^n_{\mg{G}}(X) \ar[r]^{a} \ar[d]_{b}& \K^n_{\mg{G}}(A) \ar[r] \ar[d]_{a'} & \K^n_{\mg{G}}(X,A) \ar[r]^{r} \ar[d]^{\cong}_{k} & \K^{n+1}_{\mg{G}}(X) \ar[r] \ar[d] & \K^{n+1}_{\mg{G}}(A) \ar[d] \\ \K^n_{\mg{G}}(B) \ar[r]_{b'} & \K^n_{\mg{G}}(B\cap A) \ar[r]_{c} & \K^n_{\mg{G}}(B,B\cap A) \ar[r] & \K^{n+1}_{\mg{G}}(B) \ar[r] & \K^{n+1}_{\mg{G}}(B\cap A)
    }
\end{align}
    where the horizontal sequences are the long exact sequences of the pairs $(X,A)$ and $(B,B\cap A)$, respectively, and the vertical arrows are given by the natural inclusions. The middle vertical arrow is an isomorphism because we have the natural homeomorphism $X/A\cong B/B\cap A$. So we have the long sequence
    \begin{align}
\xymatrix{
    \ldots \ar[r] & \K^n_{\mg{G}}(X) \ar[r]^{(a,-b)\hspace{1cm}} & \K^{n}_{\mg{G}}(A)\oplus \K^{n}_{\mg{G}}(B) \ar[r]^{\hspace{0.5cm}a'+b'} & \K^{n}_{\mg{G}}(A\cap B) \ar[r]^{r k^{-1}c} & \K^{n+1}_{\mg{G}}(X) \ar[r]& \ldots
    }
\end{align}
    Standard arguments in diagram chasing shows that this sequence is exact.
\end{proof}
Now we start to construct the Atiyah Hirzebruch Spectral Sequence (AHSS). The filtration by
$G$-skeletons 
\begin{align}
X^{(0)}\subset X^{(1)}\subset \ldots \subset X^{(n)}\subset \ldots \subset X
\end{align}
induces the following diagram
\begin{align} \label{diagram K-theory AHSS}
\xymatrix{
& & \K^{*}_{\mg{G}}(X^{(n+1)},X^{(n)})\ar[d]^{j^*_{n+1}} &  \K^{*}_{\mg{G}}(X^{(n)},X^{(n-1)})\ar[d]^{j^*_{n}} & \\
  \K^{*}_{\mg{G}}(X)  \ar[r] & \cdots \ar[r]^{i^*_{n+1}} & \K^{*}_{\mg{G}}(X^{(n+1)}) \ar[r]^{i^*_{n}} \ar[d]^{\partial^{n+1}} & \K^{*}_{\mg{G}}(X^{(n)}) \ar[r]^{i^*_{n-1}} \ar[d]^{\partial^n} & \cdots \\
& &  \K^{*+1}_{\mg{G}}(X^{(n+2)},X^{(n+1)}) &  \K^{*+1}_{\mg{G}}(X^{(n+1)},X^{(n)})&
}
\end{align}
where $i^*_{n}$ and $j^*_{n}$ are induced by the inclusions $X^{(n)}\longrightarrow X^{(n+1)}$ and $(X^{(n)},\emptyset)\longrightarrow (X^{(n)},X^{(n-1)})$, respectively, and $\partial^n$ is the connecting homomorphisms\ from the long exact sequence of the pair $(X^{(n+1)},X^{(n)})$.

If we add over $n$  we obtain the exact triangle
\begin{align}
\xymatrix{
   \oplus_n \K^{*}_{\mg{G}}(X^{(n)}) \ar[rr]^{i=\oplus i^*_n} && \oplus_n \K^{*}_{\mg{G}}(X^{(n)}) \ar[dl]^{k=\oplus_n \partial_n} \\  &E_1=\oplus_n \K^{*}_{\mg{G}}(X^{(n)},X^{(n-1)}) \ar[ul]^{j=\oplus_n j^{*}_n}&
    }
\end{align}
By standard argument in the construction of spectral sequences, see for example \cite{GenCohom}, we obtain an spectral sequence $(E_r,d_r)_{r\in \mathbb{N}}$ whose
elements of the first page are as follows:
\begin{align}
  E_1^{n,t}=& \K^{n+t}_{\mg{G}}(X^{(n)},X^{(n-1)}) \\ \cong&\widetilde{\K}_{\mg{G}}^{n+t}(X^{(n)}/X^{(n-1)}) \text{ Definition} \\\cong& \widetilde{\K}_{\mg{G}}^{n+t}\left(\bigvee_j\,\left( (\mg{G}/H_{D^n_j})^+\wedge S^n_j\right)\right) \text{ Collapse de }n-1 \text{skeleton }\\ \cong & \bigoplus_j \widetilde{\K}_{\mg{G}}^{n+t}\left((\mg{G}/H_{D^n_j})^+\wedge S^n_j\right) \text{Wedge axiom} \\ =& \bigoplus_j \widetilde{\K}_{\mg{G}}^{n+t}\left(\Sigma^n (\mg{G}/H_{D^n_j})^+\right) \text{ Definition of suspension}  \\ \cong & \bigoplus_j \K_{\mg{G}}^{t}(\mg{G}/H_{D^n_j}) \text{ Suspension isomorphism} \\ \cong & \bigoplus_j
  \begin{cases}
      \K_{H_{D^n_j}}^{t}(*) & \text{ if } H_{D^n_j}\not\leq G_0\\ KU_{H_{D^n_j}}^{t}(*) & \text{ if } H_{D^n_j}\leq G_0
  \end{cases} \text{Lem.}\ \ref{cocientes} \\
  = & C^n_{\mg{G}}\left(X;\K_{\mg{G}}^{t}\right).
\end{align}
\begin{remark}\label{disofdifferentials}
    Here $C^n_{\mg{G}}\left(X;\K_{\mg{G}}^{t}\right)$ denotes the \textbf{$n$-dimensional $\mg{G}$-cochain group of the $\mg{G}$-CW-complex $X$ with local coefficients in $\K^t_{\mg{G}}$}, i.e.
    \begin{align}
C^n_{\mg{G}}\left(X;\K_{\mg{G}}^{t}\right):= \bigoplus_j \K^t_{\mg{G}}(\mg{G}/H_{D^n_j}). 
\end{align}
These groups and their coboundary operators were defined by G. E. Bredon in \cite{bredon}. We will describe such coboundary operators following section 2 of \cite{willsonequivariant} by dualizing the constructions given there to reproduce the cohomological counterpart. 
For each integer $n\geq0$ let the equivariant $n$-cells of $X$  be indexed by a set $B_n$; if $b\in B$ then the corresponding cell is $\mg{G}/H_b\times D^n$ adjoined along the $G$-map $f_b: \mg{G}/H_b\times \partial D^n\longrightarrow X^{(n-1)}$. 
We shall define \begin{align}
\delta:C^n_{\mg{G}}\left(X;\K_{\mg{G}}^{t}\right)=\bigoplus_{b\in B_n} \K^t_{\mg{G}}(\mg{G}/H_{D_b}) \longrightarrow C^{n+1}_{\mg{G}}\left(X;\K_{\mg{G}}^{t}\right)=\bigoplus_{c\in B_{n+1}} \K^t_{\mg{G}}(\mg{G}/H_{D_c})
\end{align}
by defining the restrictions 
\begin{align}
\delta_{b,c}:\K^t_{\mg{G}}(\mg{G}/H_b)\longrightarrow \K^t_{\mg{G}}(\mg{G}/H_c)
\end{align}
 for $b\in B_n$ and $c\in B_{n+1}$.
\begin{itemize}
\item If $n=0$, then $D^{n+1}=[0,1]$: Let the image of $f_c|_{\mg{G}/H_c\times \{1\}}$ lie in $\mg{G}/H_{b_1}\times D^0$ and let the image of $f_c|_{\mg{G}/H_c\times \{0\}}$ lie in $\mg{G}/H_{b_0}\times D^0$.
 Then $f_c|_{\mg{G}/H_c\times \partial D^1}$ induces a $G$-map $f_{c,i}: \mg{G}/H_c\longrightarrow \mg{G}/H_{b_i}$ for $i=0,1$. Define for $s\in \K^t_{\mg{G}}(\mg{G}/H_b)$,
\begin{align}
 \delta_{b,c}(s)=\begin{cases}
(-1)^{i+1}\K^t_{\mg{G}}(f_{c,i})(s) & \text{ if } b=b_i \text{ and }b_0\not= b_1\\ \K^t_{\mg{G}}(f_{c,1})(s)-\K^t_{\mg{G}}(f_{c,0})(s) & \text{ if } b=b_0=b_1\\0 & \text{ if } b\not=b_0 \text{ and } b\not=b_1.\\ 
\end{cases}
\end{align}
\item If $n\geq 1$, we observe that the composition 
\begin{align}
k: \mg{G}/H_c\times D^{n+1}\overset{f_c}{\longrightarrow} X^{(n)}\overset{p}{\longrightarrow}(\mg{G}/H_b\times D^n)/(\mg{G}/H_b\times \partial D^n)
\end{align}
is a well-defined $G$-map where $p$ is the canonical projection. It is easily seen, via Theorem 2.1 of \cite{willsonequivariant}, that $k$ may be $\mg{G}$-homotoped to a map $\overline{k}$ such that the induced orbit map $\overline{k}/\mg{G}: \partial D^{n+1}\longrightarrow D^n/\partial D^n$ is transverse regular at $0\in D^n/\partial D^n$. Hence $\overline{k}^{-1}(\mg{G}/H_b\times \{0\})$ is $\mg{G}/H_c\times \{x_1,x_2,\dots,x_m\}$ for some finite set of points $x_i\in \partial D^{n+1}$. For each $i$, $i=1,2,\ldots,m$, let
\begin{align}
\epsilon_i=\begin{cases}
+1 & \text{ if } \overline{k}/\mg{G} \text{ preserves orientation near }x_i\\ -1 & \text{ if } \overline{k}/\mg{G} \text{ reverses orientation near }x_i.
\end{cases}
\end{align}
 Let $k_i: \mg{G}/H_c\times \{x_i\}\longrightarrow \mg{G}/H_b\times \{0\}$ be the $G$-map induced from $\overline{k}$. For $s\in \K^t_{\mg{G}}(\mg{G}/H_b)$ define
\begin{align} 
\delta_{b,c}(s):=\sum_{i=1}^m\epsilon_i \K^t_{\mg{G}}(k_i)(s).
\end{align}
\end{itemize}
 The cohomology of this cochain complex is denoted by 
\begin{align}
H^n_{\mg{G}}(X;\K^t_{\mg{G}}):=H^n\left((C^*(X;\K^t_{\mg{G}}),\delta)\right).
\end{align}
\end{remark}

It is clear that the first differential $d_1:E^{n,t}_1\longrightarrow E^{n+1,t},_1$ can be identified with the coboundary operator $\delta$ of the CW-complex cochain $\left(C^*_{\mg{G}}\left(X;\K_{\mg{G}}^{t}\right),\delta\right)$, thus the second page of the spectral sequence is isomorphic to
\begin{align}
E_2^{n,t}\cong H^n\left(X;\K_{\mg{G}}^{t}\right).
\end{align}
The filtration by skeletons also induces a filtration of $\K^*_{\mg{G}}(X)$ given by
\begin{align}
F^n\K^*_{\mg{G}}(X):=\ker\left(\K^*_{\mg{G}}(X)\longrightarrow \K^*_{\mg{G}}(X^{(n-1)})\right)
\end{align}
satisfying
\begin{align}
\ldots \subset F^{n+1}\K^*_{\mg{G}}(X) \subset  F^{n}\K^*_{\mg{G}}(X)\subset \ldots \subset \K^*_{\mg{G}}(X).
\end{align}

Summarizing we have the following result.
\begin{theorem}[Atiyah-Hirzebruch Spectral Sequence]\label{ahss}
    Let $(\mg{G},\phi)$ be a finite magnetic point group and $X$ a $G$-CW-complex. Then the Atiyah-Hirzebruch spectral sequence,
    \begin{align}
E^{n,t}_2\cong H^n\left(X;\K^t_{G}\right) \Longrightarrow \K^*_{G}(X),
\end{align}
    is a right half-plane spectral sequence with differentials $d_r:E^{n,t}_r\longrightarrow E^{n+r,t-r+1}_r$ that converges conditionally to the $\operatorname{lim}_n \K^*_{G}(X^n)$.
\end{theorem}

\begin{proof}
    See theorem 12.4 in \cite{convergence} for the conditional convergence issues.
\end{proof}

\begin{remark}
    For most of the physical applications $X$ has a finite number of cells so the spectral sequence converges to $\K^{*}_{G}(X)$.
\end{remark}

%%%%%%%%%%%%%%%%%%%%%%%%%%%%%%%%%%%%%%%%%%%%%%%%%%%%%%%

%%%%%%%%%%%%%%%%%%%%%%%%%%%%%%%%%%%%%%%%%%%%%%%%%%%%%%%

\section{Applications in Topological Phases of Matter}

Among the many interesting applications of K-theory in physics, one that has shown to be very influential is the incorporation of K-theory groups into the classification of gapped phases of interacting fermions \cite{Bellissard, ktheoryofC}.
The ten-fold way classification presented  in the periodic table for topological insulators and superconductors by Kitaev \cite{kitaev} after the work of Altland and Zirnbauer \cite{Altland}, paved the way for the immersion of real K-theoretical
invariants in the study and classification of crystals \cite{tenfoldway-survey, twistedmat}.

The approach to study gapped interaction of fermions through non-commutative geometric tools \`a la Connes \cite{Connes} was brilliantly initiated by Bellissard, van Elst and Schulz‐Baldes \cite{Bellissard}  and further developed by many others (see \cite{Prodan-Schulz-Baldes} and the references therein).
It was the work of Freed and Moore \cite{twistedmat} that organized most of the previous information on K-theory
invariants noting that the more general description needed incorporated not only equivariant real K-theory but moreover its twisted versions. Gomi went further  in \cite{Gomi2017FreedMooreK}, and using the language of groupoid, presented the
generalization of several properties of the real equivariant K-theory to the more general magnetic case.

The present work arose as attempt to provide an alternative description of the properties of the magnetic equivariant
K-theory that is better suited for the applications in condensed matter physics. We hope to have accomplished
this task in the previous sections. Now we are going to present some non-trivial calculations of magnetic equivariant K-theory groups that are of relevance in the study of altermagnets \cite{Altermagnetism1,Altermagnetism2,Altermagnetism3}. We believe that these examples
are interesting enough to justify the importance of presenting  the properties of the magnetic equivariant K-theory in an organized manner, and moreover, that they are generically enough to allow the reader to be able to mimic our calculations to other similar systems.

Before presenting the calculations we need to present the symmetries that appear in systems with spin-orbit coupling.

\subsection{Symmetries under spin-orbit interaction}

The magnetic groups that we want to consider are called magnetic point groups in condensed matter physics. 
Moreover, we need to consider magnetic point groups in the spin-orbit interaction setup.

Let us review quickly their definition. Crystals symmetries are three dimensional rigid motions that leave the lattice structure of the crystal fixed; this is the space group of symmetries. The time reversal symmetry $\mathbb{T}$ is always present in these materials.

 If the crystal has internal magnetizations, the time reversal symmetry is broken, but combinations of time reversal symmetry with rigid motions may preserve the lattice structure of the magnetic material.
  This group of symmetries is called the magnetic space group. 
  
  Using the language of Lie groups we have that the rigid motions are $\mathbb{R}^3 \rtimes \mathrm{O}(3)$ and the time reversal symmetry defines a group of order 2. The magnetic space group  $\mathcal{G} $ is thus a subgroup of $[\mathbb{R}^3 \rtimes \mathrm{O}(3) ]\times \mathbb{Z}/2$ that defines the short exact sequence
  of groups that appear on the second row of the following diagram:
  \begin{align}
  \xymatrix{
  \mathbb{R}^3 \ar[r] &[  \mathbb{R}^3 \rtimes \mathrm{O}(3)] \times \mathbb{Z}/2 \ar[r] & \mathrm{O}(3) \times \mathbb{Z}/2    \\
 \Lambda \ar[r] \ar@{^{(}->}[u] & \mathcal{G} \ar[r] \ar@{^{(}->}[u]& G. \ar@{^{(}->}[u]
  }
  \end{align}
  Here the copy of $\mathbb{Z}/2$ means the group generated by time reversal symmetry, $\Lambda \cong \mathbb{Z}^3$ is the group of  translational symmetries of the crystal and the quotient
  $\mathcal{G}/ \Lambda = G$ is known as the magnetic point group. The map $\phi : G \to \mathbb{Z}/2$
  is canonically induced from the composition $G \to \mathrm{O}(3) \times \mathbb{Z}/2 \stackrel{\pi_2}{\to} \mathbb{Z}/2$; this homomorphism separates the elements in $G$ which are built from compositions of time reversal with rigid symmetries from the rigid symmetries alone.
  
  The magnetic $\mathcal{G}$-equivariant vector bundles over $\mathbb{R}^3$ are the ones of interest in
  condensed matter physics since they represent the occupied states of a gapped Hamiltonian. Therefore
  the topological invariants coming from magnetic $\mathcal{G}$-equivariant K-theory are important.

  The magnetic space group $\mathcal{G}$ acts on $\mathbb{R}^3$ and this action induces
  an action of the magnetic point group $G$ on $\mathbb{R}^3/\Lambda \cong (S^1)^3$.
  Whenever the group of crystal symmetries is {\it symmorphic}, namely that the extension $\mathcal{G} \to G$ splits, then we have an isomorphism between the magnetic $\mathcal{G}$-equivariant K-theory of $\mathbb{Z}^3$ and the magnetic $G$-equivariant K-theory of $(S^1)^3$:
  \begin{align}
  \K_{\mathcal{G}}^{*}(\mathbb{R}^3) \cong   \K_{G}^{*}((S^1)^3).
  \end{align}
  
This K-theory group $\K_{G}^{*}((S^1)^3)$ provides topological invariants for the structure of the energy bands in a gapped hamiltonian.

  When studying electronic properties of magnetic materials the spin of the electron needs to be taken into account. The appropriate group of point symmetries becomes the double cover $\mathrm{Pin}_{-}(3)$ of $\mathrm{O}(3)$,
  which is isomorphic to $\mathrm{SU}(2) \times \mathbb{Z}/2$. Here the inversion symmetry $\mathbf{r} \mapsto - \mathbf{r}$ in $\mathrm{O}(3)$ lifts to a symmetry of order two in $\mathrm{Pin}_{-}(3)$, and therefore all mirror symmetries in 
  $\mathrm{O}(3)$ lift in $\mathrm{Pin}_{-}(3)$ to symmetries that square to $-1$.
  
  The time reversal operator, in the presence of spin orbit interaction, squares to $-1$, and this $-1$ must agree with the $-1$ of the square of the lifts of the mirror symmetries. In general 
  the time reversal operator is taken to be $i \sigma_2 \mathbb{K}$ where $\sigma_2$ is the second Pauli matrix and $\mathbb{K}$ is complex conjugation. Therefore the group of point symmetries in the presence of spin orbit interaction is the quotient group
  \begin{align}
 \left( [\mathrm{SU}(2) \times \mathbb{Z}/4 ] / \Delta \right)\times \mathbb{Z}/2
  \end{align}
where $\Delta:= \langle -\mathrm{Id},2 \rangle \cong \mathbb{Z}/2$.

If we consider the group $\mathbb{Z}/4$ as the one generated by the operator $i \sigma_2 \mathbb{K}$,
we see that $(i \sigma_2 \mathbb{K})^2=-\mathrm{Id}$ and therefore it is clear why we need to quotient by the subgroup $\Delta$.

Incorporating the previous information in the following diagram:
\begin{align} \label{SOC extension}
\xymatrix{
\mathbb{Z}/2 \ar[r]  &  \left( [\mathrm{SU}(2) \times \mathbb{Z}/4 ] / \Delta \right)\times \mathbb{Z}/2 \ar[r] & \mathrm{SO}(3) \times \mathbb{Z}/2 \times \mathbb{Z}/2\\
\mathbb{Z}/2 \ar[r] \ar@{=}[u] &  \widetilde{G} \ar[r] \ar@{^{(}->}[u] & G \ar@{^{(}->}[u]
}
\end{align}
we see that the group $ \left( [\mathrm{SU}(2) \times \mathbb{Z}/4 ] / \Delta \right)\times \mathbb{Z}/2 $ defines a $\mathbb{Z}/2$-central extension of $\mathrm{O}(3) \times \mathbb{Z}/2$ (here we are
using the canonical isomorphism $\mathrm{O}(3) \cong \mathrm{SO}(3) \times \mathbb{Z}/2$)
where $\mathbb{Z}/2$ acts by multiplication by $-1$, and $\widetilde{G}$ is the induced extension over $G$.

The group $\widetilde{G}$ incorporates the symmetries that interact accordingly with the spin, and therefore
this group is the one that will act on the fibers of the magnetic vector bundles. Of particular importance is that the kernel of the homomorphism $\widetilde{G} \to G$ acts on the fibers by the sign representation.

\begin{definition} \label{def soc group}
Let $G \subset \mathrm{O}(3) \times \mathbb{Z}/2$ be a magnetic point group. Denote by $\widetilde{G}^{soc}:= \widetilde{G}$ the $\mathbb{Z}/2$-central extension of $G$ defined in diagram \eqref{SOC extension} and call it
the {\it SOC group} of $G$. 
\end{definition}

In the presence of spin orbit interaction, in a crystal whose groups of symmetries is symmorphic, the 
 topological invariant of interest is therefore the $\widetilde{G}^{soc}$-twisted $G$-equivariant K-theory of $(S^1)^3$.
 \begin{definition} \label{definition twisted soc}
 Let $\mathcal{G}$ be a  symmorphic space magnetic group and $G$ its magnetic point group.
 Denote by ${}^{\widetilde{G}^{soc}}\K_G((S^1)^3)$ the $\widetilde{G}^{soc}$-twisted $G$-equivariant K-theory groups of the torus $(S^1)^3$, where the kernel of the homomorphism $\widetilde{G}^{soc}\to G$
 act by multiplication by $-1$.
 \end{definition}

Particular choices of magnetic point groups and actions are of importance in understanding topological invariants of magnetic crystals. In what follows we will only concentrate in the magnetic groups that
are generated by a composition of a four-fold rotation $C_4$ around the $z$-axis and time reversal $\mathbb{T}$, in both the spin and no spin orbit interaction, and in $\mathbb{R}^2$ and $\mathbb{R}^3$.

\subsection{2-dimensional torus with  $C_4\mathbb{T}$ symmetry}

Let us consider the 2-dimensional torus $T^2=S^1\times S^1$ with periodic coordinates $(x,y)$ where
 the 4-fold rotation  action is $C_4(x,y)=(-y,x)$ and the time reversal operator acts by inverting the coordinates
 $\mathbb{T}(x,y)=(-x,-y)$.
 The composition of both operations give the following action:
 \begin{align}
 C_4 \mathbb{T} (x,y) = (y,-x).
 \end{align}
 
 We will be interested in the magnetic cyclic group 
 \begin{align}
 G:= \langle C_4 \mathbb{T} \rangle  \cong \mathbb{Z}/4
 \end{align}
  acting in the 2-dimensional torus $T^2$. We will use the Atiyah-Hirzebruch spectral sequence to calculate both the magnetic equivariant K-theory and the twisted one in the presence of spin orbit interaction. Let us start by showing the $G$-CW decomposition of the torus $T^2$.
  
  The $\mathbb{Z}/4$-CW decomposition of the 2-dimensional torus  $T^2$ can be seen in Fig. \ref{Decomposition of torus} and consists of the following cells:

\begin{align} \label{cell decomposition0}
    0\mathrm{-cells:} &\underbrace{ \left(\mathbb{Z}/4\Big/\mathbb{Z}/4\right)\times D^0}_{\Gamma}  \  \sqcup \  \underbrace{\left(\mathbb{Z}/4\Big/\mathbb{Z}/2 \right) \times D^0}_{A\sqcup A'} \   \sqcup \ \underbrace{\left(\mathbb{Z}/4\Big/\mathbb{Z}/4\right)\times D^0}_{X},\\
    1\mathrm{-cells:} & \underbrace{\left(\mathbb{Z}/4\right) \times D^1}_{\gamma} \ \sqcup \  \underbrace{\left(\mathbb{Z}/4\right) \times D^1}_{\sigma},   \label{cell decomposition1}\\
   2\mathrm{-cells:} &\underbrace{\left(\mathbb{Z}/4\right) \times D^2}_{\Omega}.  \label{cell decomposition2}
\end{align}

  \begin{figure}
    \centering
    \includegraphics[width=0.4\linewidth]{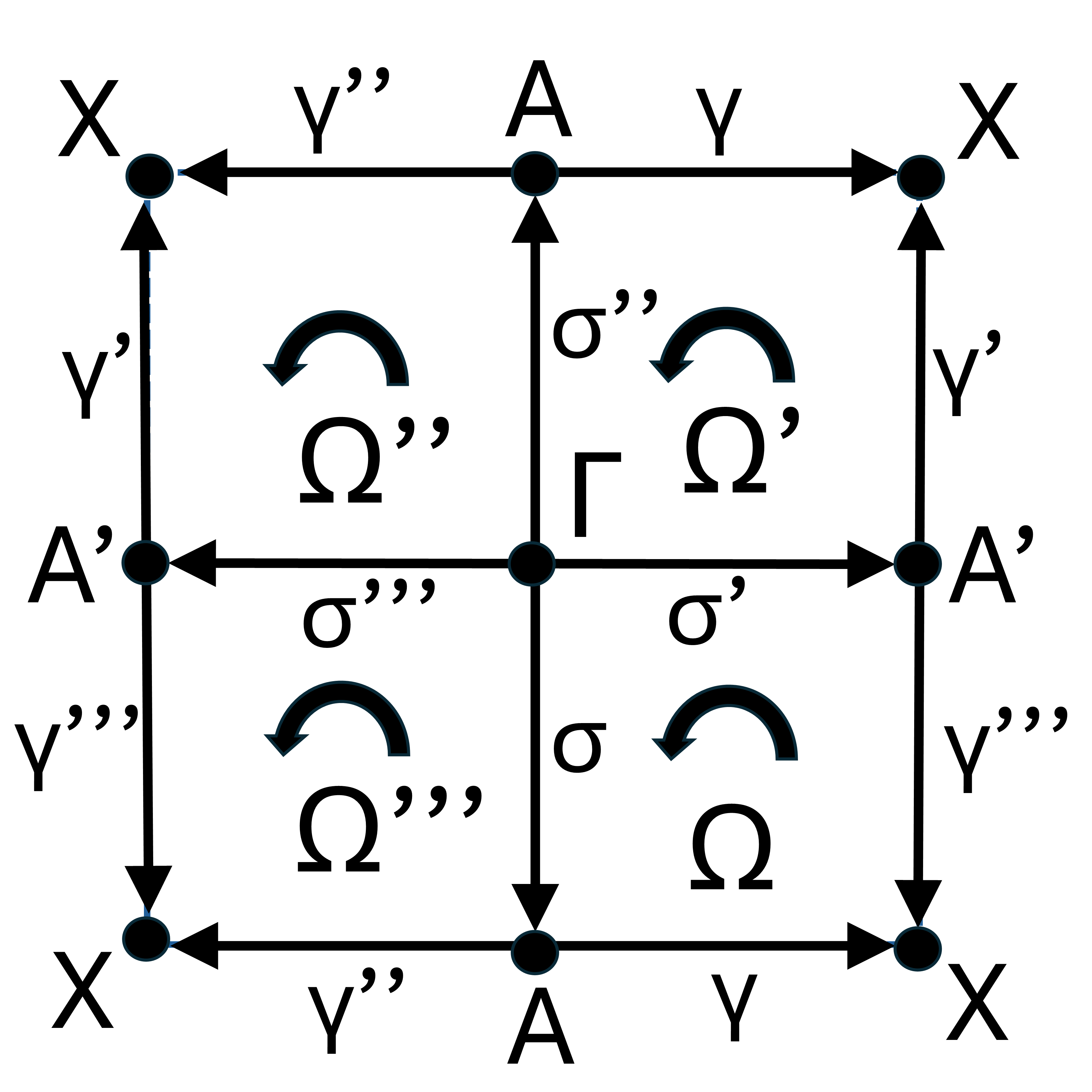}
    \caption{$\mathbb{Z}/4$-CW decomposition of  the torus $T^2=S^1\times S^1$. The 0-skeleton consists of the points $\Gamma, A, A',X$, the 1-skeleton adds the 1-cells generated by the orbits of $\gamma$ and $\sigma$, and the 2-skeleton adds the 2-cell generated by the orbit of $\Omega$. The group }
    \label{Decomposition of torus}
\end{figure}
  
  We are interested in calculating the magnetic equivariant K-theory $\K_{\mathbb{Z}/4}^{*}(T^2)$ in order
  to find out the topological invariants that materials with a point group symmetry generated by $C_4 \mathbb{T}$ may have.
  Since the magnetic group is cyclic of order bigger than 2, then we know from Prop. \ref{4-periodic}
  that the K-theory $\K_{\mathbb{Z}/4}^{*}(T^2)$ is 4-periodic. Therefore we will only calculate
$\K_{\mathbb{Z}/4}^{*}(T^2)$ for $*=0,-1,-2,-3.$

  Applying the procedure of the Atiyah-Hirzebruch spectral sequence of section \S \ref{section AHSS}, we see that the first page 
  becomes
  \begin{align}
 E_1^{0,*} \cong & \K_{\mathbb{Z}/4}^*(\{\Gamma\}) \oplus KU^*_{\mathbb{Z}/2}(\{A\}) \oplus  \K_{\mathbb{Z}/4}^*(\{X\}), \\
  E_1^{1,*} \cong & KU^*(\{\sigma\}) \oplus KU^*(\{\gamma\}), \\
   E_1^{2,*} \cong & KU^*(\{\Omega\}).
  \end{align}
  In Eqn. \eqref{K-theory Z/4} it is shown that $\K_{\mathbb{Z}/4}^* \cong K\mathbb{R}^* \oplus K\mathbb{H}^*$ where $\mathbf{R}(\mathbb{Z}/4) \cong \mathbb{Z}\oplus \mathbb{Z}$ is generated by the irreducible representations defined by complex conjugation $\mathbb{K}$ and the $2\times 2$ matrix $i \sigma_2 \mathbb{K}$; denote them by $R_1$ and $R_2$ respectively. 
  The K-theory group $KU^*_{\mathbb{Z}/2}$ is 2-periodic with  $KU^0_{\mathbb{Z}/2}\cong \mathbb{Z} \oplus \mathbb{Z}$ generated by the irreducible representations of $\mathbb{Z}/2$ which we denote $R_1'$ and $R_2'$, and $KU^*$ is $\mathbb{Z}$ in even degrees.  Hence
the first page of the spectral sequence becomes:

\begin{align}
\begin{tabular}{c c c c c}
  $E_1^{0,0}=\mathbb{Z}^2\oplus \mathbb{Z}^2\oplus  \mathbb{Z}^2$   & $\overset{d^{0,0}_1}{\longrightarrow}$ & $E^{1,0}_1= \mathbb{Z}\oplus \mathbb{Z}$ & $\overset{d^{1,0}_1}{\longrightarrow}$ & $E^{2,0}_1=\mathbb{Z}$, \\
  $E_1^{0,-1}=\mathbb{Z}/2\oplus 0\oplus  \mathbb{Z}/2$   & $\overset{d^{0,-1}_1}{\longrightarrow}$ & $E^{1,-1}_1= 0\oplus 0$ & $\overset{d^{1,-1}_1}{\longrightarrow}$ & $E^{2,-1}_1=0$, \\
  $E_1^{0,-2}=\mathbb{Z}/2\oplus \mathbb{Z}^2\oplus  \mathbb{Z}/2$   & $\overset{d^{0,-2}_1}{\longrightarrow}$ & $E^{1,-2}_1= \mathbb{Z}\oplus \mathbb{Z}$ & $\overset{d^{1,-2}_1}{\longrightarrow}$ & $E^{2,-2}_1=\mathbb{Z}$, \\
  $E_1^{0,-3}=0\oplus 0\oplus  0$   & $\overset{d^{0,-3}_1}{\longrightarrow}$ & $E^{1,-3}_1= 0\oplus 0$ & $\overset{d^{1,-3}_1}{\longrightarrow}$ & $E^{2,-3}_1=0$ .
\end{tabular}
\end{align}

Now we are going to focus on the differentials $d_1^{n,t}$. We use the orientations given in Fig. \ref{Decomposition of torus} and we write the differentials in matrix form:

\begin{align}
d_1^{0,0}=\begin{array}{|c|c|c|c|}
     \Gamma & A & X & \\ \begin{array}{cc}
         R_1 & R_2 
     \end{array} & \begin{array}{cc}
         R'_1 & R'_2 
     \end{array} & \begin{array}{cc}
         R_1 & R_2 
     \end{array} &\\ \hline
     \begin{array}{cc}
         0&0
     \end{array}& \begin{array}{cc}
         -1&-1
     \end{array} & \begin{array}{cc}
         1& 2
     \end{array} & \gamma\\ \hline \begin{array}{cc}
         -1&-2
     \end{array} & \begin{array}{cc}
         1&1
     \end{array} & \begin{array}{cc}
         0 &0
     \end{array} & \sigma
\end{array} \ \ \ \ 
d_1^{0,-2}=\begin{array}{|c|c|}
    A &  \\  \begin{array}{cc}
       R'_1  & R'_2 \end{array}
     & \\
     \hline \begin{array}{cc}
         -1&-1
     \end{array} &  \gamma \\ \hline \begin{array}{cc}
         1&1
     \end{array} & \sigma
\end{array}
\end{align}
 where $R_1=\langle \mathbb{K} \rangle$ maps to $\pm 1$, 
 $R_1=\langle i \sigma_2 \mathbb{K} \rangle$ maps to $\pm 2$
 and $ R_i'$ maps to $\pm 1$ depending whether the vertex is joined to the respective edge.
 
 The differentials $d_{1}^{1,0}$ and $d_{1}^{1,-2}$ need some extra consideration. The boundary of $\Omega$ consists of $\sigma, \gamma, \gamma'''$ and $\sigma'$ where $\sigma'$ and $\gamma'''$
 are the paths obtained from $\sigma$ and $\gamma$ after the action  of $(C_4 \mathbb{T})^3$ and 
$C_4 \mathbb{T}$ respectively. The isomorphism induced by the pullback of $C_4\mathbb{T}$ 
\begin{align}
(C_4\mathbb{T})^{*} : KU^*(\gamma''', \partial \gamma''') \stackrel{\cong}{\to} KU^*(\gamma, \partial \gamma) 
\end{align}
is an isomorphism which could be understood as the composition of the isomorphisms:
\begin{align}
 KU^*(\gamma''', \partial \gamma''') \stackrel{(C_4)^*}{\to} KU^*(\gamma'', \partial \gamma'') 
\stackrel{\mathbb{T}^*}{\to} KU^*(\gamma, \partial \gamma).
\end{align}
 The second isomorphism can be translated to the automorphism in $KU^*$ that complex conjugation defines. This is the identity in $KU^0$ while it is multiplication by $-1$ in $KU^{-2}$. In $KU^0$ the rank of the complex vector bundle is unaffected by complex conjugation, while the Hopf bundle $[H]-\mathbf{1}$ in $KU^{-2}= \widetilde{KU}^0(S^2)$
 is mapped to its conjugate bundle $[\overline{H}] - \mathbf{1}$.
 
 This means that $\mathbb{T}^*$ acts trivially on $KU^{-3}(\gamma, \partial \gamma) \cong KU^{-4}$
 while it acts by multiplication by $-1$ on $KU^{-1}(\gamma, \partial \gamma) \cong KU^{-2}$.
 
 Hence we have the remaining differentials of the first page are::
\begin{align} \label{part2 differentials no soc}
d_1^{1,0}=
\begin{array}{|cc|c|}
    \gamma & \sigma  &   \\
     \hline
        0 & 0 &  \Omega 
\end{array}\ \ \  \ 
d_1^{1,-2}=
\begin{array}{|cc|c|}
    \gamma & \sigma  &   \\
     \hline
        2 & 2 &  \Omega 
\end{array}
\end{align}

Therefore the the second page of the spectral sequence becomes:
\begin{align}
\begin{tabular}{c c c c c}
  $E_2^{0,0}=\mathbb{Z}^4$   &  & $E^{1,0}_2=0$ &  & $E^{2,0}_2=\mathbb{Z}$ \\
  $E_2^{0,-1}=(\mathbb{Z}/2)^2$   &  & $E^{1,-1}_2=0$ &  & $E^{2,-1}_2=0$ \\
  $E_2^{0,-2}=(\mathbb{Z}/2)^2\oplus \mathbb{Z}$   &  & $E^{1,-2}_2= 0$ &  & $E^{2,-2}_2=\mathbb{Z}/2$ \\
  $E_2^{0,-3}=0$   &  & $E^{1,-3}_2= 0$ &  & $E^{2,-3}_2=0$ 
\end{tabular}
\end{align}
implying that all second differentials $d_2^{n,t}$ are zero except for
the case $d_2^{0,-1} : E_2^{0,-1} \to E_2^{2,-2}$. Let us show that
this differential is surjective.

Consider the diagram
\begin{align} \label{diagram first differential}
\xymatrix{
{\mathbf K}_{\mathbb{Z}/4}^{-1}(X_1,X_0) \ar@{>->}[d]^{j_1^*} \ar[rd]^{d_1^{1,-2}}&&& \\
{\mathbf K}_{\mathbb{Z}/4}^{-1}(X_1) \ar[r]^{\partial^{1}}  \ar@{->>}[d]^{i_0^*}&{\mathbf K}_{\mathbb{Z}/4}^{0}(X_2,X_1) \ar[r]^{j_2^*} &{\mathbf K}_{\mathbb{Z}/4}^{0}(X_2) \ar[r] & {\mathbf K}_{\mathbb{Z}/4}^{0}(X_1) \\
{\mathbf K}_{\mathbb{Z}/4}^{-1}(X_0) 
}    
\end{align}
with a horizontal exact sequence,  a vertical short exact sequence, and the diagonal homomorphism
$d_1^{1,-2}$ being the first differential of the spectral sequence. Here we have followed the notation of diagram \eqref{diagram K-theory AHSS}.

The cokernel of the first differential $d_1^{1,-2}$ is the group $\mathbb{Z}/2 = E_2^{2,-2}$ while
the image of the second differential $d_2^{0,-1}$ 
is the quotient group $\frac{\mathrm{Im}(\partial^1)}{\mathrm{Im}(d_1^{1,-2})}$.
We claim that this quotient group $\frac{\mathrm{Im}(\partial^1)}{\mathrm{Im}(d_1^{1,-2})}$ is $\mathbb{Z}/2$, thus implying that $d_2^{0,-1}$ is surjective.
This claim will follow from showing that the image of the homomorphism $j_1^*$
will not generate the torsion free part of ${\mathbf K}_{\mathbb{Z}/4}^{-1}(X_1)$;
let us show this.

Consider the part of the 1-skeleton generated by the orbits of $\sigma$ and $ \gamma$; denote them
 $[\sigma] := \mathbb{Z}/4 \cdot \sigma$ and 
  $[\gamma] := \mathbb{Z}/4 \cdot \gamma$ respectively. These are a $\mathbb{Z}/4$-spaces whose quotient spaces $[\sigma]/(\mathbb{Z}/2)$ 
  $[\gamma]/(\mathbb{Z}/2)$ are both $\mathbb{Z}/2$-homeomorphic to the 1-dimensional ball $B^{1,0}$. Hence the projection
map $[\sigma] \to [\sigma]/(\mathbb{Z}/2) \cong B^{1,0}$, together with the excision isomorphisms, induce the following homomorphisms
among the short exact sequences of K-theories, which are given by applying Cor. \ref{sucesiondetres} to the triples $(B^{1,0},S^{1,0}\cup \{0\},S^{1,0})$ and $(X_1,[\gamma]\cup \{\Gamma\},[\gamma])$:
\begin{align}
\xymatrix{
\mathbb{Z} \cong K\mathbb{R}^{-1}(B^{1,0},S^{1,0}\cup \{0\}) \ar[r]^{\times 2}  \ar[d]^\cong & \mathbb{Z} \cong K\mathbb{R}^{-1}(B^{1,0},S^{1,0}) \ar[r] \ar[d]&
\mathbb{Z}/2 \cong K\mathbb{R}^{-1}(S^{1,0} \cup \{0\}, S^{1,0}) \ar[d]^{\cong} \\
 \mathbb{Z} \cong {\mathbf K}^{-1}_{\mathbb{Z}/4}(X_1,[\gamma]\cup\{\Gamma \}) \ar[r] & {\mathbf K}^{-1}_{\mathbb{Z}/4}(X_1,[\gamma])  \ar[r] &
\mathbb{Z}/2 \cong {\mathbf K}^{-1}_{\mathbb{Z}/4}([\gamma] \cup \{\Gamma\}, [\gamma])   
  }    
\end{align}
thus implying that ${\mathbf K}^{-1}_{\mathbb{Z}/4}(X_1,[\gamma]) \cong \mathbb{Z}$ and that the bottom left horizontal map is given by multiplication by $2$. Therefore the commutativity of the square
\begin{align}
\xymatrix{
\mathbb{Z} \cong {\mathbf K}^{-1}_{\mathbb{Z}/4}(X_1,[\gamma]\cup\{\Gamma \}) \ar[r]^{\times 2} \ar@{>->}[d] & {\mathbf K}^{-1}_{\mathbb{Z}/4}(X_1,[\gamma])  \cong \mathbb{Z} \ar[d] \\
\mathbb{Z}^2 \cong {\mathbf K}^{-1}_{\mathbb{Z}/4}(X_1,X_0) \ar[r] &
{\mathbf K}^{-1}_{\mathbb{Z}/4}(X_1) 
  }    
\end{align}
implies that the image of ${\mathbf K}^{-1}_{\mathbb{Z}/4}(X_1,X_0)$
does not include a generator of the torsion free part of ${\mathbf K}^{-1}_{\mathbb{Z}/4}(X_1)$. From diagram \eqref{diagram first differential}
we have that the image of the first differential does not equal the image of the
connection homomorphism $\partial^1$. We conclude that 
${\mathbf K}^{0}_{\mathbb{Z}/4}(X_2)$ is torsion free and the 
third page of the spectral sequence becomes:
\begin{align}
\begin{tabular}{c c c c c}
  $E_3^{0,0}=\mathbb{Z}^4$   &  & $E^{1,0}_3=0$ &  & $E^{2,0}_3=\mathbb{Z}$ \\
  $E_3^{0,-1}=\mathbb{Z}/2$   &  & $E^{1,-1}_3=0$ &  & $E^{2,-1}_3=0$ \\
  $E_3^{0,-2}=(\mathbb{Z}/2)^2\oplus \mathbb{Z}$   &  & $E^{1,-2}_3= 0$ &  & $E^{2,-2}_3=0$ \\
  $E_3^{0,-3}=0$   &  & $E^{1,-3}_3= 0$ &  & $E^{2,-3}_3=0$ 
\end{tabular}
\end{align}
collapsing on the $E_3$-page, e.i. $E_3=E_{\infty}$. 
We  now look at the filtration of the magnetic K-theory groups and solve the extension problems
\begin{align}
\xymatrix{
0\ar[r] & E_{\infty}^{2,-n-2} \ar[r] & F^1\K^{-n}_{\mathbb{Z}/4}(T^2) \ar[r] & E^{1,-n-1} \ar[r] & 0 
}
\end{align}
\begin{align}
\xymatrix{
0\ar[r] & F^1\K^{-n}_{\mathbb{Z}/4}(T^2) \ar[r] & \K^{-n}_{\mathbb{Z}/4}(T^2) \ar[r] & E^{0,-n}_{\infty} \ar[r] & 0
}
\end{align}
 to recover the $\K_{\mg{G}}^{*}$-groups:
\begin{itemize}
    \item $n=0$: $F^1\K^{0}_{\mathbb{Z}/4}(T^2)=0$ and $\K^{0}_{\mathbb{Z}/4}(T^2)=\mathbb{Z}^4$.
    \item  $n=-1$: $F^1\K^{-1}_{\mathbb{Z}/4}(T^2)=0$ and $\K^{-1}_{\mathbb{Z}/4}(T^2)=\mathbb{Z}/2$.
    \item  $n=-3$: $F^1\K^{-3}_{\mathbb{Z}/4}(T^2)=0$ and $\K^{-3}_{\mathbb{Z}/4}(T^2)=0$.
\end{itemize}

The only extension that needs attention is the one for $n=-2$. Here $F^1\K^{-2}_{\mathbb{Z}/4}(T^2)=\mathbb{Z}$ and the short exact sequence becomes:
\begin{align} \label{non-split sequence}
\xymatrix{
0\ar[r] & \mathbb{Z} \ar[r] & \K^{-2}_{\mathbb{Z}/4}(T^2) \ar[r] & \mathbb{Z} \oplus (\mathbb{Z}/2)^2 \ar[r] & 0.
}
\end{align}
One copy of $\mathbb{Z}/2$ splits from the short exact sequence because the composition of the  $\mathbb{Z}/4$ equivariant
maps $\{ \Gamma \} \hookrightarrow T^2 \to *$ implies that the map $\mathbb{Z}/2 \cong \K^{-2}_{\mathbb{Z}/4}(*) \hookrightarrow \K^{-2}_{\mathbb{Z}/4}(T^2)$ is injective.
We claim that the other copy of $\mathbb{Z}/2$ does not
split.

Consider the quotient $T^2\big/ \mathbb{Z}/2$ where
$\mathbb{Z}/2 = \mathrm{ker}(\mathbb{Z}/4 \to \mathbb{Z}/2)$ and note that it is homeomorphic to $S^{2,1}$, the 2-sphere with the involution given by a rotation of 180 degrees. The quotient map $q: T^2\longrightarrow S^{2,1}$ induces the following commutative diagrams:
\begin{align}
\scalebox{0.9}{
\xymatrix{
0\ar[r]& K\mathbb{R}^{-2}(S^{2,1},\{q(X)\}) \ar[r]  & K\mathbb{R}^{-2}(S^{2,1})=\mathbb{Z}  \ar[r] \ar@{=}[d] & K\mathbb{R}^{-2}(q(X)\})=\mathbb{Z}/2 \ar@{_(->}[d] \ar[r]& 0  \\ 
0\ar[r]& K\mathbb{R}^{-2}(S^{2,1},\{q(X),q(\Gamma)\}) \ar[r] \ar[u] \ar@{_(->}[d] & K\mathbb{R}^{-2}(S^{2,1})=\mathbb{Z}  \ar[r] \ar@{_(->}[d] & K\mathbb{R}^{-2}(\{q(X),q(\Gamma)\})=(\mathbb{Z}/2)^2 \ar[d]^{\cong} &  \\ 
0\ar[r]& \K_{\mathbb{Z}/4}^{-2}(T^2, \{X,\Gamma\}) \ar[r] &\K_{\mathbb{Z}/4}^{-2}(T^2) \ar[r] & \K_{\mathbb{Z}/4}^{-2}(\{X,\Gamma\})=(\mathbb{Z}/2)^2  &   }  
} \label{diagram T2 / Z2}
\end{align}
The first row of diagram \eqref{diagram T2 / Z2}
is exact because the real K-theory of the 
 pair $(S^{2,1},q(X))$ is equivalent to the one of the pair $(B^{2,0},\partial B^{2,0})$ and $K\mathbb{R}^{-1}(B^{2,0},\partial B^{2,0})=0$.
 Since the vertical map in the bottom of the right-hand side of diagram \eqref{diagram T2 / Z2} is an isomorphism, we see that the sequence in the bottom row of of diagram \eqref{diagram T2 / Z2} does not split. We therefore conclude that the sequence in 
 diagram \eqref{non-split sequence} does not split; thus implying that
\begin{align}
\K^{-2}_{\mathbb{Z}/4}(T^2)=(\mathbb{Z}\oplus \mathbb{Z}/2)\oplus \mathbb{Z}.
\end{align}

We conclude that the magnetic $\mathbb{Z}/4
$-equivariant K-theory groups of the 2-dimensional torus.

\begin{proposition} \label{C4T NOSOC}
The magnetic $\mathbb{Z}/4$-equivariant K-theory groups of the 2-dimensional torus $T^2$ are:
\begin{align}
\begin{tabular}{|c|cccc|}
 \hline
  $n$  & $0$ & $-1$ & $-2$ & $-3$   \\
      \hline
 $  \K^n_{\mathbb{Z}/4}(T^2) $ &$ \mathbb{Z}^4$ & $\mathbb{Z}/2$ & $\left( \mathbb{Z} \oplus \mathbb{Z}/2 \right) \oplus \mathbb{Z}$ &$ 0 $
 \\ \hline
\end{tabular}
\end{align}
\end{proposition}

Here we want to emphasize that the right hand side copy of $\mathbb{Z}$ in   $\K^{-2}_{\mathbb{Z}/4}(T^2)$ is the only
invariant that depends on the 2-cell in the torus. This invariant is therefore called {\it bulk} invariant in condensed matter physics.

\subsection{2-dimensional torus with  $C_4\mathbb{T}$ symmetry and spin-orbit interaction}

Now we are interested in calculating the magnetic equivariant K-theory groups of the
group generated by $C_4\mathbb{T}$ acting on the 2-dimensional torus in the framework of spin orbit 
interaction. Following Def. \ref{definition twisted soc} we are interested in the
twisted magnetic K-theory groups ${}^{\widetilde{\mathbb{Z}/4}^{soc}}\K_{\mathbb{Z}/4}^*(T^2)$.

In the presence of spin orbit interaction, the rotation $C_4$ lifts to an operator $\widetilde{C_4}$ with the property that $(\widetilde{C_4})^4=-1$, and $\mathbb{T}^2=-1$. Therefore the composition
$\widetilde{C_4} \mathbb{T}$ is of order $8$ with $(\widetilde{C_4} \mathbb{T})^4=-1$.
The SOC group $\widetilde{\mathbb{Z}/4}^{soc}$ of $\mathbb{Z}/4= \langle C_4 \mathbb{T} \rangle$ is therefore the cyclic group $\mathbb{Z}/8 \cong \langle \widetilde{C_4} \mathbb{T} \rangle$ and we will calculate this twisted magnetic equivariant K-theory group
\begin{align}
{}^{\widetilde{\mathbb{Z}/4}^{soc}}\K_{\mathbb{Z}/4}^*(T^2) = {}^{\mathbb{Z}/8}\K_{\mathbb{Z}/4}^*(T^2).
\end{align}

Using the cell decomposition of Eqns. \eqref{cell decomposition0}, \eqref{cell decomposition1} and \eqref{cell decomposition2} we get for first page of the Atiyah-Hirzebruch spectral sequence the groups:
  \begin{align}
 E_1^{0,*} \cong & {}^{\mathbb{Z}/8}\K_{\mathbb{Z}/4}^*(\{\Gamma\}) \oplus {}^{\mathbb{Z}/4}KU^*_{\mathbb{Z}/2}(\{A\}) \oplus  {}^{\mathbb{Z}/8}\K_{\mathbb{Z}/4}^*(\{X\}), \\
  E_1^{1,*} \cong & {}^{\mathbb{Z}/2}KU^*(\{\sigma\}) \oplus {}^{\mathbb{Z}/2}KU^*(\{\gamma\}), \\
   E_1^{2,*} \cong & {}^{\mathbb{Z}/2}KU^*(\{\Omega\}).
  \end{align} 

For the 0-cells $\Gamma$ and $X$ we know that the only non-trivial $\mathbb{Z}/8$-twisted magnetic  irreducible representation of $\mathbb{Z}/4$ is generated by the magnetic matrix $\big(\begin{smallmatrix}
  0 & i\\
  1 & 0
\end{smallmatrix}\big) \mathbb{K}$, and since  it squares to $\big(\begin{smallmatrix}
  i & 0\\
  0 & -i
\end{smallmatrix}\big) $ we see that this representation is of complex type.
Abusing of notation, we obtain:
\begin{align}
{}^{\mathbb{Z}/8}\K^t_{\mathbb{Z}/4}(\{\Gamma\}) = {}^{\mathbb{Z}/8}\K^t_{\mathbb{Z}/4}(\{X\}) =\begin{cases}
  {}^{\mathbb{Z}/8}\mathbf{R}(\mathbb{Z}/4)=  \mathbb{Z} \langle R=\big(\begin{smallmatrix}
  0 & i\\
  1 & 0
\end{smallmatrix}\big) \mathbb{K}\rangle & \text{ if }t\equiv 0\, \operatorname{mod}2 \\ 0 & \text{ if } t\equiv 1 \, \operatorname{mod}2.
\end{cases}
\end{align}

The 0-cells $A\sqcup A'$ have as stabilizer the non magnetic group $\mathbb{Z}/4$ which is an extension of $\mathbb{Z}/2$ by the core group $\mathbb{Z}/2$. The irreducible complex representations of $\mathbb{Z}/4$ where $\mathbb{Z}/2$ acts by multiplication by $-1$ are the ones generated by $i$ and $-i$. If we denote the representations by their generators we get: 
\begin{align}
    {}^{\mathbb{Z}/4}KU^t_{\mathbb{Z}/2}(\{A\})=& \begin{cases}
       {}^{\mathbb{Z}/4}R(\mathbb{Z}/2))=\mathbb{Z}\langle i,-i\rangle & \text{ if } t\equiv 0\,
        \operatorname{mod}2\\ 0 & \text{ if } t\equiv\, 1 \operatorname{mod}2.
    \end{cases}
\end{align}
For the 1-cells and 2-cells where the stabilizer is trivial we simply obtain the complex K-theory group where $\mathbb{Z}/2$ acts on the fibers via the sign representation:
\begin{align}
 {}^{\mathbb{Z}/2}KU^*(\{\sigma\}) =  {}^{\mathbb{Z}/2}KU^*(\{\gamma\})=  {}^{\mathbb{Z}/2}KU^*(\{\Omega\}) =& \begin{cases}
        \mathbb{Z}\langle \operatorname{sgn}\rangle & \text{ if } t\equiv 0\,
        \operatorname{mod}2\\ 0 & \text{ if } t\equiv\, 1 \operatorname{mod}2.
    \end{cases}
\end{align}

 Since the twisted magnetic equivariant K-theory group $ {}^{\mathbb{Z}/8}\K_{\mathbb{Z}/4}^*(T^2)$ is a subgroup of the equivariant magnetic K-theory $\K_{\mathbb{Z}/8}^*(T^2)$, then we know from Prop. \ref{4-periodic}
  that the K-theory groups are  4-periodic. We will only calculate  $ {}^{\mathbb{Z}/8}\K_{\mathbb{Z}/4}^*(T^2)$ for $*=0,-1,-2,-3$.

The first page of the spectral sequence then becomes:
\begin{align}
\begin{array}{l c l c l}
  E_1^{0,0}=\mathbb{Z}\oplus \mathbb{Z}^2 \oplus \mathbb{Z} & \overset{d^{0,0}_1}{\longrightarrow} & E^{1,0}_1=\mathbb{Z}\oplus \mathbb{Z} & \overset{d^{1,0}_1}{\longrightarrow} & E^{2,0}_1=\mathbb{Z} \\
  E_1^{0,-1}=0  & \overset{d^{0,-1}_1}{\longrightarrow} & E^{1,-1}_1=0 & \overset{d^{1,-1}_1}{\longrightarrow} & E^{2,-1}_1=0 \\
  E_1^{0,-2}=\mathbb{Z}\oplus \mathbb{Z}^2 \oplus \mathbb{Z} & \overset{d^{0,-2}_1}{\longrightarrow} & E^{1,-2}_1=\mathbb{Z}\oplus \mathbb{Z} & \overset{d^{1,-2}_1}{\longrightarrow} & E^{2,-2}_1=\mathbb{Z} \\
  E_1^{0,-3}=0  & \overset{d^{0,-3}_1}{\longrightarrow} & E^{1,-3}_1=0 & \overset{d^{1,-3}_1}{\longrightarrow} & E^{2,-3}_1=0  
\end{array}
\end{align}

The differential $d_1^{0,0}$ keeps the rank of the dual of the restrictions. Hence we have:
\begin{align}
d_1^{0,0}=\begin{array}{|c|c|c|c|}
     \Gamma & A & X & \\ R & \begin{array}{cc}
         i & -i 
     \end{array} & R &\\ \hline
     0& \begin{array}{cc}
         -1&-1
     \end{array} & 2 & \gamma\\ -2 & \begin{array}{cc}
         1&1
     \end{array} & 0 & \sigma
\end{array}
\end{align}
The next differential satisfies $d_1^{0,0}=0$ because the boundary of $\Omega$ is zero in terms of $\gamma$ and $\sigma$.

The first differential on the  row $-2$ becomes:
\begin{align}
d_1^{0,-2}=\begin{array}{|c|c|c|c|}
     \Gamma & A\sqcup A' & X & \\ R' & \begin{array}{cc}
         i & -i 
     \end{array} & R' &\\ \hline
     0& \begin{array}{cc}
         -1&-1
     \end{array} & 0 & \gamma\\ 0 & \begin{array}{cc}
         1&1
     \end{array} & 0 & \sigma
\end{array} 
\end{align}
where the zeroes from the first and the third columns come from the fact that
the forgetful homomorphism ${}^{\mathbb{Z}/8}\K_{\mathbb{Z}/4}^{-2}(*) \to KU^{-2}(*)$ is trivial. The reason for this is the following. A $\mathbb{Z}/8$-twisted 
magnetic 
$\mathbb{Z}/4$-equivariant vector bundle over $S^2$ (here the action of $\mathbb{Z}/4$ is trivial) can be split into the eigenbundles of $i$ and $-i$ respectively. The action of the generator $\big(\begin{smallmatrix}
  0 & i\\
  1 & 0
\end{smallmatrix}\big) \mathbb{K}$ maps one to the other, but flips the complex structure. Hence the first Chern class of the underlying complex bundle is
the sum of the Chern class of the two, but they have opposite Chern class. Therefore the
Chern class of the underlying complex vector bundle is zero and therefore the forgetful map is trivial.

The remaining differential has the same structure as the one presented with out spin orbit interaction in Eqn. \eqref{part2 differentials no soc}:
\begin{align}
d^{1,-2}_1=\begin{array}{|c|c|c|}
\sigma& \gamma&\\ \hline 2&2&\Omega
\end{array}
\end{align}

The second page of the spectral sequence is therefore:
\begin{align}
\begin{array}{l c l c l}
  E_2^{0,0}=\mathbb{Z}^2   &  & E^{1,0}_2=\mathbb{Z}/2 &  & E^{2,0}_2=\mathbb{Z} \\ E_2^{0,-1}=0   &  & E^{1,-1}_2=0 &  & E^{2,-1}_2=0 \\ E_2^{0,-2}=\mathbb{Z}^3   &  & E^{1,-2}_2= 0 &  & E^{2,-2}_2=\mathbb{Z}/2 \\
  E_2^{0,-3}=0   &  & E^{1,-3}_2=0 &  & E^{2,-3}_2=0
\end{array}
\end{align}
All the differentials $d_2^{n,t}$ are zero and the spectral sequence collapses at the second page. In this case there are no extension problems for the K-theory groups and we obtain the desired calculation

\begin{proposition} \label{C4T SOC}
The $\mathbb{Z}/8$-twisted magnetic $\mathbb{Z}/4$-equivariant K-theory groups of the 2-dimensional torus $T^2$ are:
\begin{align}
\begin{tabular}{|c|cccc|}
 \hline
  $n$  & $0$ & $-1$ & $-2$ & $-3$   \\
      \hline
 $  {}^{\mathbb{Z}/8}\K^n_{\mathbb{Z}/4}(T^2) $ &$ \mathbb{Z}^2  \oplus \mathbb{Z}/2$ & $0$ & $ \mathbb{Z}^3\oplus \mathbb{Z}$ &$ \mathbb{Z}/2 $
 \\ \hline
\end{tabular}
\end{align}
\end{proposition}

Here we emphasize that both the $\mathbb{Z}/2$ invariant in $   {}^{\mathbb{Z}/8}\K^0_{\mathbb{Z}/4}(T^2)$
as well as the right hand side copy of $\mathbb{Z}$ in   ${}^{\mathbb{Z}/8}\K^{-2}_{\mathbb{Z}/4}(T^2)$ are
invariants that depend on the 2-cell in the torus. These invariants are {\it bulk} invariants of the system.

\begin{remark}
The bulk $\mathbb{Z}/2$ invariant in ${}^{\mathbb{Z}/8}\K^{0}_{\mathbb{Z}/4}(T^2)$ was shown by Gonz\'alez-Hern\'andez and the first
two authors in \cite{gonzalezhernandez2024spinchernnumberaltermagnets} to be the indicator for topological insulators in the presence of the altermagnetic symmetry $C_4 \mathbb{T}$ \cite{Altermagnetism1,Altermagnetism2,Altermagnetism3}. Gapped Hamiltonians preserving the symmetry $C_4 \mathbb{T}$ which represent the non-trivial value 
of $\mathbb{Z}/2$ are insulating states which are topological non-trivial. Hence its classification as topological insulator. 
\end{remark}

\begin{remark}
    Both calculations presented in Props. \ref{C4T NOSOC} and \ref{C4T SOC}
agree with the calculations carried out by Shiozaki
and Ono in \cite{Shiozaki}.
\end{remark}

%%%%%%%%%%%%%%%%%%%%%%%%%%%%%%%%

\section*{Conclusions and further remarks}

We have presented several of the most important properties
of the K-theory of magnetic equivariant vector bundles. We have done so emphasizing simplicity in arguments, proofs and notation. We have carried out explicit calculations
of these K-theory groups for important symmetries in 
condensed matter physics, and we have obtained important
topological invariants associated to these symmetries.

We have complemented and enhanced the works of Freed and Moore \cite{twistedmat} and Gomi \cite{Gomi2017FreedMooreK} in the magnetic equivariant K-theory, and in doing so, we have put forward a better suited name for these K-theory groups.  

We believe that with the properties of the magnetic equivariant K-theory presented in this work, the calculation of these K-theory groups for explicit space magnetic groups becomes accessible for both mathematicians and 
condensed matter physicists. Important here is to notice that only symmorphic magnetic crystal symmetries can be addressed with the results of this work. Non symmorphic magnetic crystal symmetries require the study of twisted magnetic equivariant K-theory where the twists incorporate both information of the point group and of the torus. 

One of the motivations to address and showcase the properties of the magnetic equivariant K-theory was the 
necessity of putting on safe grounds the diverse calculations of magnetic equivariant K-theory groups that were done by the three author in previous works 
\cite{axion, rationalmagnetic, gonzalezhernandez2024spinchernnumberaltermagnets}. 
These works dealt with the topological invariants associated to a prescribed symmetry (in our works these were $C_2 \mathbb{T}$ and $C_4 \mathbb{T}$) and the calculations of the magnetic equivariant K-theory were done 
assuming they possessed the properties we showcase in the previous chapters. 
This work fills this foundational gap.

We hope that our presentation of the properties of the magnetic equivariant K-theory will allow more people to determine topological invariants of magnetic materials. We are sure that further interesting topological features of magnetic crystals are going to be discovered.

%%%%%%%%%%%%%%%%%%%%%%%%%%%%%%%%%

%%%%%%%%%%%%%%%%%%%%%%%%%%%%%%%%%

\section*{Acknowledgments}
BU acknowledges the continuous financial support of the Max Planck Institute of Mathematics in Bonn, Germany, the International Center for Theoretical Physics in Trieste, Italy, and the Alexander Von Humboldt Foundation in Bonn, Germany.

The authors would like to thank Prof. Ken Shiozaki for letting us know
of an incomplete argument in the calculation of the K-theory groups of the torus with $C_4\mathbb{T}$ symmetry in an earlier draft of this paper. 
%%%%%%%%%%%%%%%%%%%%%%%%%%%%%%%%%%

%%%%%%%%%%%%%%%%%%%%%%%%%%%%%%%%%%

\section*{Author's contributions}
This work is an enhancement of HS PhD's thesis advised by both BU and MX.
BU and MX envisioned and organized the subject of the work, HS, BU and MX 
discussed the results, HS wrote the first draft, HS, BU and MX wrote the final draft.

%%%%%%%%%%%%%%%%%%%%%%%%%%%%%%%%%%

%%%%%%%%%%%%%%%%%%%%%%%%%%%%%%%%%%

\section*{Funding information}
 All authors acknowledge the support of  CONAHCyT of Mexico through grant number CB-2017-2018-A1-S-30345-F-3125. HS acknowledges the support SECIHTI of Mexico through the PhD scholarship number CVU: 926934.

\bibliographystyle{alpha}
%\addcontentsline{toc}{chapter}{Bibliography}
\bibliography{ref}
%\bibliographystyle{abbrv}

%\bibliography{Bibliography_Magnetic}

   \end{document}